\documentclass[11pt]{article}

\usepackage[table,xcdraw,dvipsnames]{xcolor}
\usepackage[pdftex, colorlinks, linkcolor=blue, citecolor=blue, urlcolor=blue]{hyperref}
\usepackage[utf8]{inputenc}
\usepackage{amsmath,amssymb,amsthm,tikz,paralist,verbatim} 
\usepackage[nameinlink,noabbrev,capitalise]{cleveref}
\usepackage{caption,subcaption}
\usepackage{geometry}
\usepackage{blkarray}
\usepackage{todonotes}
\usepackage{xargs}
\usepackage{csquotes}

\usepackage{empheq}

\definecolor{DarkGreen}{rgb}{0.0, 0.5, 0.0}
\definecolor{MidnightBlue}{rgb}{0,0.44,0.57}
\definecolor{TableHighlight}{rgb}{0.9,0.9,0.9}
\definecolor{TableHeader}{rgb}{0.8,0.8,0.8}
\hypersetup{linkcolor=MidnightBlue, citecolor=MidnightBlue, urlcolor=MidnightBlue}

\allowdisplaybreaks

\newcommand\N{{\mathbb N}}

\newcommand\RR{{\mathbb R}}

\newcommand\CC{{\mathbb C}}
\newcommand\T{{\mathbb T}}
\newcommand\TP{{\mathbb{TP}}}

\newcommand{\set}[1]{\ensuremath{ \left\lbrace  #1 \right\rbrace }}

\newcommand{\ma}{\begin{pmatrix} }      \newcommand{\trix}{\end{pmatrix}}
\newcommand{\sma}{\left(\begin{smallmatrix} }       \newcommand{\strix}{\end{smallmatrix}\right)}
\newcommand{\inner}[1]{\left\langle \, #1  \,\right\rangle}

\newcommand\SetOf[2]{\left\{#1 \mid #2\right\}}

\DeclareMathOperator{\conv}{conv}

\DeclareMathOperator{\cone}{cone}
\DeclareMathOperator{\interior}{int}

\DeclareMathOperator{\tconv}{tconv}
\DeclareMathOperator{\rk}{rk}
\DeclareMathOperator{\trop}{trop}
\DeclareMathOperator{\sgn}{sgn}
\DeclareMathOperator{\codim}{codim}

\newcommand\R{{\mathcal R}} 
\newcommand\Rplus{{\mathcal R_+}}
\newcommand\C{{\mathcal C}} 
\newcommand\Cplus{{\mathcal C_+}} 
\DeclareMathOperator{\lt}{lt} 
\DeclareMathOperator{\lc}{lc} 
\DeclareMathOperator{\val}{val}
\newcommand\tropcplus{\trop^{+{\mathcal C}}} 
\newcommand\troprplus{\trop^{+{\mathcal R}}} 
\DeclareMathOperator{\Trop}{Trop}
\newcommand\tropcombplus{\Trop^+} 
\DeclareMathOperator{\initial}{in}
\newcommand{\tropplus}{\trop^+} 
\newcommand{\sign}{s} 
\newcommand{\tropsigned}{\trop^\sign} 
\newcommandx{\tropdet}[2][2 = {d,n}, 1=r]{T_{#2}^{#1}} 
\newcommandx{\tropdetplus}[2][2 = {d,n}, 1=r]{(\tropdet[#1][#2])^+}
\newcommand{\Birk}{B}
\newcommand{\Sym}{S}
\newcommand{\BirG}{\mathcal{G}}

\newcommandx{\prevar}[2][2 = {d,n}, 1=r]{P_{#2}^{#1}}
\newcommandx{\prevarplus}[2][2 = {d,n}, 1=r]{(\prevar[#1][#2])^+}

\newcommand{\Ph}{P} 
\newcommand{\Phspace}[1][d,n]{\mathcal{BPT}_{ \kern-2.5 pt #1}} 
\newcommand{\bicoltree}{P^{(R,G)}}
\newcommand{\Gr}[1][2,d+n]{{\rm trop}\left({\rm Gr}(#1)\right)} 
\newcommand{\Grsub}[1][(R,G)]{\mathcal{UPT}^{#1}} 
\newcommand{\Dr}[1]{Dr(#1)} 
\newcommand{\Bembedded}{B_{r+1}^{IJ}} 

\newcommand{\G}[1][C]{\Gamma(#1)}
\newcommand{\Gc}[1][C]{\Gamma(#1)^c}
\newcommand{\Gpos}{\mathcal{G}_{pos}}

\newtheorem{theorem}{Theorem}[section]
\newtheorem{lemma}[theorem]{Lemma}
\newtheorem{prop}[theorem]{Proposition}
\newtheorem{corollary}[theorem]{Corollary}

\theoremstyle{definition}
\newtheorem{definition}[theorem]{Definition}
\newtheorem{example}[theorem]{Example}
\newtheorem{remark}[theorem]{Remark}
\newtheorem{const}[theorem]{Construction}
\newtheorem{question}[theorem]{Question}

\crefname{ineq}{inequality}{inequalities}
\crefname{prop}{Proposition}{Propositions}
\creflabelformat{ineq}{#2\textup{(#1)}#3}

\title{\sc{Tropical Positivity and Determinantal Varieties}}
\author{Marie-Charlotte Brandenburg \and Georg Loho \and Rainer Sinn}

\date{}

\begin{document}
\maketitle

\begin{abstract}
  We initiate the study of positive-tropical generators as positive analogues of the concept of tropical bases.
  Applying this to the tropicalization of determinantal varieties, we develop criteria for characterizing their positive part.
  We focus on the study of low-rank matrices, in particular matrices of rank $2$ and $3$. 
  Moreover, in the case square-matrices of corank $1$, we fully classify the signed tropicalization of the determinantal variety, even beyond the positive part.
\end{abstract}

\section{Introduction}

Tropicalization is a modern and powerful tool for understanding algebraic varieties via a polyhedral `shadow', to which combinatorial tools can be applied (for instance to solve enumerative problems). We are particularly interested in identifying the tropicalization of semi-algebraic subsets of algebraic varieties as a subset of the tropicalization of the whole (complex) variety. Specifically, we care about the \emph{positive part} of an algebraic variety, which arises in various applications from combinatorial optimization \cite{slackzbMATH07105541} to physics \cite{SpeyerWilliams:2021,ArkaniHamedLamSpradlin:2021} and statistics \cite{statszbMATH06559631}.
The tropicalizations of the positive parts of many classical varieties have been studied before. In this work, we focus on determinantal varieties inspired by applications to optimization.

A finite generating set of the vanishing ideal of a given variety (in other words, an algebraic description) can be tropicalized to define a polyhedral complex known as a tropical prevariety.
If this happens to coincide with the tropicalization of the variety itself, the generating set is called a \emph{tropical basis}.
We coin the notion of \emph{positive-tropical generators} as an analog of this property for the positive part. For determinantal varieties, all cases where the appropriate minors form a tropical basis have been classified in a series of works \cite{develin_ranktropicalmatrix,chan_4x4minors,shitov_whendor}. We take the first steps towards a characterization when they are also positive-tropical generators.

This question has been studied before for other varieties: 
\cite{SpeyerWilliams:2021,ArkaniHamedLamSpradlin:2021} showed that the 3-term Plücker relations form a set of positive-tropical generators of the tropical Grassmannian (even though they are, in general, not a tropical basis). 
The main result of \cite{Boretsky:2021} implies that the tropicalizations of the incidence Plücker relations form a set of positive-tropical generators of the tropical complete flag variety. 
For the tropical Pfaffian, \cite[Corollary 4.5]{ruiz22_multitriangulationstropicalpfaffians} implies that the polynomials defining the tropical Pfaffian prevariety, when restricting to a certain (Gröbner) cone, form a set of positive-tropical generators of the restriction of the tropical Pfaffian to this cone. 
In the context of cluster varieties, the proof of \cite[Proposition 4.1]{JahnLoeweStump:2021} implies that the generators of the cluster variety form a positive-tropical generating set (although it is unknown whether they form a tropical basis).

The notions of positivity differ in tropical geometry, e.g. distinguishing between positive solutions over the complex Puisseux series and positive solutions that are fully real. 
We therefore also introduce the notion of \emph{really positive-tropical generators}, which cut out the fully real, positive part.
Inspired by Viro's patchworking \cite{Viro:1983,Viro:2006} -- a combinatorial tool to construct real algebraic curves with prescribed topology --  we extend this idea of positive generators to arbitrary sign patterns, introducing the notion of \emph{(really) signed-tropical generators}. 
Generating sets for signed tropicalizations have been studied in \cite{tabera15_realtropicalbases} under the name `real tropical bases'.
Really signed-tropical generators turn out to allow for more flexibility and may exist even if real tropical bases do not.

Our main results are combinatorial criteria for the (signed) tropicalization of determinantal varieties, i.e. the set of matrices of bounded rank \cite{MillerSturmfels:2005}. This variety is closely related to the Grassmannian. 
For this study, we introduce the \emph{triangle criterion}, which is our main tool for identifying positivity. This criterion is purely combinatorial, and relies on the graph structure of the Birkhoff polytope. 
As a special case, we consider the determinantal varieties of low rank matrices, i.e. matrices of rank $2$ and $3$. 

In rank $2$, the $(3\times 3)$-minors form a tropical basis \cite{develin_ranktropicalmatrix}
and they are positive-tropical generators by \Cref{th:rank-2-positive-generators}. The points of the tropical variety (of matrices of Kapranov rank at most $2$) are matrices, whose column span is contained in a tropical line and the columns can be interpreted as marked points on this line. This is how Develin, and Markwig and Yu associate a bicolored phylogenetic tree to such a matrix in \cite{develin_modulispacen,markwig_spacetropicallycollinear}. We show that this tree determines the positivity of the tropical matrix: The (tropical) matrix lies in the tropicalization of the positive part if and only if the associated tree is a caterpillar tree (\Cref{cor:positive-trees-rk-2}).
This relies on the fact, that the nonnegative rank is equal to the rank for a real matrix of rank $2$ (with nonnegative entries) and a result from \cite{ardila_tropicalmorphismrelated}, showing that the positive part consists precisely of those matrices with Barvinok rank $2$ (\Cref{th:positive-rank-2-barvinok}).
The construction of bicolored phylogenetic trees realizes the tropical determinantal variety combinatorially as a subfan of the tropical Grassmannian \cite{markwig_spacetropicallycollinear}.
In the spirit of \cite{markwig_spacetropicallycollinear}, we establish a bijection on the level of the corresponding matrices and tropical Plücker vectors, highlighting that this bijection is induced by a simple coordinate projection (\Cref{th:pluecker-bijection}). 

In rank $3$, the $(4\times 4)$-minors are in general not a tropical basis and we do not know if they are positive-tropical generators. For the combinatorial criterion, we thus only obtain a necessary condition for positivity -- or, in other words, a combinatorial certificate of non-positivity (\Cref{th:starship-critereon}), which we call Starship Criterion. In this case, the column span of a matrix in the tropicalization of the determinantal variety is contained in a tropical plane. By \cite{herrmann_howdrawtropical}, a tropical plane is uniquely determined by its tree arrangement (namely the trees obtained by intersecting the plane with the hyperplanes at infinity). However, we show that the bicolored tree arrangement that we can derive from a tropical matrix (of Kapranov rank at most $3$) does not contain sufficient information to determine positivity: The main issue is that some of the marked points (coming from the columns of the matrix) can be on bounded faces of the tropical plane, whereas the tree arrangement is unable to capture this information (\Cref{ex:bicolored-tree-arrangements}). However, if the tropical matrix is positive, then the resulting arrangement of bicolored phylogenetic trees solely consists of caterpillar trees (\Cref{th:bicolored-tree-arrangements}).

In corank $1$, we show that the characterization of positivity for the determinantal hypersurface extends nicely to the other orthants (\Cref{lem:sign-flipping}).
 This heavily relies on the fact that in this case the tropical prevariety coincides with the tropical variety.

\medskip

Our paper is structured as follows:
In \Cref{section: positivity} we discuss the different notions of positivity in tropical geometry and generators of positivity. We extend this in \Cref{sec:signed-tropical-generators} to arbitrary orthants and we introduce tropical determinantal varieties in \Cref{sec:determantal-varieties}. 
\Cref{sec:hypersurfaces} covers determinantal hypersurfaces, whose Newton polytope is the Birkhoff polytope. 
In this section, we begin the combinatorial translation of positivity by introducing \emph{cartoons}, which leads to the triangle criterion in terms of cartoons, followed by its geometric version.
We extend this to all orthants in \Cref{sec:extension-to-all-orthants}. In \Cref{sec:labels}, we explain how one can describe maximal cones of the determinantal prevariety by unions of perfect matchings, and obtain the triangle criterion in terms of bipartite graphs. In \Cref{sec:rank-2}, we consider the special case of rank $2$, and describe a bijection between a subfan of the Grassmannian and the tropical determinantal variety. In \Cref{sec:rank-3}, we consider the rank 3 case. We obtain the starship criterion for positivity and consider bicolored tree arrangements.

\bigskip
\textbf{Acknowledgments.} We thank Jorge Olarte for fruitful discussions, and pointing us to \Cref{ex:tropical-linear-space}.
We are grateful to Bernd Sturmfels for suggesting the references on tropically collinear points.
We thank Michael Joswig for pointing us to~\cite{ruiz22_multitriangulationstropicalpfaffians} and Christian Stump for explaining the construction in~\cite{JahnLoeweStump:2021}. 

\section{Positivity in tropical geometry}\label{section: positivity}

In this section, we describe different notions of positivity that can be found in the literature. Based on the differences in these notions, we introduce \emph{(really) positive-tropical generating sets}, which characterize the (real) positive part of a tropical variety. We then generalize this to \emph{(really) signed-tropical generating sets}, which describe the signed tropicalization of a variety with respect to a fixed orthant, and discuss the differences between these notions. Finally, we introduce \emph{tropical determinantal varieties}, the main protagonists in this article.

\subsection{Notions of positivity}\label{sec:notions-of-positivity}

Let $\C = \bigcup_{n=1}^\infty \CC((t^{1/n}))$ and  $\R = \bigcup_{n=1}^\infty \RR((t^{1/n}))$ be the fields of Puisseux series over $\CC$ and $\RR$, respectively.
We denote by $\lc(x(t))$ the leading coefficient of a Puisseux series $x(t)$, i.e. the coefficient of the term of lowest exponent, and the leading term by $\lt(x(t))$.
The degree map $\val(x(t))$ returns the degree of the leading term, and we consider the valuation map $\val\left(x_1(t), \dots, x_n(t)\right) = ~\left(\val(x_1(t)) , \dots, \val(x_n (t))\right)$.
We define $\Cplus = \{ x(t) \in~ \C \mid \lc(x(t)) ~\in~ \RR_{>0} \}$ and  $\Rplus =~ \{ x(t) \in~\R \mid \lc(x(t)) \in \RR_{>0} \}$, which are both convex cones. (Note that this notation differs e.g. from \cite{speyer_tropicaltotallypositive}.)

We write $\T$ for the tropical semiring $(\T,\oplus,\odot) = (\RR\cup\{\infty\},\min, +)$.
Let $I \subseteq \C[x_1,\dots,n_n]$ be an ideal.
The tropicalization $\trop (V(I))$ of the variety $V(I)\subseteq~ \C^n$ is the closure of the set $\SetOf{\val(z)}{z \in V(I) \cap (\C^*)^n}\subseteq  (\RR\cup\{\infty\})^n$.
We consider $\trop(V(I))$ as a polyhedral complex, in which $w,w'$ are contained in the relative interior of the same face if $\initial_w(I)=\initial_{w'}(I)$.
We define the
\emph{positive part} of $\trop(V(I))$ to be $\tropcplus (V(I)) = \trop(V(I) \cap \mathcal C_+^n)$ and the 
\emph{really positive part} to be $\troprplus (V(I)) =~ \trop(V(I) \cap \mathcal R_+^n)$.
Similarly, a point $w \in \trop(V(I))$ is \emph{positive} (respectively \emph{really positive}) if it is contained in the positive part (respectively the really positive part).
For any set $B$ of generators of an ideal $I$ we have
\[
	\trop(V(I)) \subseteq \bigcap_{f \in B} \trop(V(f)),
\]
and so also 
\begin{equation}\label{eq:positive-part-inclusion}
	\tropcplus(V(I)) \subseteq \bigcap_{f \in B} \tropcplus(V(f)).
\end{equation}

We reserve the notation $\tropplus(V(I))$ for the case when $\trop^{+_{\mathcal C}} (V(I))  = \trop^{+_{\mathcal R}} (V(I))$ holds, so that there can be no confusion about the notion of positivity.
The positive part of a tropical variety was characterized by Speyer and Williams as follows:
\begin{prop}[{\cite[Proposition 2.2]{speyer_tropicaltotallypositive}}]\label{proposition:speyer-williams-positivity}
	A point $w$ lies in $\tropcplus(V(I))$ if and only if the initial ideal of $I$ with respect to $w$ does not contain any (non-zero) polynomial whose (non-zero) coefficients all have the same sign, i.e.~if and only if
	\[ \initial_w(I)~\cap~\RR_{>0}[x_1, \dots, x_n]=\inner{0}.\]
\end{prop}

This proposition implies that the positive part is closed.
There is an equivalent definition of positivity in terms of the tropicalizations of the polynomials in the vanishing ideal of the variety (for ideals in $\R[x_1,\ldots,x_n]$): 
Let $f = \sum_{e \in E^+} f_e x^e - \sum_{f_e \in E^-} f_e x^e \in~ \R[x_1,\dots,x_n]$ be a polynomial such that $f_e \in \Rplus$ for all $e \in E^+ \cup E^-$. (More generally, this definition can be made for polynomials $f\in \C[x_1,\ldots,x_n]$ as long as the coefficients are either in $\Cplus$ or $-\Cplus$.)
The \emph{combinatorially positive part $\tropcombplus(f)$ of the tropical hypersurface} $\trop(V(f))$ is the set of all points $w \in \trop(V(f))$ such that the minimum of $\{ \langle w,e \rangle + \val(f_e) \mid e \in E^+ \cup E^- \}$ is achieved at some $e \in E^+$ and at some $e \in E^-$.
In \cite{speyer_tropicalgrassmannian}, this definition is made for polynomials $f \in \RR[x_1,\dots,x_n]$, in which case $\val(f_e)=0$ for all $e \in E^+ \cup E^-$.
\begin{corollary}\label{cor:hypersurface-combinatorial-positivity}
 For hypersurfaces the positive part coincides with the combinatorially positive part, i.e. $\tropcombplus(f) = \tropcplus(V(f))$ for any polynomial $f\in \R[x_1,\ldots,x_n]$ (or more generally $f\in \C[x_1,\ldots,x_n]$ such that its coefficients are either in $\Cplus$ or $-\Cplus$). 
\end{corollary}
\begin{proof}
	By definition, $w \in \tropcombplus(f)$ if and only if the minimum of $\{ \langle w,e \rangle + \val(f_e) \mid e \in~ E^+ \cup E^- \}$ is achieved at some $e^+ \in E^+$ and at some $e^- \in E^-$. Equivalently, the initial form $in_w(f)$ contains the terms $f_{e^+} x^{e^+} - f_{e^-} x^{e^-}$, i.e. $	\initial_w(\langle f \rangle )~\cap~\RR_{>0}[x_1, \dots, x_n]=\inner{0}.$ By \Cref{proposition:speyer-williams-positivity}, this is equivalent to $w \in \tropcplus(V(f))$.
\end{proof}

Let $\mathcal{F}$ be a set of polynomials and $\mathcal P = \bigcap_{f \in \mathcal F} \trop(V(f))$ a tropical prevariety.
The \emph{combinatorially positive part of $\mathcal{P}$ with respect to $\mathcal F$} is $\tropcombplus(\mathcal P) = \bigcap_{f \in \mathcal F} \tropcombplus(f)$.
If $\mathcal P = \trop(V(I))$ is also a tropical variety, then $\tropcplus(V(I)) = \bigcap_{f \in I} \tropcombplus(V(f))$.
In this sense, the notions of positivity and combinatorial positivity agree for tropical varieties.

\subsection{Generators of positivity}

We make the following definitions.

\begin{definition}
	Let $\mathcal F \subseteq \R[x_1,\dots,x_n]$ be a finite set of polynomials.
	Then $\mathcal{F}$ is a set of \emph{positive-tropical generators} (or is a \emph{positive-tropical generating set}) if 
	$$ \trop( V(I) \cap \C_+^n ) = \bigcap_{f \in \mathcal F} \trop(V(f)\cap\C_+^n).$$
	
	It is a set of \emph{really positive-tropical generators} if
	$$ \trop( V(I) \cap \R_+^n ) = \bigcap_{f \in \mathcal F} \trop(V(f)\cap\R_+^n).$$
\end{definition}

We note that a set of positive-tropical generators is conceptually different from a tropical basis. However, under some circumstances a tropical basis is guaranteed to be a set of positive-tropical generators.

\begin{theorem} \label{thm:hypersurfaces-positive-bases}
	If $\trop(V(f))$ is a tropical hypersurface, then $f$ is a positive-tropical generator and a really positive-tropical generator for any $f\in \R[x_1,\ldots,x_n]$.
\end{theorem}

\begin{proof}
	This follows directly from the definition of (really) positive-tropical generators.
\end{proof}

\begin{theorem}
	For a binomial ideal, every tropical basis containing a reduced Gr\"obner basis (with respect to any ordering) forms a set of positive-tropical generators.
\end{theorem}
\begin{proof}
	Let $I$ be a binomial ideal with tropical basis $B$.
	By assumption, $B$ contains a reduced Gr\"obner basis $G$. 
	This Gr\"obner basis $G$ consists solely of binomials by~\cite[Proposition 1.1]{EisenbudSturmfels:1996}.
	Let $w \in \trop(V(I))$. Then $\initial_w(I) = \langle \initial_w(f) \mid f \in G \rangle$ does not contain a monomial and so $\initial_w(f) = f$. Thus, $I = \initial_w(I)$ and $G$ is a reduced Gröbner basis of $\initial_w(I)$. 
	Now, \Cref{proposition:speyer-williams-positivity} together with \cite[Lemma 5.6]{bendle_parallelcomputationtropical} implies that $w \in \tropcplus(V(I))$ if and only if $\initial_w(I) \cap \RR_{\geq 0}[x] = \langle 0 \rangle$ if and only if $G \cap \RR_{\geq 0}[x] = \emptyset$.
	Therefore, $w \in  \tropcplus(V(I)) = \bigcap_{f \in I} \tropcplus(V(f))$ if and only if $w \in \bigcap_{f \in G} \tropcplus(V(f))$. Note that 
	\[
		\bigcap_{f \in I} \tropcplus(V(f)) \subseteq \bigcap_{f \in B} \tropcplus(V(f)) \subseteq \bigcap_{f \in G} \tropcplus(V(f)),
	\]
	and so $w \in \tropcplus(V(I))$ if and only if $w \in \bigcap_{f \in B} \tropcplus(V(f))$.
\end{proof}

\begin{example}[The totally positive tropical Grassmannian]\label{th:grassmannian-positive-generators}
	Positive-tropical generators have been studied in the case of the tropical Grassmannian. More precisely, it was shown that the $3$-term Plücker relations are not a tropical basis, but they are indeed a positive-tropical generating set \cite{SpeyerWilliams:2021,ArkaniHamedLamSpradlin:2021}. 
	 It is also known that the positive part and the really positive part agree \cite{speyer_tropicaltotallypositive}.
	Hence, the $3$-term Plücker relations also form a really positive-tropical generating set.
\end{example}

If the positive part and the really positive part of a tropical variety coincide, then every set of positive-tropical generators is also a set of really positive-tropical generators, because

\begin{table}[h]
	\centering
\begin{tabular}{ccc}
	$\tropcplus V(I)$  & $=$ &  $ \displaystyle\bigcap_{f \in \mathcal F} \tropcplus(V(f))$ \\
	\rotatebox[origin=c]{90}{$=$}			&      & \rotatebox[origin=c]{90}{$\subseteq$} \\
	$\troprplus V(I)$  &  $\subseteq$  & $\displaystyle\bigcap_{f \in \mathcal F}  \troprplus(V(f))$.
\end{tabular}
\end{table}

\begin{remark}\label{remark:tabera-1}
  In \cite{tabera15_realtropicalbases} the notion of a real tropical basis was introduced. This is a finite generating set, which cuts out the signed tropicalization of the real part of the variety. In particular, a real tropical basis is always a set of really positive-tropical generators.
  We elaborate on this further in \Cref{remark:tabera-signed-generators}.
\end{remark}

For our notion of positive-tropical generators,
\Cref{th:grassmannian-positive-generators} and \Cref{table:positive-generators-overview} indicate that tropical bases and positive-tropical generators are distinct concepts of similar flavor. This motivates the following question.

\begin{question}\label{question:tropical-basis-positive-generators}
	Is there a tropical variety where a tropical basis is not a set of positive-tropical generators?
\end{question}

In particular, we raise this question for tropical determinantal varieties, which will be the main object of study in the following sections. 
Given \Cref{table:positive-generators-overview}, does this already fail for the minors?

\subsection{Signed-tropical generators}\label{sec:signed-tropical-generators}

We devote the remainder of this section to discuss how a description of a tropical hypersurface $\trop(V(f))$ for one orthant (the positive orthant) can be extended to all orthants by `flipping signs'.
This goes back to the idea of Viro's patchworking~\cite{Viro:2006}.
It is crucial that for a hypersurface, the notions of positivity and combinatorial positivity coincide (cf. \Cref{cor:hypersurface-combinatorial-positivity}). 
More precisely,  let $f = \sum_\alpha c_\alpha x^\alpha$ be a polynomial in $n$ variables $x = (x_1,\dots, x_n)$ and  $\sign\in~ \{-1,1\}^n$ a sign vector.
Analogously to the notions introduced in \Cref{sec:notions-of-positivity} we define 
\[
\C^\sign = \SetOf{ (\xi_1(t), \dots, \xi_n(t)) \in \C^n}{\lc(\xi_i) \in \RR \text{ and } \sgn(\lc(\xi_i)) = \sign_i \ \forall i \in [n] }
\]
and $\tropsigned(V(f)) = \trop(V(f) \cap \C^\sign)$.
We consider the modified polynomial $f^\sign = \sum_\alpha \left(\prod s_i^{\alpha_i}\right) c_\alpha x^\alpha$.
By construction, for a point $\xi = (s_1 \xi_1, \dots, s_n \xi_n), \xi_i \in \Cplus$ one obtains
\begin{equation*}
	f(\xi) = \sum_\alpha c_\alpha \xi^\alpha =  \sum_\alpha \left(\prod s_i^{\alpha_i}\right) c_\alpha (\xi_1, \dots, \xi_n)^\alpha = f^\sign(\xi_1, \dots, \xi_n).
\end{equation*}
In other words, $\xi \in V(f) \cap \C^\sign$ if and only if $(\xi_1, \dots, \xi_n) \in V(f^\sign) \cap \Cplus^n $ and hence
\begin{equation*}
	 \tropsigned(V(f)) = \tropcplus(V(f^\sign))  .
\end{equation*}

We can extend this idea to make the following definitions.
\begin{definition}
  Let $V(I)\subseteq \C^n$ be a variety and $\mathcal F \subseteq \R[x_1,\dots x_n]$ be a finite set of polynomials.
  Let $s\in \{-1,1\}^n$ be a fixed sign vector.
  The set $\mathcal F$ is a set of \emph{signed-tropical generators} of $\trop(V(I))$ with respect to $s$ if 
	\[
	\trop(V(I) \cap \C^\sign) = \bigcap_{f \in \mathcal F} \trop(V(f)\cap \C^\sign) = \bigcap_{f \in \mathcal F} \trop^\sign(V(f)).
	\]
	The finite set  $\mathcal F \subseteq \R[x_1,\dots x_n]$ is a set of \emph{really signed-tropical generators} of $\trop(V(I))$ with respect to $s$ if 
	\[
	\trop(V(I) \cap \R^\sign) = \bigcap_{f \in \mathcal F} \trop(V(f)\cap \R^\sign).
	\]
\end{definition}

\begin{prop}\label{prop:singed-tropical-generators-hypersurface}
	If $\trop(V(f))$ is a tropical hypersurface, then $f$ is a (really) signed-tropical generator for $\trop^\sign(V(f)$ with respect to every sign vector $\sign \in \{-1,1\}^n$ for any polynomial $f\in\R[x_1,\dots,x_n]$. 
	Furthermore, if $f = \sum_{\alpha} c_\alpha x^\alpha$, then $\trop^\sign(V(f)) = \tropcplus(V(f^\sign))$, where 
	$f^\sign = \sum_\alpha \left(\prod s_i^{\alpha_i}\right) c_\alpha x^\alpha$.
\end{prop}

\begin{proof}
	The first part of the statement follows directly from the definition of (really) signed-tropical generators. The second part is implied by the discussion above.
\end{proof}

Let $P$ denote the tropical prevariety $P = \bigcap_{f \in \mathcal F} \trop(V(f))$ with finite generating set $\mathcal F$. Note that, similarly to the positive part, also for the more general signed part we have 
$$\trop(V(I) \cap \C^\sign) \subseteq \bigcap_{f \in \mathcal F} \trop^\sign(V(f)) \subseteq P.$$ 
In this sense, one can interpret the ``signed-tropical prevariety'' $\bigcap_{f \in \mathcal F} \trop^\sign(V(f))$ as a combinatorial approximation of the signed tropicalization $\trop(V(I) \cap \C^\sign)$. 
Thus, when considering signed tropicalizations, a finite set $\mathcal F$ that is a signed-tropical generating set with respect to every sign pattern $\sign \in \{-1,1\}^n$ simultaneously might be a useful tool for understanding the different orthants $\trop(V(I) \cap \C^\sign)$ in a combinatorial fashion.
Note that a set of positive-tropical generators does not necessarily has to be a signed-tropical generating set for other orthants, as illustrated in the following example.

	\begin{example}[Positive generators do not generate all orthants]\label{ex:tropical-linear-space}
		Consider the tropicalization of the linear space that is the row span of $M$ with Plücker vector $p$ given by
		\[
		M = \ma 
		1 & 0 & -1 & 1 \\
		0 & 1 & -1 & -2
		\trix, \
		\begin{array}{rccccccl}
			&12 & 13 & 14  & 23 & 24 & 34 & \\
			p = ( &  1 &  -1 &  -2  &   1 &  -1 &  3 & ).
		\end{array}
		\]
		The linear space $L$ is given by the polynomials
		\begin{align*}
			f_1 = p_{12} x_3 - p_{13} x_2 + p_{23} x_1 &= x_3 + x_2 + x_1 = 0 \\
			f_2 = p_{13} x_4 - p_{14} x_3 + p_{34} x_1 &= -x_4 + 2x_3 + 3x_1 = 0 \\
			f_3 = p_{12} x_4 - p_{14} x_2 + p_{24} x_1 &= x_4 + 2x_2 - x_1 = 0 \\
			f_4 = p_{23} x_4 - p_{24} x_3 + p_{34} x_2 &= x_4 + x_3 + 3x_2 = 0
		\end{align*}
		and $\tropcplus(L) \subseteq \bigcap_{i=1}^4 \tropcplus(V(f_i))$.
                Note that $\tropcplus(V(f_1)) = \emptyset$, so $\tropcplus(L) = \emptyset$ and $\{f_1\}$ is a positive-tropical generating set. 
		Let $s = (-1,1,1,1)$.
                Then $\tropsigned(L) \subseteq \bigcap_{i=1}^4 \tropsigned(V(f_i))$.
                Since $f^s_3$ has only positive coefficients, it follows that $\tropsigned(V(f_3)) = \emptyset$, and so  $\tropsigned(L) = \emptyset$.
                However, $\tropsigned(V(f_1))$ is non-empty, so $\{f_1\}$ is not a signed-tropical generating set with respect to $s$.
	\end{example}

	\begin{remark}\label{remark:tabera-signed-generators}
		As mentioned in \Cref{remark:tabera-1}, a real tropical basis \cite{tabera15_realtropicalbases} cuts out a signed version of the tropicalization of the real part of the variety.
		By definition, a real tropical basis is a set of 
		really signed-tropical generators with respect to every sign pattern $\sign \in \{-1,1\}^n$ simultaneously.
		We note however, that the converse is not true. For example, there are hypersurfaces for which there exists no real tropical basis \cite[Example 3.15]{tabera15_realtropicalbases}. On the other hand, by \Cref{prop:singed-tropical-generators-hypersurface} the defining polynomial of a hypersurface is always a really signed-tropical generator for every orthant.
	\end{remark}

\subsection{Determinantal varieties}\label{sec:determantal-varieties}

We set the stage for the following sections by introducing the determinantal varieties we will consider. 
Let $I_r \subseteq \C[x_{ij} \mid ij \in [d]\times[n]]$ be the ideal generated by all $(r+1) \times (r+1)$-minors of a $(d\times n)$-matrix.
The determinantal variety $V(I_r) \subseteq \C^{d\times n}$ consists of all $(d\times n)$-matrices of rank at most $r$.
As in the case of the Grassmannian, also for the tropicalization of  determinantal varieties the positive part and the really positive part coincide.
The proof of the next statement is deferred to \Cref{sec:appendix-proof-positivity}.

\begin{prop} \label{prop:mixed-to-real-lifts}
	Let $A \in \C^{d \times n}$ be a matrix such that the leading coefficient of every entry $A_{ij}\in \C$ is real.
	Then there exists a matrix $B \in \mathcal \mathcal \R^{d\times n}$ of real Puisseux series that has the same rank as $A$ and the Puisseux series in every entry has the same leading term, meaning that $\lt(A_{ij}) = \lt(B_{ij})$ holds for all $(i,j)\in [d]\times[n]$.
\end{prop}

\begin{corollary} \label{cor:really-positive-part-minors}
	The positive and the really positive part of the tropicalization of the variety $V(I_r) \subseteq \C^{d \times n}$ of $(d\times n)$-matrices of rank at most $r$ coincide: 
	\[ \trop^{+_{\mathcal C}} (V(I_r))  = \trop^{+_{\mathcal R}} (V(I_r)).\]
	In particular, every set of positive-tropical generators for the ideal $I_r$ of $(r+1)\times(r+1)$-minors is a set of really positive-tropical generators.\qed
\end{corollary}

We denote by $\tropdet$ the tropicalization of the determinantal variety of $(d\times n)$-matrices of rank at most $r$.
Since the minors of a matrix are polynomials with constant coefficients, the tropical determinantal variety $\tropdet$ is a polyhedral fan, and its positive part is a closed subfan. 
By \Cref{cor:really-positive-part-minors} the positive part is independent of the choice between $\C$ and $\R$, hence we denote it by $\tropplus(V(I_r)) = \tropdetplus$.
While the ideal $I_r$ is generated by the $(r+1)\times(r+1)$-minors, this does not necessarily carry over to the tropical variety $\tropdet$. 
Indeed, in a sequence of works, it has been characterized when they actually form a tropical basis.

\begin{theorem}[\cite{develin_ranktropicalmatrix, chan_4x4minors,shitov_whendor}]\label{thm:shitov}
	The $(r \times r)$-minors of a $(d \times n)$-matrix of variables form a tropical basis of the ideal they generate if and only if $r \leq 3$, or $r = \min\{d, n\}$, or else $r = 4$ and $\min\{d, n\} \leq 6$.
\end{theorem}

It is thus worthwhile to define the \emph{tropical determinantal prevariety} 
\begin{equation*}
	\prevar = \bigcap_{\substack{f \text{ is a } \\ ((r+1)\times (r+1))\text{-minor}}} \trop(V(f)) = \bigcap_{\substack{I \subseteq \binom{[d]}{r+1} \\ J \subseteq \binom{[n]}{r+1}}} \trop(V(f^{IJ})).
\end{equation*}
and the \emph{positive tropical determinantal prevariety} 
\begin{equation} \label{eq:positive-tropical-determinantal-prevariety}
	\prevarplus = \bigcap_{\substack{f \text{ is a } \\ ((r+1)\times (r+1))\text{-minor}}} \tropcplus(V(f))  = \bigcap_{\substack{I \subseteq \binom{[d]}{r+1} \\ J \subseteq \binom{[n]}{r+1}}} \tropcplus(V(f^{IJ})),
\end{equation}
where $f^{IJ}$ denotes the polynomial corresponding to the minor given by the rows indexed by $I$ and columns indexed by $J$.
As mentioned above, we have $\tropdet \subseteq \prevar$ and $\tropdetplus \subseteq \prevarplus$.
We emphasize again, that the notion of positivity for prevarieties is purely combinatorial, and the inclusion of a positive tropical variety and the corresponding positive tropical prevariety may be strict.

\begin{table}
	\begin{tabular}{|c|c|c|c|c|c|}
		\hline
		\rowcolor{TableHeader} & $d=3$  & $d=4$   & $d=5$   & $d=6$  & $d=7$   \\ \hline
		\cellcolor{TableHeader} $r= 2$  
		& \cellcolor{TableHighlight}\begin{tabular}[c]{@{}c@{}}YES\\ \Cref{thm:hypersurfaces-positive-bases,th:rank-2-positive-generators}\end{tabular} 
		&\cellcolor{TableHighlight}\begin{tabular}[c]{@{}c@{}}YES\\ \Cref{th:rank-2-positive-generators}\end{tabular}  
		&\cellcolor{TableHighlight}\begin{tabular}[c]{@{}c@{}}YES\\ \Cref{th:rank-2-positive-generators}\end{tabular}   
		&\cellcolor{TableHighlight}\begin{tabular}[c]{@{}c@{}}YES\\ \Cref{th:rank-2-positive-generators}\end{tabular}
		&\cellcolor{TableHighlight}\begin{tabular}[c]{@{}c@{}}YES\\ \Cref{th:rank-2-positive-generators}\end{tabular} \\ \hline
		\cellcolor{TableHeader} $r = 3$ &  
		&\cellcolor{TableHighlight}\begin{tabular}[c]{@{}c@{}}YES\\ \Cref{thm:hypersurfaces-positive-bases}\end{tabular} 
		&\cellcolor{TableHighlight}?                                            
		&\cellcolor{TableHighlight}? 
		& ?   \\ \hline
		\cellcolor{TableHeader} $r=4$   &   &  
		&\cellcolor{TableHighlight}\begin{tabular}[c]{@{}c@{}}YES\\ \Cref{thm:hypersurfaces-positive-bases}\end{tabular}
		& ? 
		& ?    \\ \hline
		\cellcolor{TableHeader} $r=5$   &  &  &  
		&\cellcolor{TableHighlight}\begin{tabular}[c]{@{}c@{}}YES\\ \Cref{thm:hypersurfaces-positive-bases}\end{tabular} 
		& ?    \\ \hline
		\cellcolor{TableHeader} $r=6$   &  &  &  &
		&\cellcolor{TableHighlight}\begin{tabular}[c]{@{}c@{}}YES\\ \Cref{thm:hypersurfaces-positive-bases}\end{tabular} \\ \hline
	\end{tabular}
	\caption{When do the $(r+1)\times (r+1)$-minors form a set of positive-tropical generators, i.e. when is $\tropdetplus=\prevarplus$ for $d\leq n$? A cell is colored in gray if the set of minors forms a tropical basis according to \Cref{thm:shitov}.}
	\label{table:positive-generators-overview}
\end{table}

We can interpret the matrices in the above sets in terms of the different notions of ranks for tropical matrices.
\begin{definition}[Tropical notions of rank] \label{def:tropical-notions-rank}
	Let $A \in \T^{d \times n}$ be a tropical matrix and $M \subseteq \T^{r\times r}$ a submatrix.
	The submatrix $M$ is \emph{tropically singular} if the minimum in the evaluation of the tropical determinant 
	$
	\bigoplus_{\sigma \in S_r} \left(\bigodot_{i = 1}^r M_{i \sigma(i)}  \right) =
        \min_{\sigma \in S_r} \left(\sum_{i = 1}^r M_{i \sigma(i)}  \right)
	$
	is attained at least twice.
	The \emph{tropical rank} of $A$ is the largest integer $r$ such that $A$ has a tropically non-singular submatrix. 
	The \emph{Kapranov rank} of $A$ is the smallest integer $r$ such that there exists a matrix $\tilde{A} \in \C^{d \times n}$ of rank $r$ such that $A = \val(\tilde{A})$.
 	 The \emph{Barvinok rank} of $A$ is the smallest integer $r$ for which $A$ can be written as the tropical sum of $r$ rank-$1$ matrices. A $(d\times n)$-matrix has rank $1$ if it is the tropical matrix product of a $(d\times 1)$-matrix and a $(1\times n)$-matrix.
\end{definition}
It was shown in \cite{develin_ranktropicalmatrix} that 
\begin{equation}\label{eq:inequalities-of-rank}
	\text{tropical rank of } A \leq \text{ Kapranov rank of } A \leq \text{ Barvinok rank of } A
\end{equation}
and that indeed all of these inequalities can be strict. In the light of these notions of rank, we can view the tropical determinantal variety $\tropdet$ and prevariety $\prevar$ as sets as
\begin{align*}
	\tropdet &= \{ A \in \T^{d  \times n} \mid A \text{ has Kapranov rank } \leq r\}\,, \\
		\prevar &= \{ A \in \T^{d  \times n} \mid A \text{ has tropical rank } \leq r\}\;.
\end{align*}
Note that the first inequality in \eqref{eq:inequalities-of-rank} also implies the inclusion $\tropdet \subseteq \prevar$.
We now describe the geometric interpretations of these notions.
Throughout this article, it will be enough to consider tropical linear spaces that arise as tropicalizations of classical linear spaces. The \emph{tropical projective torus} is $\TP = \T/(\RR \odot (1,\dots,1))$, i.e. modulo tropical scalar multiplication. The following is a well-known fact in tropical geometry. We provide a proof in \Cref{sec:appendix-proof-positivity} for completeness.

\begin{prop}\label{prop:kapranov-rank-points-on-hyperplane}
	Let $A \in \tropdet$. Then the columns of $A$ are $n$ points in $\TP^{d-1}$ lying on a tropical linear space of dimension at most $r-1$.
\end{prop}

\section{Determinantal hypersurfaces}\label{sec:hypersurfaces}

In this section, we seek to understand the positive part of the tropicalization of singular quadratic matrices. 
By definition of the Kapranov rank, the set $\tropdet[n-1][n,n]$ is formed by all tropical $(n\times n)$-matrices of Kapranov rank at most $n-1$.
Let $A\in \T^{n\times n}$ be such a matrix.
We can interpret the columns of $A$ as a point configuration of $n$ labeled points on a tropical hyperplane in the tropical projective torus $\TP^{n-1}$.
In this section, we characterize the \emph{positive point configurations}, those given by matrices $A$ in the positive part $\tropdetplus[n-1][n,n]$ of the tropical variety.
We say that a cone (or a point) is positive if it lies in the positive part.

\subsection{Edges of the Birkhoff polytope}\label{sec:Birkhoff-polytope}

We begin by investigating the maximal cones of $\tropdetplus[n-1][n,n]$ in the tropical hypersurface $\tropdet[n-1][n,n] = \trop(V(\det))$, where we abbreviate
\begin{equation}
		\det = \sum_{\sigma \in S_n}  \left( \sgn(\sigma) \prod_{i=1}^{n} x_{i\sigma(i)} \right) \,.
\end{equation}

This entails that $\tropdet[n-1][n,n]$ is the $\codim 1$-skeleton of the normal fan of the Newton polytope of the polynomial $\det$.
It is the well-known \emph{Birkhoff polytope} $\Birk_n$ (also called \emph{perfect matching polytope} or \emph{assignment polytope}) whose vertices are the $(n\times n)$-permutation matrices.
The sum of all entries in a row or a column of a matrix in $\Birk_n$ is $1$.
Therefore, $\tropdet[n-1][n,n]$ has a lineality space spanned by the vectors
\begin{equation}\label{eq:lineality space}
	 \sum_{i=1}^n E_{ij} \text{ for } j\in[n] \text{ and }	\sum_{j=1}^n E_{ij} \text{ for } i\in[n],
\end{equation}
where $E_{ij}$ denotes the standard basis matrix in $\RR^{n\times n}$. For notational convenience we identify a permutation $\sigma \in S_n$ with the permutation matrix that represents it.
Two vertices of $\Birk_n$ (corresponding to permutations $\sigma, \pi\in \Sym_n$) are connected by an edge if and only if $\sigma \pi^{-1}$ is a cycle \cite{billera96_combinatoricspermutationpolytopes}.
A maximal cone $C \subseteq \tropdet[n-1][n,n]$ is the normal cone of an edge $\conv(\sigma, \pi)$ of $\Birk_n$ for $\sigma, \pi \in S_n$ (such that $\sigma \pi^{-1}$ is a cycle).
For a weight vector $w$ in the interior $\interior(C)$ of the cone, the initial ideal $\initial_w(I)$ of $I = \langle \det \rangle$ is generated by the binomial $\initial_w(\det) = \sgn(\sigma) \prod_{i=1}^{n} x_{i\sigma(i)}+ \sgn(\pi)  \prod_{i=1}^{n} x_{i\pi(i)}$.
Applying \Cref{proposition:speyer-williams-positivity} to this polynomial yields a characterization of positive maximal cones of $\tropdet[n-1][n,n]$.

\begin{prop}\label{cor: positive iff different signs of permutations}
	Let $C \subseteq \tropdet[n-1][n,n]$ be a maximal cone which is dual to an edge $\conv(\sigma, \pi)$ of the Birkhoff polytope $\Birk_n$.
	Then $C$ is positive if and only if $\sgn(\sigma) \neq \sgn(\pi)$.\qed
\end{prop}

\begin{remark}\label{rem:maximal-cones-positive-corank-1}
\Cref{cor: positive iff different signs of permutations} fully characterizes the positivity of all cones of $\tropdet[n-1][n,n]$. 
Let $C$ be a (low-dimensional) positive cone of $\tropdet[n-1][n,n]$ and $A \in \interior(C)$. The initial form $\initial_A(\det)$ has terms of mixed signs. Since every monomial of the initial form corresponds to a vertex of $B_n$, and the edge graph of the face $F_C$ dual to $C$ is connected, this implies that there is an edge of $F_C$ whose vertices correspond to monomials (i.e. permutations) of different signs. This edge is dual to a positive maximal cone of $\tropdet[n-1][n,n]$ containing $C$. 
\end{remark}

\subsection{Triangle criterion for positivity}\label{sec:triangle-crit}

In this section, we identify the positive part of the tropical determinantal hypersurface $\tropdet[n-1][n,n]$. We make use of \Cref{cor: positive iff different signs of permutations} to obtain the \emph{triangle criterion}. 
It turns out that this works well for $n = 3,4$ but there is an example for $n = 5$ where this fails. 
\Cref{rem:maximal-cones-positive-corank-1} implies that it is suffices to consider maximal cones of $\tropdet[n-1][n,n]$.
The triangle criterion assigns a cartoon to each such maximal cone.
We seek to determine the positivity of this cone from the respective cartoon.
First, we give the construction of the cartoon and give the triangle criterion for detecting positivity for $n=3,4$.  
Afterwards, we describe its geometric interpretation in terms of tropical point configurations, and show that the triangle criterion does not hold for $n\geq 5$. 

\begin{const}[Cartoons of maximal cones]\label{construction:triangle-criterion}\label{construction:cartoon-for-cones}
	Let $C \subseteq \tropdet[n-1][n,n]$ be a maximal cone. Then $C$ is dual to an edge $\conv(\sigma,\pi)$ of the Birkhoff polytope $B_n$, whose vertices correspond to permutations in $S_n$. Let $K_n$ denote the complete graph on nodes $\{v_1, \dots, v_n\}$. To obtain the \emph{cartoon} of $C$ we decorate the complete graph with $n$ points placed on edges and nodes of $K_n$ as follows: for each $j \in [n]$, decorate the edge $v_{\sigma^{-1}(j)}v_{\pi^{-1}(j)}$ of $K_n$ with a marking if $\sigma^{-1}(j)\neq \pi^{-1}(j)$. If $\sigma^{-1}(j) = \pi^{-1}(j)$, decorate the vertex $v_{\sigma^{-1}(j)}$.
	\end{const}
	
		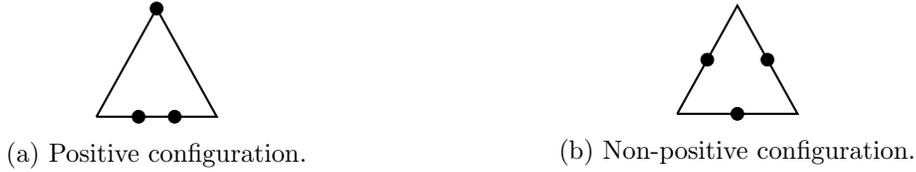
\begin{figure}
		\centering
		\begin{subfigure}{0.49 \textwidth}
			\centering
			\begin{tikzpicture}[scale=0.8]
				\filldraw (1,1.8) circle (3pt);
				\filldraw (0.7,0) circle (3pt);
				\filldraw (1.3,0) circle (3pt);
				\draw[thick] (0,0) -- (2,0) -- (1,1.8) -- (0,0);
			\end{tikzpicture}
			\caption{Positive configuration.}
			\label{subfig:transposition3}
		\end{subfigure}
		\begin{subfigure}{0.49 \textwidth}
			\centering
			\begin{tikzpicture}[scale = 0.8]
				\filldraw (1,0) circle (3pt);
				\filldraw (0.5,0.9) circle (3pt);
				\filldraw (1.5,0.9) circle (3pt);
				\draw[thick] (0,0) -- (2,0) -- (1,1.8) -- (0,0);
			\end{tikzpicture}
			\caption{Non-positive configuration.}
			\label{subfig:cycle3}
		\end{subfigure}
		\caption{The possible cartoons of maximal cones in $\tropdet[2][3,3]$.}
		\label{fig: proof of starship criterion}
	\end{figure}

	\begin{example}[Cartoons of maximal cones]\label{ex:cartoon-example}
		Consider the cone $C = \cone(E_{13},E_{23},E_{31},E_{32}) \subseteq \tropdet[2][3,3]$. It is dual to the edge $\conv(\sigma, \pi)$, where $\sigma = (1,2)$ is a transposition and $\pi = id$. The cartoon is a decorated $K_3$, with two markings on the edge $v_1 v_2$ and a marking placed at the node $v_3$. A cartoon of this type is shown in \Cref{subfig:transposition3}.
	\end{example}

\begin{prop}[Triangle criterion for cartoons]\label{prop: triangle criterion}
	Let $n = 3,4$ and $C \subseteq \tropdet[n-1][n,n]$ be a maximal cone.
	$C$ is positive if and only if its cartoon does not contain a marked triangle, i.e. a triangle with three distinct markings, where each edge contains at least one marking in its interior or on an incident vertex.
\end{prop}

\begin{proof}
	
	The cone $C$ is dual to an edge $\conv(\sigma, \pi)$. 
	Two permutations form an edge of $B_n$ if and only if $\sigma \pi^{-1}$ is a cycle. 
	By \Cref{cor: positive iff different signs of permutations} the cone $C$ is positive if and only if $\sigma \pi^{-1}$ is a cycle of even length.
          Let $\sigma \pi^{-1}$ be a cycle of length $\ell$, so it is of the form $(i_1, \dots, i_{\ell})$.
          Now, denote $I = \{i_1, \dots, i_{\ell}\}$ and $J = [n] \setminus I$.
          Then \Cref{construction:triangle-criterion} decorates each node $v_{\sigma^{-1}(j)}$ for $j \in J$ and decorates the cycle formed by the edges $(v_{\sigma^{-1}(i_{k+1})},v_{\pi^{-1}(i_{k})})$ (where $i_{\ell+1} = i_1$).
          Therefore, up to symmetry, it is enough to consider the potential cycle lengths to determine the configurations.

	\Cref{fig: proof of starship criterion} shows all possible cartoons for $n=3$, up to permutation of the nodes of the graph.
	More precisely, 
	\Cref{subfig:transposition3} shows the cartoon for when $\sigma \pi^{-1}$ is a transposition and \Cref{subfig:cycle3} shows the cartoon for when $\sigma \pi^{-1}$ is a $3$-cycle.
	The possible cartoons for $n=4$ are shown in \Cref{fig: proof of starship cirterion n=4}: \Cref{subfig:transposition4} shows the cartoon for when $\sigma \pi^{-1}$ is a transposition, \Cref{subfig:cycle43} the cartoon of a $3$-cycle, and \Cref{subfig:cycle44} the cartoon of a $4$-cycle.
       Summarizing, the configurations depicted in \Cref{subfig:transposition3}, \Cref{subfig:transposition4} and \Cref{subfig:cycle44} are positive, while \Cref{subfig:cycle3} and \Cref{subfig:cycle43} are negative. 
\end{proof}

	\begin{figure}
	\centering
	\begin{subfigure}{0.3 \textwidth}
		\centering
		\begin{tikzpicture}[scale = 0.4]
			\node (1) at (4, 11) {};
			\node (2) at (1, 6) {};
			\node (3) at (9, 7) {};
			\node (4) at (7, 4.75) {};
			\draw[thick] (2.center) to (4.center);
			\draw[thick] (4.center) to (3.center);
			\draw[thick] (3.center) to (1.center);
			\draw[thick] (1.center) to (2.center);
			\draw[thick] (2.center) to (3.center);
			\draw[thick] (1.center) to (4.center);
			\filldraw (7.75, 5.5) circle (7pt);
			\filldraw (8.25, 6.25) circle (7pt);
			\filldraw (1) circle (7pt);
			\filldraw (2) circle (7 pt);
		\end{tikzpicture}
		\caption{Positive configuration.}
		\label{subfig:transposition4}
	\end{subfigure}
	\begin{subfigure}{0.35 \textwidth}
		\centering
		\begin{tikzpicture}[scale = 0.4]
			\node (1) at (4, 11) {};
			\node (2) at (1, 6) {};
			\node (3) at (9, 7) {};
			\node (4) at (7, 4.75) {};
			\draw[thick] (2.center) to (4.center);
			\draw[thick] (4.center) to (3.center);
			\draw[thick] (3.center) to (1.center);
			\draw[thick] (1.center) to (2.center);
			\draw[thick] (2.center) to (3.center);
			\draw[thick] (1.center) to (4.center);
			\filldraw (1) circle (7pt);
			\filldraw (5, 6.5) circle (7pt);
			\filldraw (4.25, 5.25) circle (7pt);
			\filldraw (8, 5.75) circle (7pt);
		\end{tikzpicture}
		\caption{Non-positive configuration.}
		\label{subfig:cycle43}
	\end{subfigure}
	\begin{subfigure}{0.3 \textwidth}
		\centering
		\begin{tikzpicture}[scale = 0.4]
			\node (1) at (4, 11) {};
			\node (2) at (1, 6) {};
			\node (3) at (9, 7) {};
			\node (4) at (7, 4.75) {};
			\draw[thick] (2.center) to (4.center);
			\draw[thick] (4.center) to (3.center);
			\draw[thick] (3.center) to (1.center);
			\draw[thick] (1.center) to (2.center);
			\draw[thick] (2.center) to (3.center);
			\draw[thick] (1.center) to (4.center);
			\filldraw (4.25, 5.25) circle (7pt);
			\filldraw (8, 5.75) circle (7pt);
			\filldraw (2.5, 8.5) circle (7pt);
			\filldraw (6.75, 8.75) circle (7pt);
		\end{tikzpicture}
		\caption{Positive configuration.}
		\label{subfig:cycle44}
	\end{subfigure}
	\caption{The possible cartoons of maximal cones in $\tropdet[3][4,4]$.}
	\label{fig: proof of starship cirterion n=4}
\end{figure}
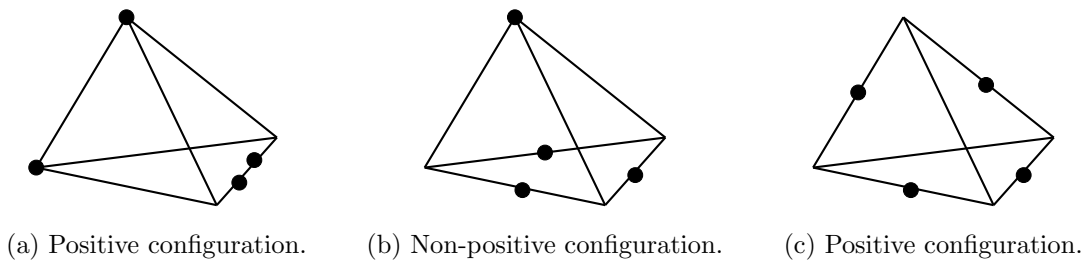

\begin{example}[Triangle criterion fails for $n\geq 5$]\label{ex:triangle-crit-counterexample}
	Let $C \subseteq \tropdet[4][5,5]$ be the maximal cone that is dual to the edge $\conv(\sigma,\pi)$ of $\Birk_5$, where $\sigma = (4,5)$ is a transposition and $\pi = id$.
	Then, modulo the lineality space of $\tropdet[4][5,5]$, every matrix $A \in C$ satisfies the zero pattern
	\[
	\ma
	0 &  &  &  &  \\
	& 0 &  &  & \\
	& & 0 &  &  \\
	& &  & 0 & 0 \\
	& & & 0 & 0 
	\trix,
	\]
	i.e. $A_{ij} = 0$ whenever $\sigma(i) = j$ or $\pi(i) = j$, and all other entries of $A$ are nonnegative.
	The cone $C$ has $18$ rays, corresponding to the blank spaces in the zero pattern above. By \Cref{cor: positive iff different signs of permutations} this cone is positive. However, the cartoon of $C$, as shown in \Cref{fig:constructionDiagram}, contains a triangle in which each node is decorated with a marking. This example can be generalized to any $n\geq 5$.
	\begin{figure}[h]
	\centering
	\begin{tikzpicture}[scale = 0.45]
		\node (1) at (4, 11) {};
		\node (2) at (1, 6) {};
		\node (3) at (9, 7) {};
		\node (4) at (7, 4.75) {};
		\node (5) at (4, 7.75) {};
		\draw[thick] (2.center) to (4.center);
		\draw[thick] (4.center) to (3.center);
		\draw[thick] (3.center) to (1.center);
		\draw[thick] (1.center) to (2.center);
		\draw[thick] (2.center) to (3.center);
		\draw[thick] (1.center) to (4.center);
		\draw[thick]  (5.center) to (1.center);
		\draw[thick]  (2.center) to (5.center);
		\draw[thick]  (5.center) to (4.center);
		\draw[thick]  (5.center) to (3.center);
		\filldraw (7.75, 5.5) circle (7pt);
		\filldraw (8.25, 6.25) circle (7pt);
		\filldraw (1) circle (7pt);
		\filldraw (2) circle (7 pt);
		\filldraw (5) circle (7pt);
	\end{tikzpicture}
	\caption{The diagram of the cone in \Cref{ex:triangle-crit-counterexample}.}
	\label{fig:constructionDiagram}
\end{figure}
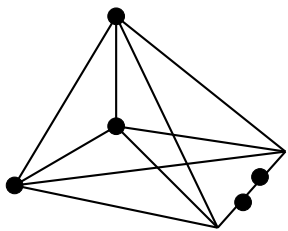
\end{example}

\subsection{Geometric triangle criterion}\label{sec:geom-triangle-crit}

	As described in \Cref{prop:kapranov-rank-points-on-hyperplane}, the columns of $A$ can be viewed as $n$ points in $\TP^{n-1}$ lying on a common tropical linear space of dimension $n-2$.
	We now show how the cartoons describe the geometry of these point configurations. But first, we describe two different important kinds of cones.
	The Birkhoff polytope $B_n \subseteq \RR^{n \times n}$ has vertices corresponding to permutations in $S_n$. Modulo lineality space of the normal fan of $B_n$, an edge $\conv(\sigma, \pi)$ has normal cone 
	\begin{equation}\label{eq:rays-of-cone-tropdet}
		C = \cone(E_{ij} \mid i,j \in [n], \sigma(i) \neq j, \pi(i) \neq j).
	\end{equation}
	The standard simplex $\Delta_{n-1} \subseteq \RR^n$ is the convex hull of the unit vectors $e_1,\dots, e_n$. Modulo lineality space of the normal fan of $\Delta_{n-1}$, an edge $\conv(e_k,e_l), k,l \in [n]$ has normal cone
	\begin{equation}
		W_{kl} = \cone(e_i \mid i \in [n], i \neq k, i \neq l).
	\end{equation}
	
	Up to translation, there is a unique tropical hyperplane $H$ of dimension $n-2$ in $\TP^{n-1}$.
	This hyperplane can be viewed as the codimension-$1$ skeleton of the normal fan of $\Delta_{n-1}$.
	Equivalently, the tropical hyperplane $H=H_c$ is the set of points 
	\begin{equation}\label{eq:tropical-hyperplane}
		H_c = \{x \in \TP^{d-1} \mid \text{ the minimum of } x_i + c_i, \ i \in [n] \text{ is attained at least twice} \}
	\end{equation}
	and the point $-(c_1, \dots, c_d)$ is the \emph{apex} of $H$. 
 We call a cone $W_{kl}$ of dimension $n-2$ a \emph{wing} of $H$.

	\begin{example}[A tropical point configuration]\label{ex:cartoon-example-cont}
		Let $C$ be the cone from \Cref{ex:cartoon-example} and consider the matrix
		\[
		A = \ma 
		0 & 0 & 2 \\
		0 & 0 & 1 \\
		3 & 1 & 0 \trix \in \interior(C).
		\] 
		The point configuration in $\TP^2$ is displayed in \Cref{fig:cartoon-example-cont} (in the chart where the last coordinate is $0$). The points lie on the common hyperplane with apex $(1, 1, 0)$.
		The first two columns lie on the wing $W_{1,2} = \cone(e_3)$. The third column lies on the wing $W_{2,3} = \cone(e_1)$.
		
		\begin{figure}
			\centering
			\begin{tikzpicture}[scale=0.35]
				\filldraw (-3,-3) circle (7pt);
				\filldraw (-1,-1) circle (7pt);
				\filldraw (2,1) circle (7pt);
				\draw[thick] (-4,-4) -- (1,1) -- (4,1);
				\draw[thick] (1,1) -- (1,4);
			\end{tikzpicture}
			\caption{The point configuration from \Cref{ex:cartoon-example-cont}.}
			\label{fig:cartoon-example-cont}
		\end{figure}
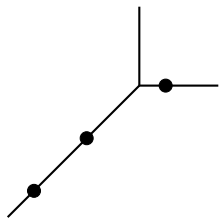
	\end{example}

	The lineality space of $\tropdet[n-1][n,n]$ is spanned by the vectors in \eqref{eq:lineality space} (in \Cref{sec:Birkhoff-polytope}). We describe the more general lineality space of $\tropdet$ in more detail in \Cref{excursion:lineality-spaces}. However, we exploit one main property here.
	
	\begin{lemma}\label{lem:wlog-matrix-nonneg-mod-linspace}
		Let $C \subseteq \tropdet[n-1][n,n]$ be a cone and $A \in C$.  
		There exists a matrix $A' \in \tropdet[n-1][n,n]$ such that $A \sim A'$ modulo lineality space of $\tropdet[n-1][n,n]$ and $A'_{ij} \geq 0$ for all $i,j \in [n]$. Furthermore, the columns of $A'$ are points on the tropical hyperplane $H_0$ with apex at the origin. 
		If $C$ is a maximal cone dual to the edge $\conv(\sigma,\pi)$ of $B_n$, then $A'_{ij} = 0$ if $j \in \{\sigma(i), \pi(i)\}$.
	\end{lemma}

\begin{proof}
	Let $A \in C$.
	Then by \eqref{eq:rays-of-cone-tropdet} there is a matrix $A' \in C$ such that $A \sim A'$ modulo lineality space of $\tropdet[n-1][n,n]$, and $A'_{ij} \geq 0$ for all $i,j \in [n]$ such that $ j \not \in \{\sigma(i), \pi(i)\}$ and $A'_{ij} = 0$ otherwise. 
	For each column $j \in [n]$ this means that the minimum value is $0$, and \eqref{eq:tropical-hyperplane} implies that the columns of $A$ lie on the tropical hyperplane $H_0$.
\end{proof}
	
\begin{example}
	Consider the matrix from \Cref{ex:cartoon-example-cont}.
	We first subtract the apex $c=(1,1,0)$  of the tropical line from every column of the matrix. Then we add $m_j (1,\dots,1)^t$ to every column, where $m_j$ is the minimum entry of the $j$th column. This yields
	\[
		 \ma 
		0 & 0 & 2 \\
		0 & 0 & 1 \\
		3 & 1 & 0 \trix \sim 
		\ma 
		 -1 & -1 &  1 \\
		-1 & -1  & 0 \\
		  3 &  1 & 0 \trix \sim
		\ma
		  0 & 0 &  1 \\
		  0 & 0 & 0 \\
		  4 & 2 & 0 \trix .
	\] 
\end{example}

\begin{lemma}\label{lem:cartoon-geometric-interpretation}
	Let $C$ be a maximal cone of $\tropdet[n-1][n,n]$ and $\conv(\sigma, \pi)$ be the dual edge of the Birkhoff polytope $B_n$.
	Let $A \in C$ and let $H$ be a tropical hyperplane containing the columns of $A$. If the edge $v_{\sigma^{-1}(j)} v_{\pi^{-1}(j)}$ is decorated in the cartoon of $C$, then the $j$-th column $A_j$ of $A$ lies on the wing $W_{\sigma^{-1}(j), \pi^{-1}(j)}$ of $H$. If the node $v_{\sigma^{-1}(j)}$ is decorated in the cartoon, then the column $A_j$ lies on the wing $W_{k,\sigma^{-1}(j)}$ for some $k \in [n]$.
\end{lemma}

\begin{proof}
By \Cref{lem:wlog-matrix-nonneg-mod-linspace} we can assume that $A_{ij} = 0$ for all $i,j \in [n]$ such that $j\in\{\sigma(i), \pi(i)\}$, and $A_{ij}\geq 0$ otherwise, and that $H = H_0$ is the tropical hyperplane with apex at the origin.
Equivalently, $A_{ij} = 0$ if $i \in \{\sigma^{-1}(j), \pi^{-1}(j)\}$.
	In particular $A_{\sigma^{-1}(j)j}=A_{\pi^{-1}(j)j}=0$. 
	The cartoon of $C$ has a decorated edge $v_{\sigma^{-1}(j)} v_{\pi^{-1}(j)}$ if and only if $\sigma^{-1}(j) \neq \pi^{-1}(j)$. If $\sigma^{-1}(j) \neq \pi^{-1}(j)$, then the column $A_j$ is contained in the wing $W = \cone(e_i \mid i \neq \sigma^{-1}(j), i \neq \pi^{-1}(j))$. 
	The cartoon has a decorated node $v_{\sigma^{-1}(j)}$ if and only if $\sigma^{-1}(j)=\pi^{-1}(j)$, and the column $A_j$ may lie on any wing not containing the ray in direction $e_{\sigma^{-1}(j)}$. 
\end{proof}

\begin{const}[Cartoons of matrices]
	Let $C \subseteq \tropdet[n-1][n,n]$ be a maximal cone with dual edge $\conv(\sigma, \pi)$ and $A \in \interior(C)$. Let $H$ be a tropical hyperplane containing the columns of $A$. 
	To obtain the \emph{cartoon of $A$ with respect to $H$}, we decorate the boundary complex of the $(n-1)$-dimensional simplex $\Delta_{n-1}$ with $n$ points placed on faces of $\Delta_{n-1}$. More precisely, for each $j \in [n]$, decorate the face $F$ of $\Delta_n$ with a marking if the column $A_j$ lies in the interior of the cone of $H$ that is dual to the face $F$. 
\end{const}

\begin{lemma}\label{lem:cartoon-sliding}
	Let $C \subseteq \tropdet[n-1][n,n]$ be a maximal cone. Let $A \in C$ and $H$ be a tropical hyperplane such that each column of $A$ lies in the interior of a wing of $H$. Then the cartoon of $A$ with respect to $H$ can be obtained from the cartoon of $C$ by sliding markings from nodes of $C$ to incident edges.
\end{lemma}
\begin{proof}
	By assumption, each column lies in the interior of a wing of $H$, so the cartoon of $A$ with respect to $H$ has only markings on edges of $\Delta_{n-1}$. If the cartoon of $C$ has a marked edge $v_{\sigma^{-1}(j)} v_{\pi^{-1}(j)}$, then \Cref{lem:cartoon-geometric-interpretation} implies that the column $A_j$ lies on the wing $W_{\sigma^{-1}(j), \pi^{-1}(j)}$ of $H$, and so the column $A_j$ marks the same edge in the cartoon of $A$ w.r.t. $H$. 
	If the cartoon of $C$ has a marked node $v_{\sigma^{-1}(j)}$, then \Cref{lem:cartoon-geometric-interpretation} implies that the column $A_j$ lies on some wing $W_{k,\sigma^{-1}(j)}, k \in [n]$ of $H$, and so the column $A_j$ marks the edge with vertices $v_k$ and $v_{\sigma^{-1}(j)}$ in the cartoon of $A$ w.r.t $H$.
\end{proof}

\begin{theorem}[Geometric triangle criterion]\label{th:geometric-triangle-criterion}
	Let $n=3,4$ and $C \subseteq \tropdet[n-1][n,n]$ be a maximal cone. Let $A \in C$ and $H$ be a tropical hyperplane such that each column of $A$ lies in the interior of a wing of $H$. Then the cartoon of $A$ with respect to $H$ has markings only on edges of $\Delta_{n-1}$, and $C$ is positive if and only if the cartoon of $A$ with respect to $H$ does not contain a marked triangle.
\end{theorem}

\begin{proof}
	By \Cref{prop: triangle criterion} (\hyperref[prop: triangle criterion]{Triangle criterion for cartoons}), $C$ is positive if and only if the cartoon of $C$ does not contain a marked triangle, i.e. a triangle with three distinct markings, where each edge contains at least one marking in its interior, or on an incident vertex. \Cref{lem:cartoon-sliding} implies that the cartoon of $A$ w.r.t $H$ can be obtained from the cartoon of $C$ by sliding the markings from nodes to edges. Hence, the set of marked triangles of the cartoon of $A$ w.r.t $H$ is a subset of the marked triangles of the cartoon of $C$. It thus remains to show that if the cartoon of $C$ contains a marked triangle, then so does the cartoon of $A$ w.r.t to $H$. For $n=3$, there is a unique such configuration (\Cref{subfig:cycle3}) and all markings of the cartoon of $C$ are already on edges. For $n=4$, there is also a unique such configuration (\Cref{subfig:cycle43}), and the markings of the marked triangle are on edges. Hence, this is also a marked triangle in the cartoon of $A$ w.r.t $H$. 
\end{proof}

\begin{example}[Geometric triangle criterion fails for $n\geq 5$]\label{ex:geometric-triangle-crit}
	Consider the (positive) cone from \Cref{ex:triangle-crit-counterexample}. The cartoon of cone $C$, which is depicted in \Cref{fig:constructionDiagram}, has a marked triangle. However, sliding the markings from nodes to edges yields the cartoon in \Cref{fig:example-geometric-triangle-crit}, which does not have a marked triangle. This example can be generalized to any $n\geq 5$.
	
	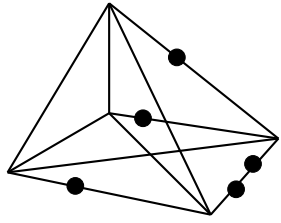
\begin{figure}
		\centering
		\begin{tikzpicture}[scale = 0.45]
			\node (1) at (4, 11) {};
			\node (2) at (1, 6) {};
			\node (3) at (9, 7) {};
			\node (4) at (7, 4.75) {};
			\node (5) at (4, 7.75) {};
			\draw[thick] (2.center) to (4.center);
			\draw[thick] (4.center) to (3.center);
			\draw[thick] (3.center) to (1.center);
			\draw[thick] (1.center) to (2.center);
			\draw[thick] (2.center) to (3.center);
			\draw[thick] (1.center) to (4.center);
			\draw[thick]  (5.center) to (1.center);
			\draw[thick]  (2.center) to (5.center);
			\draw[thick]  (5.center) to (4.center);
			\draw[thick]  (5.center) to (3.center);
			\filldraw (7.75, 5.5) circle (7pt);
			\filldraw (8.25, 6.25) circle (7pt);
			\filldraw (6,9.4) circle (7pt);
			\filldraw (3,5.6) circle (7 pt);
			\filldraw (5, 7.6) circle (7pt);
		\end{tikzpicture}
		\caption{The cartoon from \Cref{ex:geometric-triangle-crit}.}
		\label{fig:example-geometric-triangle-crit}
	\end{figure}
	
\end{example}

\subsection{Extension to all orthants}\label{sec:extension-to-all-orthants}

In this section, we want to exploit the observation made in \Cref{sec:signed-tropical-generators} in order to understand the signed tropicalizations of the variety $\tropdet[n-1][n,n]$ with respect to sign patterns beyond the positive orthant. 
Hence, we fix a sign matrix $\sign \in \{-1,1\}^{(n \times n)}$. 
Recall from \Cref{cor: positive iff different signs of permutations} that a maximal cone of $\tropdet[n-1][n,n] = \trop(V(\det))$ is positive if and only if the permutation $\sigma \pi^{-1}$ for the corresponding edge $\conv(\sigma, \pi)$ is an even cycle. 
Therefore, we can interpret the partition of the maximal cones in positive and non-positive cones as a coloring of the edges of the graph $\BirG_n$ of the Birkhoff polytope $\Birk_n$.
We color the edges dual to positive cones in green ("positive edges"), and the remaining ones in red ("non-positive edges").

The Newton polytopes of $\det$ and $\det^\sign$ agree.
Hence, for each sign pattern $\sign$ we obtain a $2$-coloring of the edges of $\BirG_n$, corresponding to the (non-)positivity of the maximal cones of $ \trop(V(\det^\sign))$.
Then, the green edges correspond to maximal cones of $\trop^\sign(V(\det))$, i.e. the tropicalization of $(n\times n)$-matrices of rank $n-1$ in $\C^\sign$.
We begin by investigating the $2$-coloring for $s = (1)_{ij}$, i.e. the 
coloring given by $\tropcplus(V(\det))$.

\begin{lemma} \label{lemma: sign flips start configuration}
  The $2$-coloring of $\BirG_n$ given by $\tropcplus(V(\det))$ has exactly $2$ connected components formed by red edges. 
  The vertices in one component correspond to the alternating group $A_n \subseteq S_n$, the even permutations of $S_n$.
  The vertices in the other component correspond to the odd permutations $S_n \setminus A_n$.
  Furthermore, the induced subgraphs on $A_n$ and $S_n \setminus A_n$ only have red edges and the green edges are exactly the edges in the cut $(A_n, S_n \setminus A_n)$. 
\end{lemma}
\begin{proof}
  We identify the vertices in $\BirG_n$ with the permutations in $\Sym_n$ so that the edge set is given by the pairs $\SetOf{(\sigma,\pi)}{\sigma \pi^{-1} \text{ is a cycle} }$. 
	Let $\sigma \in A_n,$ and $c \in A_n$ be a $3$-cycle, and consider $\pi = \sigma c$.
	Then $\pi \in A_n$, and so $\pi$ is a neighbor of $\sigma$ in $\BirG_n$.
	The permutations $\sigma$ and $\pi$ have equal sign, so by \Cref{cor: positive iff different signs of permutations} the edge $(\sigma, \pi)$ is colored in red.
	Since the alternating group $A_n$ is generated by $3$-cycles, it follows that all permutations $\pi \in A_n$ are contained in one red connected component.
	All remaining vertices are in $S_n \setminus A_n$. Note that if 
	$\tau$ is a transposition, then
	$S_n \setminus A_n = \tau A_n$, and that all edges inside $\tau A_n$ are red.
	Finally, permutations in $\tau A_n$ have negative sign, so all edges between $A_n$ and $\tau A_n$ are green.
\end{proof}

\begin{prop}\label{lem:sign-flipping}
	Let $\sign \in \{-1,1\}^{(n \times n)}$.
	The graph of $\trop^\sign(V(f))$ has $2$ red connected components, which partition the vertices into $2$ parts.
	Equivalently, the green edges are the edges of a cut.
\end{prop}
\begin{proof}
	By the discussion above, we are interested in the coloring of the graph $\BirG_n$ given by the positive cones of $\tropcplus(V(\det^\sign))$.
	If $s = (1)_{ij}$, then the claim holds by \Cref{lemma: sign flips start configuration}.
	Fix $(k,\ell) \in [n]\times[n]$.
	We show that if the claim holds for a fixed sign pattern $\sign \in \{-1,1\}^{(n \times n)}$, then it also holds for the sign pattern $s'$, where $s'_{k\ell} = - s_{k\ell}$ and and $s_{ij} = s'_{ij}$ for all other entries.
	That is, we show that the property is preserved under flipping the sign of the $(k,\ell)$th entry.
	Let $(A,B)$ be the partition of vertices of the coloring induced by $\det^s$.
	Note that 
	\[
			{\det}^\sign = \sum_{\sigma \in S_n}  \left( \sgn(\sigma) \prod_{i=1}^{n} s_{i \sigma(i)} x_{i\sigma(i)} \right),
	\]
	so an edge $(\sigma, \pi)$ is red if and only if
        \begin{equation*}
          \sgn(\sigma) \prod_{i=1}^{n} s_{i \sigma(i)} = \sgn(\pi) \prod_{i=1}^{n} s_{i \pi(i)}. 
        \end{equation*}
	Flipping the sign at $(k,\ell)$ thus switches the color of all edges $\conv(\sigma, \pi)$ where there exists an $i' \in [n]$ such that $(i',  \sigma(i')) = (k,\ell)$ and $(i, \pi(i)) \neq (k, \ell)$ for all $i \in [n]$ (or if there exists an $i'' \in [n]$ such that $(i'', \pi(i')) = (k,\ell)$ and $(i , \sigma(i)) \neq (k ,\ell)$ for all $i \in [n]$). Equivalently, flipping the sign at $(k,\ell)$ switches the color of all edges
	where $\sigma(k)=\ell$ and $\pi(k)\neq \ell$ (or $\sigma(l)\neq \ell$ and $\pi(k)= \ell$).
	Hence, we partition $A$ into
        $A^= = \{\sigma \in A \mid \sigma(k) = \ell\}, A^{\neq} = A \setminus A^=$ 
        and similarly $B = B^= \sqcup B^{\neq}$.
	We then flip the colors of all edges between $(A^=,A^{\neq})$, $(A^=,B^{\neq}),$ $(B^=,A^{\neq}),$ $(B^=,B^{\neq})$, as shown in \Cref{fig: proof of sign flips}. The resulting graph has red components $A^= \sqcup B^{\neq}$ and $A^{\neq} \sqcup B^=$.
\end{proof}
\begin{figure}
	\centering
	\begin{tikzpicture}
		\node at (0,0) {$B^=$};
		\node at (0,2) {$A^=$};
		\node at (2,0) {$B^{\neq}$};
		\node at (2,2) {$A^{\neq}$};
		\draw[color=red, very thick] (0,0) circle (1em);
		\draw[color=red, very thick] (2,0) circle (1em);
		\draw[color=red, very thick] (0,2) circle (1em);
		\draw[color=red, very thick] (2,2) circle (1em);
		\draw[very thick, color=red] (0.5,0) -- (1.5,0);
		\draw[very thick, color=red] (0.5,2) -- (1.5,2);
		\draw[very thick, color = DarkGreen, dashed] (0,0.5) -- (0,1.5);
		\draw[very thick, color = DarkGreen, dashed] (2,0.5) -- (2,1.5);
		\draw[very thick, color = DarkGreen, dashed] (0.35, 0.35) -- (1.65,1.65); 
		\draw[very thick, color = DarkGreen, dashed] (1.65, 0.35) -- (0.35,1.65); 
	\end{tikzpicture}
	\hspace{2cm}
	\begin{tikzpicture}
		\node at (0,0) {$B^=$};
		\node at (0,2) {$A^=$};
		\node at (2,0) {$B^{\neq}$};
		\node at (2,2) {$A^{\neq}$};
		\draw[color=red, very thick] (0,0) circle (1em);
		\draw[color=red, very thick] (2,0) circle (1em);
		\draw[color=red, very thick] (0,2) circle (1em);
		\draw[color=red, very thick] (2,2) circle (1em);
		\draw[very thick, color = DarkGreen, dashed] (0.5,0) -- (1.5,0);
		\draw[very thick, color = DarkGreen, dashed] (0.5,2) -- (1.5,2);
		\draw[very thick, color = DarkGreen, dashed] (0,0.5) -- (0,1.5);
		\draw[very thick, color = DarkGreen, dashed] (2,0.5) -- (2,1.5);
		\draw[very thick, color=red] (0.35, 0.35) -- (1.65,1.65); 
		\draw[very thick, color=red] (1.65, 0.35) -- (0.35,1.65); 
	\end{tikzpicture}
\caption{
	Auxiliary graphs for the $2$-coloring for sign patterns $s$ (left) and $s'$ (right) after the sign flip of a single entry.
	Within the sets $A^=, A^{\neq}, B^=, B^{\neq} $ all edges are red.
	The color of the edge between two parts in the auxiliary graph represents the color of all edges in $\BirG_n$ between the parts.
	Green edges are dashed.}
\label{fig: proof of sign flips}
\end{figure}
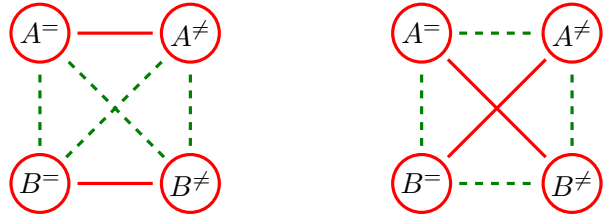

The above statement implies that the elements in the set $\{\trop^\sign(f) \mid s \in \{-1,1\}^{(n \times n)}\}$ correspond to certain cuts in the graph $\BirG_n$.

\begin{question}
	Is there a group theoretical interpretation of the $2^{n^2}$ partitions given by the cuts for every sign pattern? 
\end{question}

\section{Determinantal prevarieties and bipartite graphs}
\label{sec:labels}

Let $\det$ be the polynomial representing the determinant of a $(n\times n)$-matrix. Recall from \Cref{sec:Birkhoff-polytope} that $\trop(V(\det))$ is the codimension-$1$ skeleton of the normal fan of the Birkhoff polytope $B_n$. In \cite{paffenholz15_facesbirkhoffpolytopes} the faces of $B_n$ are identified with \emph{face graphs}, which are unions of perfect matchings on the bipartite graph on vertices $[n]\sqcup [n]$.

\begin{const}[Face graphs \cite{paffenholz15_facesbirkhoffpolytopes}]
	Let $C \subseteq \trop(V(\det))=\tropdet[n-1][n,n]$ be a cone in the tropical hypersurface, and let $\Lambda \subseteq S_{n}$ such that $\conv(\Lambda)$ is the face of $B_{n}$ dual to $C$. We associate the bipartite graph $\G$ on vertices $V(\G) = R\sqcup G$, $R = \{r_1, \dots, r_{n} \}, G = \{g_1, \dots, g_{n}\}$ and edges
	\[
	E(\G) = \{r_i g_j \mid \sigma(i) = j \text{ for some } \sigma \in \Lambda\}.
	\]
\end{const}

This extends to a labeling of the entire normal fan of $B_{n}$, where the label of the normal cone of a vertex $\sigma$ is a perfect matching with edges $(r_i, g_{\sigma(i)}), i \in [n]$. The label of a cone dual to a face $F$ is the union of all labels of normal cones of vertices contained in $F$. Thus, such a label is a union of perfect matchings.

\begin{prop}[Triangle criterion for bipartite graphs]\label{prop:label-triangle-crit}
	Let $C \subseteq \tropdet[n-1][n,n]$ be a maximal cone. 
	Then $\G$ consists of a cycle of length $2l$, and a perfect matching of the remaining $2(n-l)$ vertices. $C$ positive if and only if $l$ is even.
\end{prop}
\begin{proof}
	Let $C$ be a maximal cone and $\conv(\sigma, \pi)$ be the edge of $B_n$ dual to $C$. 
	The bipartite graph $\G$ is a union of $2$ perfect matchings, corresponding to $\sigma$ and $\pi$. Since these permutations form an edge on $B_n$, we have that $\sigma \pi^{-1}$ is a cycle of length $l$, i.e. there are elements $r_{i_1},\dots r_{i_l} \in [n]$ such that $\sigma \pi^{-1} (i_k) = i_{k+1}$ (and $i_{l+1} = i_1$). Equivalently, $\sigma(i_k) = \pi(i_{k+1})$ and $\sigma(i_{k-1}) = \pi(i_k)$. For all other elements $j \in [n]$ holds $\sigma(j)=\pi(j)$.
	Thus, $\G$ consists of isolated edges $(i,\sigma(i))$ (forming a perfect matching) and a cycle $(r_{i_1}, g_{\sigma(i_1)}), (g_{\pi(i_2)}, r_{i_2}), (r_{i_2}, g_{\sigma(i_2)}), \dots, (g_{\pi(i_l)}, r_{i_l}), (r_{i_l}, g_{\sigma(i_l)}), (g_{\pi(i_{l+1})}, r_{i_1})$. Therefore, $\G$ consists of a cycle of length $2l$ and isolated edges. By \Cref{cor: positive iff different signs of permutations}, the cone $C$ is positive if and only if $\sgn(\sigma \pi^{-1}) = -1$, and equivalently the length $l$ of the cycle $\sigma \pi^{-1}$ is even.
\end{proof}

We extend the idea of face graphs as labels of cones of $\tropdet[n-1][n,n]= \prevar[n-1][n,n]$ by embedding these face graphs, for each $I$ and $J$, in a bipartite graph $\G$ on vertices $[d]\sqcup [n]$. This yields a label $\G$ of cones in the tropical determinantal prevariety $\prevar$.

\begin{definition}
	Let $I = \{i_1, \dots, i_{r+1}\} \in \binom{[d]}{r+1}, \ J = \{j_1, \dots, j_{r+1}\} \in \binom{[n]}{r+1}$ where $i_k < i_{k+1}, j_{k} < j_{k+1}$, and let $\sigma \in S_{r+1}$ be a permutation $\sigma: [r+1] \to [r+1]$. In the following, sets $I$ and $J$ are always of this form. We define the \emph{embedded permutation} to be the map
	\begin{align*}
		\sigma^{IJ}: I &\longrightarrow J \\
		i_k &\longmapsto j_{\sigma(k)}.
	\end{align*}
	The \emph{embedded Birkhoff polytope} $B_{r+1}^{IJ} \subseteq \RR^{d \times n}$ is the convex hull of the permutation matrices of the embedded permutations $\sigma^{IJ}, \sigma \in S_{r+1}$, where in this embedding, for each $ij \not \in I\times J$ we set the $ij$-th entry of each matrix in $B_{r+1}^{IJ}$ to zero.  
\end{definition}

Recall from \Cref{sec:determantal-varieties} that $\tropdet \subseteq \prevar = \bigcap_{f \in I_r} \trop(V(f))$, where $f \in I_r$ ranges over all $(r+1)\times(r+1)$-minors of a $(d\times n)$-matrix. More precisely, the ideal $I_r$ is generated by polynomials
\[
f^{IJ} = \sum_{\sigma \in S_{r+1}} \sgn(\sigma) \prod_{k=1}^{r+1} x_{i_k j_{\sigma(k)}} =  \sum_{\sigma \in S_{r+1}} \sgn(\sigma) \prod_{k=1}^{r+1} x_{i_k \sigma^{IJ}(i_k)}.
\]
Thus, a cone $C^{IJ} \subseteq \trop(V(f^{IJ}))$ can be seen as cone in the normal fan of $\Bembedded$.

\begin{const}[Labels of cones in $\prevar$]\label{const:bipartite-labels}
	Let $C \subseteq \prevar$ be a cone in the tropical determinantal prevariety. Then for each $I,J$ there exists a unique inclusion-minimal cone $C^{IJ} \in \trop(V(f^{IJ}))$ such that 
	$
	C = \bigcap_{I,J} C^{IJ}.
	$
	Let $\Lambda(I,J) \subseteq \{\sigma^{IJ} \mid \sigma \in S_{r+1}\}$ such that $\conv(\Lambda(I,J))$ is the face of $\Bembedded$ dual to $C^{IJ}$. Let $R = \{r_1, \dots, r_d\}$ and $G = \{g_1, \dots, g_n\}$. $R$ corresponds to row indices of matrices in $\prevar$, and $G$ corresponds to column indices. To $C$ we associate the bipartite graph $\G[C]$ on vertices $V(\G[C]) = R \sqcup G$ and edges
	\begin{align*}
		E(\G[C]) = \bigcup_{I,J} \left\{r_{i_k} g_{j_l} \mid \sigma^{IJ}(i_k) = j_l \text{ for some } \sigma^{IJ} \in  \Lambda(I,J) , l,k \in [r+1] \right\} .
	\end{align*}
\end{const}

\begin{definition}
	Let $\Gamma$ be a bipartite graph on vertices $V(\Gamma)= R \sqcup G$. The \emph{bipartite complement} $\Gamma^c$ is the bipartite graph on vertices $V(\Gamma)=V(\Gamma^c)$ and edges
	\[
	E(\Gamma^c) = \{r_i g_j \mid r_i \in R, g_j \in G, r_i g_j \not \in E(\Gamma)\}.
	\]
\end{definition}

\begin{example}[Label of a cone in {$\tropdet[2][3,4]$}]\label{ex:label}
	Let $d = 3, n=4$ and $r = 2$. Consider the cone $C \subseteq \prevar[2][3,4]= \tropdet[2][3,4]$
	with rays
	\[
	C = \cone(E_{11}, E_{22}, E_{33}, E_{34}).
	\]
	Then $C = \bigcap_{I,J} C^{IJ}$, where
	$I = [3]$ and for 
	$J = \{1,2,3\}, \{1,24\}$ the cone $C^{IJ}$ is dual to the edge $\conv((1,2,3)^{IJ}, (1,3,2)^{IJ})$ of $B_{3}^{IJ}$, while for $J = \{1,3,4\}, \{2,3,4\}$ the cone $C^{IJ}$ is dual to the edge $\conv((1,2,3)^{IJ}, (1,3)^{IJ})$. Thus, the label $\G$ is the bipartite complement of the graph with edges $r_1 g_1, r_2 g_2, r_3 g_3$ and $r_3 g_4$, as shown in \Cref{fig:example-label}.
	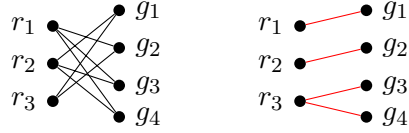
\begin{figure}
		\centering
		\begin{tikzpicture}
			[
			vertex/.style={circle, fill, minimum size=1em, scale=0.4},
			]
			\node[vertex,label=left:{$r_1$}] (r1) at (0,1) {};
			\node[vertex,label=left:{$r_2$}] (r2) at (0,.5) {};
			\node[vertex,label=left:{$r_3$}] (r3) at (0,0) {};
			\node[vertex, above right = 0.25em and 2em of r1, label=right:{$g_1$}] (c1) {};
			\node[vertex, above right = 0.25em and 2em of r2, label=right:{$g_2$}] (c2) {};
			\node[vertex,  above right = 0.25em and 2em of r3, label=right:{$g_3$}] (c3) {};
			\node[vertex,  below right = 0.25em and 2em of r3, label=right:{$g_4$}] (c4) {};
			\draw (r1) -- (c2) -- (r3) -- (c1) -- (r2) -- (c3) -- (r1);
			\draw (r1) -- (c4) -- (r2);
		\end{tikzpicture}
		\hspace*{2em}
		\begin{tikzpicture}
			[
			vertex/.style={circle, fill, minimum size=1em, scale=0.4},
			]
			\node[vertex,label=left:{$r_1$}] (r1) at (0,1) {};
			\node[vertex,label=left:{$r_2$}] (r2) at (0,.5) {};
			\node[vertex,label=left:{$r_3$}] (r3) at (0,0) {};
			\node[vertex, above right = 0.25em and 2em of r1, label=right:{$g_1$}] (c1) {};
			\node[vertex, above right = 0.25em and 2em of r2, label=right:{$g_2$}] (c2) {};
			\node[vertex,  above right = 0.25em and 2em of r3, label=right:{$g_3$}] (c3) {};
			\node[vertex,  below right = 0.25em and 2em of r3, label=right:{$g_4$}] (c4) {};
			\draw[color=red] (c1) -- (r1);
			\draw[color=red] (c3) -- (r3) -- (c4);
			\draw[color=red] (c2) -- (r2);
		\end{tikzpicture}
		\caption{The label $\G$ of the cone in \Cref{ex:label} (left) and the bipartite complement $\Gc$ (right).}
		\label{fig:example-label}
	\end{figure}
		\end{example}
	
	\begin{theorem}\label{prop:labels-positive-subgraphs}
		If $C \subseteq \prevarplus[r]$, then each induced subgraph on vertices $I \subseteq \binom{[d]}{r+1}, J \subseteq \binom{[n]}{r+1}$ contains a subgraph consisting of a cycle of length $2l$, where $l=l(I,J) \in \N$ and a perfect matching of the remaining $2(r+1-l)$ vertices. If $C$ is positive, then for each $I,J$ the length of the cycle $l=l(I,J)$ is even.
	\end{theorem}
	\begin{proof}
		Let $C$ be a cone. 
		Then there exist unique inclusion-minimal cones $C^{IJ}\subseteq \tropcplus(V(f^{IJ}))$ such that $C = \bigcap_{I,J} C^{IJ}$, and $\G$ is the union of the labels $\G[C^{IJ}]$. If $C$ is positive, then so is $C^{IJ}$ for each $I,J$. By \Cref{prop:label-triangle-crit}, every subgraph $H$ on vertices $V(H) = I\sqcup J \subseteq R \sqcup G$, $|I| = |J| = r+1$ contains a subgraph consisting of a cycle of length $2l$, and a perfect matching of the remaining vertices. If $C$ is positive, then $l = l(I,J)$ is even by \Cref{prop:label-triangle-crit}.
	\end{proof}
	
	We note that the converse of the statement above is not true. In fact, for most cones $C$ that are not maximal (including non-positive cones), the label $\G$ is the complete bipartite graph $K_{n,d}$.
	We close this section with a property of the label $\G$.

	\begin{prop}\label{prop:labels-matchings-degrees}
		Let $C \subseteq \prevar$ be a cone and $\G$ the label on vertices $V(\G) = R \sqcup G$.
		Each vertex $r \in R$ has degree at least $n - r$, and every vertex $g \in G$ has degree at least $d -r$. 
	\end{prop}
	\begin{proof}
		By \Cref{prop:labels-positive-subgraphs}, each subgraph of $\G$ of size $(r+1)+(r+1)$ contains a union of perfect matchings.
		Let $v_1 \in R$ and assume for contradiction that $\deg(r)\leq |R| -(r +1)$. Then there are nodes $g_1, \dots, g_{r+1}$ that are not adjacent to $v_1$. Hence, for any $v_2,\dots,v_{r+1} \in R$, the vertex $r_1$ is isolated in the induced subgraph $H$ on vertices $\{v_1,\dots,v_{r+1}\}\sqcup \{g_1,\dots g_{r+1}\}$. However, by \Cref{prop:labels-positive-subgraphs}, the graph $H$ does not contain an isolated vertex, which yields the desired contradiction. An analogous argument implies that $\deg(g)\geq |G| - r$ for all $g \in G$.
	\end{proof}
	
	We illustrate the difference of between the applicability of the triangle criteria for cartoons (\Cref{prop: triangle criterion}) and for bipartite graphs (\Cref{prop:label-triangle-crit}). Indeed, for maximal cones of $\tropdet[n-1][n,n]$ the description via cartoons and bipartite graphs are equivalent, as the proof of \Cref{prop:label-triangle-crit} suggests. 
	For arbitrary choices of $d$ and $n$, there is single bipartite graph describing a cone $C \subseteq \prevar$ as given in \Cref{const:bipartite-labels}.
	We can describe $C$ by cartoons as follows:  for each $I \in \binom{[d]}{r+1}, J \in \binom{[n]}{r+1}$, detect all maximal cones $C^{IJ} \subseteq \trop(V(f^{IJ}))$ such that $C \subseteq C^{IJ}$ and consider their cartoons.
	This describes the cone $C$ by a collection of at least $\binom{d}{r+1}\binom{n}{r+1}$ cartoons. Each of the cartoons can be obtained from the graph $\G[C^{IJ}]$ and $\G$ is the union over all these graphs. Therefore, the label $\G$ contains strictly less information than the collection of cartoons. Still, \Cref{prop:labels-positive-subgraphs} gives a criterion to detect (combinatorial) non-positivity.

	\begin{example}[Detecting non-positivity from {$\G$}]\label{ex:difference-cartoons-labels}
		Let $r = 2, d = 4, n = 3$, and consider the matrix
		\[
			A = \ma
				k_1 & 0 & 0 \\
				0 & k_2 & 0 \\
				0 & 0 & 1+k_3 \\
				0 & 0 & 1 
				\trix \in \tropdet[2][4,3], \ k_1, k_2, k_3 >0.
		\]
		Let $C$ be the maximal cone $C\subseteq \tropdet[2][4,3]$ containing $A \in \interior(C)$, and let $J = [3]$. For each $I \in \binom{4}{3}$ there is a unique maximal cone $C^{IJ}$ containing $C$. Their cartoons are displayed in \Cref{fig:example-difference-cartoons-labels} (left).
		The cone $C$ is positive if and only if for each $I$ (and $J$) there exists a positive cone $C^{IJ} \supseteq C$. $C$ is not positive, which can be seen from the cartoons in \Cref{fig:example-difference-cartoons-labels} by the \hyperref[prop: triangle criterion]{Triangle criterion for cartoons} (\Cref{prop: triangle criterion}). 
		The label $\G$ can be seen in \Cref{fig:example-difference-cartoons-labels} (right). The subgraph $H$ on vertices $\{r_1, r_2, r_3\}\sqcup \{g_1, g_2, g_3\}$ does not contain a cycle of length $2l$, $l$ even. Hence, \Cref{prop:labels-positive-subgraphs} also implies that $C$ is not positive. For $r=2$, we present a full characterization of labels of maximal cones in terms of positivity in \Cref{th:positive-labels}. 
	
	\begin{figure}
		\centering
			\begin{tikzpicture}[scale=0.8]
				\filldraw (1,1.8) circle (3pt);
				\filldraw (0.7,0) circle (3pt);
				\filldraw (1.3,0) circle (3pt);
				\draw[thick] (0,0) -- (2,0) -- (1,1.8) -- (0,0);
				\node at (1,-1) {$I = \{1,2,3\}$}; 
			\end{tikzpicture}
			\begin{tikzpicture}[scale=0.8]
				\filldraw (1,1.8) circle (3pt);
				\filldraw (0.7,0) circle (3pt);
				\filldraw (1.3,0) circle (3pt);
				\draw[thick] (0,0) -- (2,0) -- (1,1.8) -- (0,0);
				\node at (1,-1) {$I = \{1,2,4\}$}; 
			\end{tikzpicture}
			\begin{tikzpicture}[scale = 0.8]
				\filldraw (1,0) circle (3pt);
				\filldraw (0.5,0.9) circle (3pt);
				\filldraw (1.5,0.9) circle (3pt);
				\draw[thick] (0,0) -- (2,0) -- (1,1.8) -- (0,0);
				\node at (1,-1) {$I = \{1,3,4\}$}; 
			\end{tikzpicture} 
			\begin{tikzpicture}[scale = 0.8]
				\filldraw (1,0) circle (3pt);
				\filldraw (0.5,0.9) circle (3pt);
				\filldraw (1.5,0.9) circle (3pt);
				\draw[thick] (0,0) -- (2,0) -- (1,1.8) -- (0,0);
				\node at (1,-1) {$I = \{2,3,4\}$}; 
			\end{tikzpicture} \hspace*{2em}
		\begin{tikzpicture}
			[
			vertex/.style={circle, fill, minimum size=1em, scale=0.4},
			scale = 2
			]
			\node[vertex,label=left:{$r_1$}] (r1) at (0,1) {};
			\node[vertex,label=left:{$r_2$}] (r2) at (0,.5) {};
			\node[vertex,label=left:{$r_3$}] (r3) at (0,0) {};
			\node[vertex,label=left:{$r_4$}] (r4) at (0,-0.5) {};
			\node[vertex, above right = 0.5 and 2em of r2, label=right:{$g_1$}] (g1) {};
			\node[vertex,  above right = 0.5 and 2em of r3, label=right:{$g_2$}] (g2) {};
			\node[vertex,  below right = 0.5 and 2em of r3, label=right:{$g_3$}] (g3) {};
			\draw (r1) -- (g2) -- (r3) -- (g1) -- (r2) -- (g3) -- (r1); 
			\draw[dashed] (g1) -- (r4) -- (g2);
	\end{tikzpicture}
	
		\caption{The cartoons of the maximal cones $C^{IJ}$ in \Cref{ex:difference-cartoons-labels} (left) and the label $\G$ (right). The edges outside the subgraph $H$ are dashed.}
		\label{fig:example-difference-cartoons-labels}
	\end{figure}
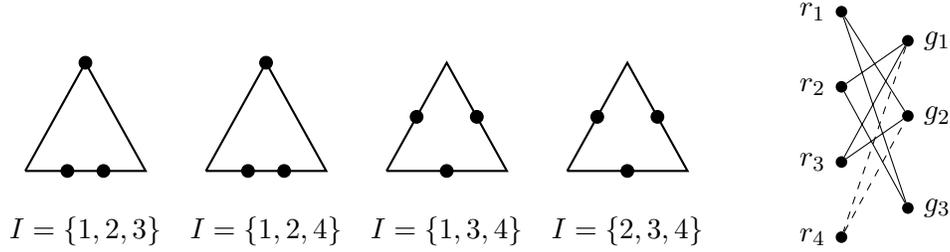
\end{example}

\section{Rank 2}\label{sec:rank-2}

In this section, we consider the tropicalization of the matrices of rank at most $2$, that means the tropical determinantal variety $\tropdet[2]$ of tropical matrices of (Kapranov) rank at most $2$.
It was shown in $\cite{develin_ranktropicalmatrix}$ that the notions of tropical rank and Kapranov rank agree for rank $2$. 

\subsection{Positivity and Barvinok rank}\label{sec:rank-2-barvinok}

Ardila showed in \cite{ardila_tropicalmorphismrelated} that a tropical matrix of tropical rank $2$ is positive if and only if it has Barvinok rank $2$. The proof reveals a crucial connection between the positivity of tropical matrices and the nonnegative rank of matrices with ordinary rank $2$.
We begin by reviewing different characterizations of the Barvinok rank. 

\begin{prop}\label{prop:barvinok-rank}
	For a tropical matrix $A \in \RR^{d \times n}$, the following are equivalent:
	\begin{enumerate}[(i)]
		\item $A$ has Barvinok rank at most $r$.\label{eq:prop-barvinok-rank-1}
		\item The columns of $A$ lie in the tropical convex hull of $r$ points in $\TP^{d-1}$.
		\item There are matrices $X \in \RR^{d \times r}, Y \in \RR^{r \times d}$ such that $A = X \odot Y$.\label{eq:prop-barvinok-rank-3}
	\end{enumerate}
\end{prop}

Here, $X \odot Y$ denotes the tropical matrix multiplication, i.e. $(X \odot Y)_{ij} = \bigoplus_{k = 1}^r X_{ik}\odot Y_{kj} = \min\{X_{ik}+ Y_{kj} \mid k \in [r]\}$.
The equivalence of \eqref{eq:prop-barvinok-rank-1} and \eqref{eq:prop-barvinok-rank-3} leads to the argument in \cite{ardila_tropicalmorphismrelated}, which we give for completeness.

\begin{theorem}[{\cite{ardila_tropicalmorphismrelated}}]\label{th:positive-rank-2-barvinok}
	The positive part of the tropical determinantal variety $\tropdet[2]$ coincides with the set of matrices of Barvinok rank $2$.
\end{theorem}
\begin{proof}
  Consider the map $f \colon \RR^{d \times 2} \times \RR^{2 \times n} \to \RR^{d \times n}, \ (X,Y) \mapsto XY$.
  The image of this map is the determinantal variety $V(I_2) \subseteq \RR^{d \times n}$, i.e. the set of matrices of rank at most $2$.
  We can write $f$ as a polynomial map 
  \begin{equation*}
  f = (f_{11},\dots,f_{dn}) \colon \RR^{2d+2n} \to \RR^{dn} \quad \text{ where } \quad f_{ij}(X,Y) = X_{i1}Y_{1j} + X_{i2}Y_{2j} \enspace .   
  \end{equation*}
  Each $f_{ij}$ has only positive coefficients, i.e. $f$ is \emph{positive}.
  By replacing $+$ with $\min$ and $\cdot$ with $+$ in the definition of $f$, we obtain its tropicalization 
  \begin{equation*}
  g \colon \RR^{d \times 2} \times \RR^{2 \times n} \to \RR^{d \times n}, \ (X,Y) \mapsto X\odot Y \enspace .
  \end{equation*}
  It follows from \cite[Theorem 2]{pachter04_tropicalgeometrystatistical} that since $f$ is positive, we have $\text{Im}(g) \subseteq \tropplus(V(I_2)) = \tropdet[2]$.
  Furthermore, if $f(\RR^{dn}_{> 0}) = \text{Im}(f)\cap \RR^{d\times n}_{>0}$, then $\text{Im}(g) = \tropdet[2]$.
  Indeed, this holds since every positive $(d\times n)$-matrix of rank $2$ can be written as the product of a positive $(d\times 2)$-matrix and a positive $(2 \times n)$-matrix \cite[Theorem 4.1]{cohen93_nonnegativeranksdecompositions}. 
  Finally, note that $\text{Im}(g)$ is precisely the set of matrices of Barvinok rank $2$ by \Cref{prop:barvinok-rank}.
\end{proof}

Consider the columns of $A\in\tropdet[2]$ as the coordinates of $n$ points in $\TP^{d-1}$. \Cref{prop:kapranov-rank-points-on-hyperplane} implies that these are $n$ points lying on a common tropical line $L$ in $\TP^{d-1}$. A tropical line in $\TP^{d-1}$ is a pure, connected $1$-dimensional polyhedral complex not containing any cycles. This complex consists of $d$ unbounded rays in direction of the standard basis $e_1, e_2,\dots, e_{d-1}, e_d \cong - (e_1 + \dots + e_{d-1})$. It has $k\leq d-3$ vertices, which are connected by $k-1$ bounded edges. It was shown in \cite{speyer_tropicalgrassmannian} that tropical lines are in bijection with phylogenetic trees on $d$ leaves, and the space of tropical lines in $\TP^{d-1}$ is the tropical Grassmannian $\Gr[2,d]$. We describe the tropical Grassmannian in more detail in \Cref{sec:tropical-grassmannian}.
On the other hand, the tropical convex hull of the $n$ columns of $A$ is a $1$-dimensional polyhedral complex that only consists of bounded line segments. This complex has two different kinds of vertices, called \emph{tropical vertices} (which is a subset of the $n$ columns of $A$) and \emph{pseudovertices}. As a set, the tropical convex hull is strictly contained in the tropical line $L$. A detailed exposition on tropical convex hulls can be found e.g. in \cite{develin04_tropicalconvexity} and \cite[Chapter 6]{joswig21_essentialstropicalcombinatorics}.\\

\Cref{prop:barvinok-rank,th:positive-rank-2-barvinok} together characterize the possible ``positive'' point configurations of $n$ points on a tropical line:
Such a tropical point configuration is positive if and only if its tropical convex hull has (at most) $2$ tropical vertices. This means that the columns lie on a tropical line segment, which is a concatenation of classical line segments \cite[Proposition 3]{develin04_tropicalconvexity}.
Based on this connection with the Barvinok rank, we obtain a stronger result about the representation by $(3 \times 3)$-minors. 

\begin{theorem}\label{th:rank-2-positive-generators}
	The $(3\times 3)$-minors form a set of positive-tropical generators for $\tropdetplus[2]$.
\end{theorem}

\begin{proof}
	By \Cref{th:positive-rank-2-barvinok} and \eqref{eq:positive-part-inclusion} (in \Cref{sec:notions-of-positivity}), we have that
	\begin{equation*}
		\{A \in \T^{d \times n} \mid A \text{ has Barvinok rank } \leq 2\} = \tropdetplus[2] \subseteq \bigcap_{\substack{f \text{ is a } \\ (3\times3)\text{-minor}}} \tropcplus(V(f)).
	\end{equation*}

	It thus remains to show the reverse inclusion. 
	Let $A \in \tropcplus(V(f))$ for every $(3 \times 3)$-minor $f$.
	Let $I = \{i_1, i_2, i_3\} \in [d],\  J = \{j_1, j_2, j_3\} \in [d]$ and
	\[
		f^{IJ}(x_{i_1j_1},x_{i_1j_2},x_{i_1j_3},x_{i_2j_1},x_{i_2j_2},x_{i_2j_3},x_{i_3j_1},x_{i_3j_2},x_{i_3j_3}) = \sum_{\sigma \in S_3} \sgn(\sigma) \prod_{k=1}^3 x_{i_k j_{\sigma(k)}}.
 	\]
	Then $A$ is a $(d \times n)$-matrix such that $A = \val(\tilde{A}^{IJ})$ for some $\tilde{A}^{IJ} \in (\Cplus)^{d \times n}$ and 
	\[	f^{IJ}(\tilde{A}_{11},\tilde{A}_{12},\tilde{A}_{13},\tilde{A}_{21},\tilde{A}_{22},\tilde{A}_{23},\tilde{A}_{31},\tilde{A}_{32},\tilde{A}_{33}) = 0.
	\]
	Recall that $A \in \prevarplus[2] \subseteq \prevar[2] = \tropdet[2]$ by \eqref{eq:positive-tropical-determinantal-prevariety} (in \Cref{sec:determantal-varieties}) and \Cref{thm:shitov}, and so $A$ (and each submatrix) has Kapranov rank $\leq 2$. Hence, the columns of $A$ lie on a tropical line $L$, and the convex hull of its columns is a $1$-dimensional polyhedral complex supported by $L$.
	We want to show that $A$ has Barvinok rank $\leq 2$, i.e. that the tropical convex hull of the columns of $A$ has at most $2$ tropical vertices.
	Let $M$ be the $(3 \times n)$-submatrix of $A$ with rows $i_1, i_2, i_3 \in [d]$.
	We view the columns of $M$ as $n$ points in $\TP^2$ and consider the tropical convex hull of these points as polyhedral complex of ordinary line segments. 
	First, we show that the tropical convex hull of the $n$ columns of $M$ does not contain a pseudovertex that is incident to more than $2$ edges. Assume for contradiction that there is such a pseudovertex $p$ incident to $3$ line segments $l_1, l_2, l_3$. Then there must be $3$ columns $j_1, j_2, j_3$ of $M$ whose tropical convex hull contains $p$ and the three line segments $l_1, l_2, l_3$. Consider the $(3\times 3)$-submatrix $N$ with rows $I = \{i_1, i_2, i_3\}$ and columns $J = \{j_1, j_2, j_3 \}$. 
	Note that $N$ is the valuation of the submatrix $\tilde{N} \in (\Cplus)^{3\times 3}$ of the matrix $\tilde{A}^{IJ}$.
	By assumption, the matrix $\tilde{N}$ is positive and has rank $\leq 2$. Thus, $N$ has Kapranov rank $\leq 2$, i.e. $N \in \tropdetplus[2][3,3]$.
	Therefore,
	 \Cref{th:geometric-triangle-criterion} (\hyperref[th:geometric-triangle-criterion]{Geometric triangle criterion})
	implies that the tropical convex hull of columns of $N$ cannot contain a pseudovertex incident to $3$ line segments. Hence, $M$ does not contain such a subconfiguration.
	We have thus shown that $A$ does not contain a $(3 \times n)$-submatrix where the convex hull of the columns contain a pseudovertex that is incident to (at least) $3$ line segments. Note that the tropical convex hull of the columns of a matrix equals the tropical convex hull of its rows.
	We can thus apply the same argument to $A^t$ to obtain that $A$ also does not contain a $(3 \times d)$-submatrix with this property either.
	
	We now show the same statement for the matrix $A$. Assume for contradiction that the convex hull of the $n$ columns in $\TP^{d-1}$ contains a pseudovertex $p$ of that is incident to line segments $l_1, l_2, l_3$. Then again there must be $3$ columns $j_1, j_2, j_3$ of $A$ such that their tropical convex hull contains $p$ and $l_1,l_2, l_3$. However, these columns form a $(3\times d)$-submatrix of $A$, which yields a contradiction to the argument above. 
\end{proof}

\subsection{Bicolored phylogenetic trees}\label{sec:bicolored-phylogenetic-trees}

It was shown in \cite{develin_modulispacen,markwig_spacetropicallycollinear} that $\tropdet[2]$ is a shellable complex of dimension $d+n-4$. 
Furthermore, it admits a triangulation by the space of \emph{bicolored phylogenetic trees} $\Phspace$. 
That is, $\Phspace$ is a simplicial fan, whose maximal cones are in correspondence with the combinatorial types of bicolored phylogenetic trees.
The identification of matrices in $\tropdet[2]$ with bicolored trees is as follows:

\begin{const}[Bicolored phylogenetic trees \cite{develin_modulispacen,markwig_spacetropicallycollinear}]\label{construction:bicolored-phylogenetic-trees} 
Let $A_1,\dots,A_n$ be tropically collinear points in $\TP^{d-1}$ and $L$ be a tropical line through these points.
The tropical convex hull $\tconv(A_1,\dots,A_n)$ is a connected $1$-dimensional polyhedral complex supported on a subset of $L$. 
First, for each $j \in [n]$ attach a green leaf with label $j$ at the point $A_j$. 
Note that the line $L$ has an unbounded ray in each coordinate direction $e_1,\dots,e_d$. 
Shorten such an unbounded ray in direction $e_i$ to obtain a red leaf with label $i$.
This procedure results in a tree on $d$ red and $n$ green leaves.
We refer to these color classes as $R$ (for ``red'' or ``rows of $A$'') and $G$ (for ``green'', corresponds to columns of $A$). An example of this construction is shown in \Cref{fig:bicolored-trees-construction}.
\end{const}

		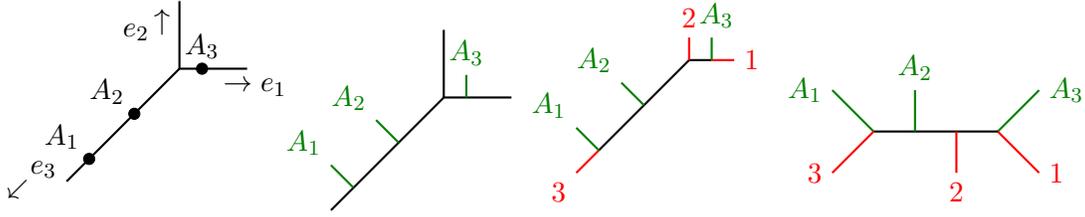
\begin{figure}
		\centering
		\begin{tikzpicture}[scale=0.3]
			\filldraw (-3,-3) circle (7pt);
			\filldraw (-1,-1) circle (7pt);
			\filldraw (2,1) circle (7pt);
			\draw[thick] (-4,-4) -- (1,1) -- (4,1);
			\draw[thick] (1,1) -- (1,4);
			\node[anchor = south east] at (-3,-3) {$A_1$};
			\node[anchor = south east] at (-1,-1) {$A_2$};
			\node[anchor = south] at (2,1) {$A_3$};
			\node[anchor = north west] at (2.5,1) {$\rightarrow e_1$};
			\node[anchor = east] at (1,3) {$e_2$ \rotatebox{90}{$\rightarrow$}};
			\node[anchor = east] at (-4,-4) {\rotatebox{225}{$\rightarrow$}$e_3$};
		\end{tikzpicture}\hspace*{-1em}
		\begin{tikzpicture}[scale=0.3]
			\draw[thick] (-4,-4) -- (1,1) -- (4,1);
			\draw[thick] (1,1) -- (1,4);
			\draw[thick, DarkGreen] (-3,-3) -- (-4,-2);
			\draw[thick, DarkGreen] (-1,-1) -- (-2,0);
			\draw[thick, DarkGreen] (2,1) -- (2,2);
	 		\node[anchor = south east, DarkGreen] at (-4,-2) {$A_1$};
			\node[anchor = south east, DarkGreen] at (-2,-0) {$A_2$};
			\node[anchor = south, DarkGreen] at (2,2) {$A_3$};
		\end{tikzpicture}
		\begin{tikzpicture}[scale=0.3]
			\draw[thick] (-3,-3) -- (1,1) -- (2,1);
			\draw[thick, red] (1,1) -- (1,2);
			\draw[thick, red] (-3,-3) -- (-4,-4);
			\draw[thick, red] (2,1) -- (3,1);
			\draw[thick, DarkGreen] (-3,-3) -- (-4,-2);
			\draw[thick, DarkGreen] (-1,-1) -- (-2,0);
			\draw[thick, DarkGreen] (2,1) -- (2,2);
			\node[anchor = south east, DarkGreen] at (-4,-2) {$A_1$};
			\node[anchor = south east, DarkGreen] at (-2,-0) {$A_2$};
			\node[anchor = south, DarkGreen] at (2,2) {\vspace*{0.5 em} $A_3$};
			\node[anchor = north east, red] at (-4,-4) {$3$};
			\node[anchor = south, red] at (1,2) {$2$};
			\node[anchor = west, red] at (3,1) {$1$};
		\end{tikzpicture}
		\begin{tikzpicture}[scale=0.55]
			\node[anchor=east,DarkGreen] at (-1,1) {$A_1$};
			\node[anchor=east, red] at (-1,-1) {$3$};
			\node[anchor=south,DarkGreen] at (1,1) {$A_2$};
			\node[anchor=north,red] at (2,-1) {$2$};
			\draw[thick,DarkGreen] (3,0) -- (4,1);
			\draw[thick,red] (3,0) -- (4,-1);
			\node[anchor=west,red] at (4,-1) {$1$};
			\node[anchor=west,DarkGreen] at (4,1) {$A_3$};
			\draw[thick] (0,0) -- (3,0);
			\draw[thick,DarkGreen] (0,0) -- (-1,1);
			\draw[thick,red] (0,0) -- (-1,-1);
			\draw[thick,DarkGreen] (1,0) -- (1,1);
			\draw[thick,red] (2,0) -- (2,-1);
		\end{tikzpicture} 
		\caption{The construction of the bicolored phylogenetic tree for the matrix from \Cref{ex:cartoon-example-cont}. In the figure on the right, the green leaves are at the top and the red leaves at the bottom.}
		\label{fig:bicolored-trees-construction}
	\end{figure}

\begin{definition}\label{def:bicolored-splits}
The removal of an internal edge splits the tree into two connected components, where each component contains leaves of both colors.
These partitions $(S,([d]\sqcup[n])\setminus S)$ are the \emph{bicolored splits} of the tree.
A bicolored split is \emph{elementary} if one of the two parts has only $2$ elements.
An \emph{internal edge} of a bicolored tree is an edge between two vertices that are not leaves. 
A vertex of a bicolored tree is an \emph{internal vertex} if it is adjacent to at least $2$ internal edges.
A tree is a \emph{caterpillar tree} if every vertex is incident to at most two internal edges. A tree is \emph{maximal} if it is contained in the interior of a maximal cone of $\Phspace$. Equivalently, a tree is maximal if it has $d+n-3$ internal edges.
\end{definition}

We remark that
the rays of a cone in $\Phspace$
correspond to precisely to the bicolored splits of the trees in the cone.
More precisely, the matrices in a ray correspond to a tree with one internal edge, separating the leaves $S$ and $([d]\sqcup[n])\setminus S$. We refer to a bicolored phylogenetic tree as \emph{positive} if it can be obtained by applying \Cref{construction:bicolored-phylogenetic-trees} to a positive matrix $A \in \tropdetplus[2]$.
The geometric interpretation of (positive) matrices of Barvinok rank $2$ implies the following for bicolored trees: 

\begin{corollary}\label{cor:positive-trees-rk-2}
	A bicolored phylogenetic tree is positive if and only if it is a caterpillar tree.
\end{corollary}
\begin{proof}
	Let $P$ be a bicolored phylogenetic tree. 
	Then the tree corresponds to a cone in the triangulation $\Phspace$ of $\tropdet[2]$ in which for every matrix $A$  \Cref{construction:bicolored-phylogenetic-trees} yields $P$. Note that, by construction, the set of bounded edges of $P$ coincides with the tropical convex hull of the columns of $A$. 
	By \Cref{th:positive-rank-2-barvinok}, $A$ is positive if and only if $A$ has Barvinok rank $2$, i.e. the tropical convex hull of the columns of $A$ is the concatenation of ordinary line segments (\Cref{prop:barvinok-rank}). 
	Equivalently, the tropical convex hull (and the resulting phylogenetic tree) does not contain a vertex that is incident to $3$ internal edges or more.
\end{proof}

\begin{remark}\label{rem:maximal-cones-rank-2}
	Every positive cone of $\tropdet[2]$ is contained in a positive maximal cone, as every (non-maximal) caterpillar tree can be obtained by contracting internal edges of a maximal caterpillar tree, and maximal trees correspond to maximal cones (cf. \Cref{excursion:lineality-spaces}).
\end{remark}

We present one result regarding the triangulation, that will be useful for characterizing the positive labels of cones of $\tropdet[2]$ in terms of bipartite graphs.
This proof uses the \emph{balancing condition} of tropical lines. A first introduction to tropical lines was given in \Cref{sec:rank-2-barvinok}, describing a tropical line as a pure, embedded $1$-dimensional polyhedral complex. Let $v$ be a vertex of this polyhedral complex. The balancing condition describes that the slopes of all edges incident to $v$ sum to the zero vector.

	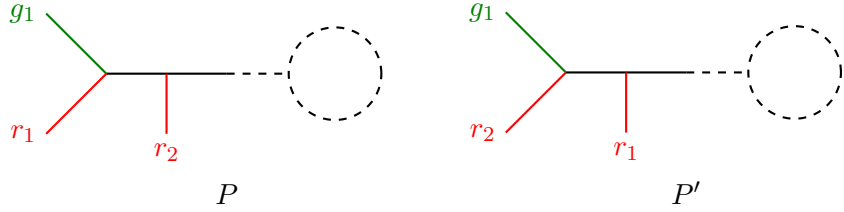
\begin{figure}
	\centering
	\begin{tikzpicture}[scale=0.8]
		\node[anchor=east,DarkGreen] at (-1,1) {$g_1$};
		\node[anchor=east,red] at (-1,-1) {$r_1$};
		\node[anchor=north,red] at (1,-1) { $r_2$};
		\draw[thick] (0,0) -- (2,0);
		\draw[thick, dashed](2,0) -- (3,0);
		\draw[thick,DarkGreen] (0,0) -- (-1,1);
		\draw[thick,red] (0,0) -- (-1,-1);
		\draw[thick,red] (1,0) -- (1,-1);
		\draw[dashed,thick] (3.8,0) circle (2em);
		\node at (2,-2) {$P$};
	\end{tikzpicture}  \hspace{2 em}
	\begin{tikzpicture}[scale=0.8]
		\node[anchor=east,DarkGreen] at (-1,1) {$g_1$};
		\node[anchor=east,red] at (-1,-1) {$r_2$};
		\node[anchor=north,red] at (1,-1) { $r_1$};
		\draw[thick] (0,0) -- (2,0);
		\draw[thick, dashed](2,0) -- (3,0);
		\draw[thick,DarkGreen] (0,0) -- (-1,1);
		\draw[thick,red] (0,0) -- (-1,-1);
		\draw[thick,red] (1,0) -- (1,-1);
		\draw[dashed,thick] (3.8,0) circle (2em);
		\node at (2,-2) {$P'$};
	\end{tikzpicture} 
	\caption{The trees $P$ and $P'$ from \Cref{lem:triangulation-in-bicolored-trees}.}
	\label{fig:lemma-triangulation-in-bicolored-trees-upstairs}
\end{figure}

\begin{lemma}\label{lem:triangulation-in-bicolored-trees}
  Let $C\subseteq \tropdet[2]$ be a maximal cone of the tropical determinantal variety and $\mathcal C_P \subseteq \Phspace$ be a maximal cone of its triangulation, the space of bicolored phylogenetic trees. 
  Then all matrices in the interior of $\mathcal C_P$ correspond to a maximal bicolored phylogenetic tree $P$ with fixed combinatorial type (as depicted in \Cref{fig:lemma-triangulation-in-bicolored-trees-upstairs} on the left).
  Let $\mathcal S$ be the set of splits of $P$.
  If $\mathcal S$ contains the splits
  \begin{equation*}
    (\{r_1,g_1,r_2\}, (R \sqcup G) \setminus \{r_1,g_1,r_2\}) \text{ and }  S_{r_1,g_1} = (\{r_1,g_1\}, (R\sqcup G) \setminus \{r_1,g_1\}) 
  \end{equation*}
  then $C$ also contains the maximal cone $\mathcal C_{P'} \subseteq \Phspace$, where all matrices correspond to trees $P'$ of fixed combinatorial type, and the set of splits of $P'$ is  
  $\mathcal{S}' = \mathcal S \setminus S_{r_1,g_1} \cup S_{r_2,g_1}$, where
  $ S_{r_2,g_1} = (\{r_2,g_1\}, (R\sqcup G) \setminus \{r_2,g_1\}) $  (\Cref{fig:lemma-triangulation-in-bicolored-trees-upstairs} on the right).
  \end{lemma}
The proof of this lemma can be found in \Cref{sec:appendix-proofs-bicolored-trees}. The roles of $R$ and $G$ can be exchanged in this statement. Thus, if $d, n \geq 3$, then every maximal bicolored caterpillar tree has exactly $2$ such pairs of splits. This implies the following result on the number of triangulating cones for positive cones in $\tropdet[2]$.

\begin{corollary}\label{cor: triangulation into four parts}
	Let $C \subseteq \tropdet[2]$ be a maximal cone and $P \in C$ a bicolored caterpillar tree.
	Then the triangulation of $C$ by the space of bicolored phylogenetic trees subdivides $C$ into at least $4$ parts.
\end{corollary}

\subsection{Positivity and bipartite graphs}
We now describe the labels of $\prevar[2]= \tropdet[2]$, which were introduced in \Cref{sec:labels}. In particular, we show that for $r=2$, even though different cones might have the same label, the labels detect positivity, i.e. when a cone lies in $\prevarplus[2]$. Recall from \Cref{th:rank-2-positive-generators} that $\prevarplus[2]=\tropdetplus[2]$, so this criterion also applies to positive cones of the tropical determinantal variety of rank $2$. By \Cref{rem:maximal-cones-rank-2}, it suffices to consider maximal cones.

\begin{lemma}\label{lem: edges in label contain elementary splits}
	Let $C\subseteq \tropdet[2]$ be a maximal cone and $\mathcal{C}_1,\dots\mathcal{C}_m \subseteq \Phspace$ be maximal cones of the space of bicolored phylogenetic trees triangulating $C$. Let $\Ph_k$ be the combinatorially unique (maximal) bicolored phylogenetic tree corresponding to $\mathcal C_k$. Then 
	$$E(\Gc) \supseteq \set{r_i g_j \mid \{i,j\} \text{ is an elementary bicolored split of } \Ph_k \text{ for some } k\in [m]}.$$
\end{lemma}

\begin{proof}
	The rays of the space of bicolored phylogenetic trees are in bijection with bicolored splits $(A,B)$, i.e. trees with one internal edge partitioning the set of leaves into two parts $A$ and $B$, such that both parts contain leaves of both colors (cf. \Cref{sec:appendix-rays-and-cones}). As a matrix in $\RR^{d\times n}$, a ray generator (modulo lineality space) can be given as $\sum_{i,j \in A} E_{ij}, |A|\leq |B|$.
	If $A = \{i,j\}$ is an elementary split, then $E_{ij}$ spans a ray of some $\mathcal C_k$, so $\cone(E_{ij})$ is contained in $C$. Any point except the rays of $C$ are nontrivial nonnegative combinations of rays of $C$. Thus, $\cone(E_{ij})$ is also an extremal ray of $C$. It follows that 
	\[
	\set{r_i g_j \mid E_{ij} \text{ spans a ray of } C} \supseteq
	\set{r_i g_j \mid \{i,j\} \text{ is an elementary split of } \Ph_k \text{ for some } k\in [m]}.
	\]

	Let $C^{IJ} \subseteq \trop(V(f^{IJ}))$ be the inclusion-minimal cone containing $C$. Then $\cone(E_{ij}) \subseteq C^{IJ}$. Since $C^{IJ}$ is a cone in the normal fan of $B_{3}^{IJ}$, all rays of $C^{IJ}$ are of the form $\cone(E_{kl})$. Hence, $\cone(E_{ij})$ cannot be written as a nontrivial nonnegative combination of rays of $C^{IJ}$, and so $\cone(E_{ij})$ is a ray of $C^{IJ}$. By construction, for the cone $C^{IJ}$ holds 
	\[
		E(\Gc[C^{IJ}]) = \{r_i g_j \mid E_{ij} \text{ spans a ray of } C^{IJ}\},
	\]
	and $r_i g_j$ is an edge in $\Gc$ if and only if it is contained in $\Gc[C^{IJ}]$ for all $I,J$ such that $i \in I$ and $j \in J$. Thus,
	$$
	E(\G^c)) \supseteq \set{r_i g_j \mid E_{ij} \text{ spans a ray of } C}.
	$$

\end{proof}

\begin{example}
	Consider the maximal cone $C$ from \Cref{ex:label}. The edges in $\Gc$ are $r_1 g_1, r_2 g_2, r3 g_3$ and $r_3 g_4$. Let $A \in C$ and $P$ be the bicolored phylogenetic tree corresponding to $A$.
	\Cref{lem: edges in label contain elementary splits} implies that these are all the possible candidates for elementary splits of bicolored phylogenetic trees in $C$. Note that the splits corresponding to the edges $r_3 g_3$ and $r_3 g_4$ are not compatible. Thus, there exists no phylogenetic tree having both as elementary splits simultaneously.
\end{example}

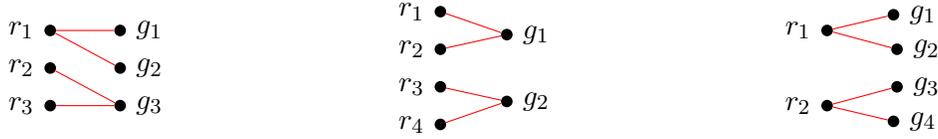
\begin{figure}
	\begin{subfigure}{0.33\textwidth}
		\centering
		\begin{tikzpicture}
			[
			vertex/.style={circle, fill, minimum size=1em, scale=0.4},
			]
			\node[vertex,label=left:{$r_1$}] (r1) at (0,1) {};
			\node[vertex,label=left:{$r_2$}] (r2) at (0,.5) {};
			\node[vertex,label=left:{$r_3$}] (r3) at (0,0) {};
			\node[vertex, right = 2em of r1, label=right:{$g_1$}] (c1) {};
			\node[vertex, right = 2em of r2, label=right:{$g_2$}] (c2) {};
			\node[vertex, right = 2em of r3, label=right:{$g_3$}] (c3) {};
			\draw[color=red] (c1) -- (r1) -- (c2);
			\draw[color=red] (r2) -- (c3) -- (r3);
		\end{tikzpicture}
	\end{subfigure}
	\begin{subfigure}{0.32\textwidth}
		\centering
		\begin{tikzpicture}
			[
			vertex/.style={circle, fill, minimum size=1em, scale=0.4},
			]
			\node[vertex,label=left:{$r_1$}] (r1) at (0,1) {};
			\node[vertex,label=left:{$r_2$}] (r2) at (0,.5) {};
			\node[vertex,label=left:{$r_3$}] (r3) at (0,0) {};
			\node[vertex,label=left:{$r_4$}] (r4) at (0,-.5) {};
			\node[vertex, below right = 0.5em and 2em of r1, label=right:{$g_1$}] (c1) {};
			\node[vertex, above right = 0.5em and 2em of r4, label=right:{$g_2$}] (c4) {};
			\draw[color=red] (r1) -- (c1) -- (r2);
			\draw[color=red] (r3) -- (c4) -- (r4);
		\end{tikzpicture}
	\end{subfigure}
	\begin{subfigure}{0.33\textwidth}
		\centering
		\begin{tikzpicture}
			[
			vertex/.style={circle, fill, minimum size=1em, scale=0.4},
			]
			\node[vertex,label=left:{$r_1$}] (r1) at (0,0.75) {};
			\node[vertex,label=left:{$r_2$}] (r4) at (0,-.25) {};
			\node[vertex, above right = 0.25 em and 2em of r1, label=right:{$g_1$}] (c1) {};
			\node[vertex, right = 2em of r2, label=right:{$g_2$}] (c2) {};
			\node[vertex, right = 2em of r3, label=right:{$g_3$}] (c3) {};
			\node[vertex, below right = 0.25 em and 2em of r4, label=right:{$g_4$}] (c4) {};
			\draw[color=red] (c1) -- (r1) -- (c2);
			\draw[color=red] (c3) -- (r4) -- (c4);
		\end{tikzpicture}
	\end{subfigure}
	\caption{The possible complements $\Gc$ of positive labels (without isolated vertices).}
	\label{fig:positive-labels}
\end{figure}

\begin{theorem}\label{th:positive-labels}
	Let $\G$ be the label of a maximal cone $C\subseteq \tropdet[2]$. Then $C$ is positive if and only if the complement of $\G$ consists of $4$ edges which form two disjoint paths of length $2$ (as shown in \Cref{fig:positive-labels}), and isolated vertices.
\end{theorem}

\begin{proof}
	Let $\Gc$ be a graph consisting of $4$ edges which form two disjoint paths $p^1,p^2$. As first case, consider $p^1 = g_1r_1 g_2$ and $p^2 = r_2 g_3 r_3$. \Cref{lem: edges in label contain elementary splits} implies that the union of elementary splits of all bicolored phylogenetic trees in $C$ is contained in
	\[
	S = \set{ \set{r_1, g_1}, \set{r_1, g_2}, \set{r_2, g_3}, \set{r_3, g_3}}.
	\]
	Any subset of $S$ of size $3$ contains two splits that are not compatible (cf. \Cref{sec:appendix-rays-and-cones}). Hence, every maximal phylogenetic tree $\Ph$ in $C$ has at most $2$ elementary splits. At the same time, every maximal phylogenetic tree contains at least $2$ elementary splits, and this number is $2$ if and only if $\Ph$ is a caterpillar tree. Thus, each maximal phylogenetic tree in $C$ is a  caterpillar tree.
	The argument is similar for $p^1 = r_1g_1 r_2, p^2 = r_3 g_2 r_4$ and for the case $p^1 = g_1 r_1 g_2, p^2 = g_3 r_2 g_4$.
	In all of these cases,  \Cref{cor:positive-trees-rk-2} implies that $C$ is positive.

	Suppose $C$ is a positive cone.
	By \Cref{prop:labels-positive-subgraphs}, every induced subgraph $H$ on $3+3$ vertices contains a subgraph $\Gpos$ consisting of a $4$-cycle and a disjoint edge.
	For the remainder of this proof, we consider the bipartite complement $H^c$.
	By \Cref{prop:labels-matchings-degrees}, every vertex in $H^c$ has degree at most $2$. Note that $H^c$ cannot have a $P_4$ (\Cref{fig:forbidden-P4}) as a subgraph, since its complement (\Cref{fig:forbidden-P4-complement}) does not contain a graph isomorphic to $\Gpos$. Hence, $H^c$  (and therefore $\Gc$) consists of disjoint paths of length $3$ and isolated edges.
	
	\begin{figure} [h]
		\begin{subfigure}{0.49\textwidth}
			\centering
			\begin{tikzpicture}
				[
				vertex/.style={circle, fill, minimum size=1em, scale=0.4},
				]
				\node[vertex,label=left:{$r_1$}] (r1) at (0,1) {};
				\node[vertex,label=left:{$r_2$}] (r2) at (0,.5) {};
				\node[vertex,label=left:{$r_3$}] (r3) at (0,0) {};
				\node[vertex, right = 2em of r1, label=right:{$g_1$}] (c1) {};
				\node[vertex, right = 2em of r2, label=right:{$g_2$}] (c2) {};
				\node[vertex, right = 2em of r3, label=right:{$g_3$}] (c3) {};
				\draw[color=red] (c1) -- (r1) -- (c2) -- (r2);
			\end{tikzpicture}
			\caption{}
			\label{fig:forbidden-P4}
		\end{subfigure}
		\begin{subfigure}{0.49\textwidth}
			\centering
			\begin{tikzpicture}
				[
				vertex/.style={circle, fill, minimum size=1em, scale=0.4},
				]
				\node[vertex,label=left:{$r_1$}] (r1) at (0,1) {};
				\node[vertex,label=left:{$r_2$}] (r2) at (0,.5) {};
				\node[vertex,label=left:{$r_3$}] (r3) at (0,0) {};
				\node[vertex, right = 2em of r1, label=right:{$g_1$}] (c1) {};
				\node[vertex, right = 2em of r2, label=right:{$g_2$}] (c2) {};
				\node[vertex, right = 2em of r3, label=right:{$g_3$}] (c3) {};
				\draw[color=black] (r1) -- (c3) -- (r2) -- (c1) -- (r3) -- (c2);
				\draw[color=black] (c3) -- (r3);
			\end{tikzpicture}
			\caption{}
			\label{fig:forbidden-P4-complement}
		\end{subfigure}
		\caption{}
	\end{figure}
	
	Similarly, $H^c$ cannot contain a perfect matching (3 isolated edges as shown in \Cref{fig:forbidden-C6-complement}), since the complement (\Cref{fig:forbidden-C6}) does not contain a graph isomorphic to $\Gpos$.
	
	\begin{figure} [h]
		\begin{subfigure}{0.49\textwidth}
			\centering
			\begin{tikzpicture}
				[
				vertex/.style={circle, fill, minimum size=1em, scale=0.4},
				]
				\node[vertex,label=left:{$r_1$}] (r1) at (0,1) {};
				\node[vertex,label=left:{$r_2$}] (r2) at (0,.5) {};
				\node[vertex,label=left:{$r_3$}] (r3) at (0,0) {};
				\node[vertex, right = 2em of r1, label=right:{$g_1$}] (c1) {};
				\node[vertex, right = 2em of r2, label=right:{$g_2$}] (c2) {};
				\node[vertex, right = 2em of r3, label=right:{$g_3$}] (c3) {};
				\draw[color=red] (c1) -- (r1); 
				\draw[color=red](c2) -- (r2);
				\draw[color=red](c3) -- (r3);
			\end{tikzpicture}
			\caption{}
			\label{fig:forbidden-C6-complement}
		\end{subfigure}
		\begin{subfigure}{0.49\textwidth}
			\centering
			\begin{tikzpicture}
				[
				vertex/.style={circle, fill, minimum size=1em, scale=0.4},
				]
				\node[vertex,label=left:{$r_1$}] (r1) at (0,1) {};
				\node[vertex,label=left:{$r_2$}] (r2) at (0,.5) {};
				\node[vertex,label=left:{$r_3$}] (r3) at (0,0) {};
				\node[vertex, right = 2em of r1, label=right:{$g_1$}] (c1) {};
				\node[vertex, right = 2em of r2, label=right:{$g_2$}] (c2) {};
				\node[vertex, right = 2em of r3, label=right:{$g_3$}] (c3) {};
				\draw[color=black] (r1) -- (c3) -- (r2) -- (c1) -- (r3) -- (c2) -- (r1);
			\end{tikzpicture}
			\caption{}
			\label{fig:forbidden-C6}
		\end{subfigure}
		\caption{}
	\end{figure}
	
	Thus, $\Gc$ has at most $4$ edges which form two disjoint paths of length $2$.
	However, by \Cref{lem: edges in label contain elementary splits} the edges of $\Gc$ contain the union of all elementary splits of trees in $C$.
	\Cref{cor: triangulation into four parts} implies that there are at least $4$ distinct such trees and hence at least $4$ distinct elementary splits.
	Therefore, the number of edges of $\Gc$ is $4$.
\end{proof}

\subsection{Bicolored trees and the tropical Grassmannian $\Gr$}\label{sec:tropical-grassmannian}

We now investigate the striking resemblance between $\tropdet[2]$, (its triangulation given by) the moduli space $\Phspace$ of bicolored phylogenetic trees on $d+n$ leaves, and the tropical Grassmannian $\Gr$, the moduli space of ordinary phylogenetic trees.
It was shown in  \cite{speyer_tropicalgrassmannian} that a tropical Plücker vector $p\in \Gr$ gives a tree metric by
\begin{equation} \label{eq:Pluecker-vector-tree-metric}
- p_{ij} = \text{ length of the path between leaves } i \text{ and } j,
\end{equation}
where the length of a path is the sum of the lengths of all edges of the path and the length of the leaves $j_1, j_2$.
A point $p\in\TP^{\binom{d+n}{2}}$ is a tropical Plücker vector if and only if it satisfies the $3$-term Plücker relation (or \emph{$4$-point condition)}
\[
\min\{p_{ij}+p_{kl}, p_{il}+p_{jk}, p_{ik}+p_{jl}\} \text{ is attained at least twice}
\]
for all distinct $i,j,k,l \in [d+n].$ 
The maximal cones of the tropical Grassmannian $\Gr$ correspond to phylogenetic trees with $d+n-3$ internal edges and $d+n$ leaves with labels $\{1,\dots,d+n\} = [d+n]$.
The rays of $\Gr$ (modulo lineality space) correspond to partitions of the leaves into two parts
(a more detailed description of the fan structure of $\Gr$ can be found in \Cref{excursion:lineality-spaces}).
It is thus natural to raise the question of the connection between the space of bicolored phylogenetic trees $\Phspace$ and $\Gr$.

\begin{definition}\label{def:trees-d-n-coloring}
	Let $\Ph$ be a phylogenetic tree on $d+n$ leaves.
	A \emph{split} of $\Ph$ is a partition of the leaves into two parts induced by the deletion of an internal edge of $\Ph$.
	A split is \emph{elementary} if one of the two parts has only $2$ elements.
	A \emph{(d,n)-bicoloring} of $\Ph$ is a $2$-coloring of the leaves into $d$ green and $n$ red leaves such that no split of $\Ph$ has a monochromatic part.
\end{definition}

\begin{lemma}\label{lem:number-of-elementary-splits}
	If $d+n \geq 5$, then a maximal phylogenetic tree on $d+n$ leaves has at most $\frac{d+n}{2}$ elementary splits.
\end{lemma}
The proof of this lemma can be found in \Cref{sec:appendix-proofs-tropical-grassmannian}.

\begin{remark}
The number of elementary splits is minimized by caterpillar trees (which have precisely $2$ elementary splits) and the bound in \Cref{lem:number-of-elementary-splits} is attained by snowflake trees, which are trees with a unique internal vertex incident to all internal edges. If $d+n = 4$, then there exists a unique tree, which has precisely one elementary split. If the tree admits a $(d,n)$-coloring, then $d=n=2$, since every the split cannot have a monochromatic part. If $d+n \leq 3$, then there exists no elementary split.
\end{remark}

There is a simple characterization of the existence of a bicoloring in terms of the number of leaves and elementary splits. 

\begin{prop}\label{prop:numer-of-bicolorings}
	Let $\Ph$ be a maximal phylogenetic tree on $d+n$ leaves for some fixed $d,n\in\N$ and $k$ be the number of elementary splits of $\Ph$.
	Then $\Ph$ has a $(d,n)$-bicoloring if and only if $k\leq \min (d,n)$.
        In this case, the number of possible $(d,n)$-bicolorings is $2^k {d+n-2k \choose d-k}$.

	In particular, for any phylogenetic tree $\Ph$ on $m$ leaves, there exist $d,n\in \N$ such that $d+n = m$ and $\Ph$ has a $(d,n)$-bicoloring.
\end{prop}
The proof of this proposition can again be found in \Cref{sec:appendix-proofs-tropical-grassmannian}. We illustrate the existence of bicolorings with an example.

\begin{example}\label{ex:tree-on-five-leaves}
	\begin{figure}
		\centering
			\begin{tikzpicture}[scale=0.75]
			\node[anchor=east] at (-1,1) {$1$};
			\node[anchor=east] at (-1,-1) {$2$};
			\node[anchor=north] at (1.5,-1) {$3$};
			\draw[thick] (3,0) -- (4,1);
			\draw[thick] (3,0) -- (4,-1);
			\node[anchor=west] at (4,-1) {$4$};
			\node[anchor=west] at (4,1) {$5$};
			\draw[thick] (0,0) -- (3,0);
			\draw[thick] (0,0) -- (-1,1);
			\draw[thick] (0,0) -- (-1,-1);
			\draw[thick] (1.5,0) -- (1.5,-1);
		\end{tikzpicture} 
	\caption{The tree from \Cref{ex:tree-on-five-leaves}}
	\label{fig:example-tree-on-five-leaves}
	\end{figure}

	Consider the maximal tree on $5$ leaves as shown in \Cref{fig:example-tree-on-five-leaves}, with elementary splits $(\{1,2\},\{3,4,5\})$ and $(\{1,2,3\}, \{4,5\})$.
	We choose the partition $d+n = 2+3$, i.e., we want to color $2$ leaves in red and $3$ leaves in green.
	In order to obtain a bicoloring, we need to color one of the leaves $\{1,2\}$ in red, and one of the leaves $\{4,5\}$ in red, and color the remaining $3$ leaves in green.
	For the partition $1+4$, then there is no $(1,4)$-bicoloring of this tree, as every such $2$-coloring has at least one monochromatic elementary split.
\end{example}

\subsection{Bicoloring trees and back}
\label{sec:pluecker-bijection}

We now show that the combinatorial idea of ``bicoloring a tree'' can be made precise also on the algebraic level.
In this section we describe a map that, for each $(d,n)$-bicoloring of leaves $[d+n]$, establishes a bijection between the polyhedral fan $\Phspace$ and a suitable subfan of $\Gr$. On the level of trees, the map correspond to ``coloring a tree'' and its inverse to ``forgetting the colors''. A similar result was established in \cite[Lemma 2.10]{markwig_spacetropicallycollinear}. \Cref{th:pluecker-bijection} reveals that this map can be seen as a coordinate projection. \\

We first describe the subfan of $\Gr$, and the map from this subfan to $\Phspace$. Fix $d,n\in \N$ and let $R\sqcup G = [d+n]$ be a $2$-coloring of the leaves in color classes $(R,G)$ such that $|R|=d, |G|=n$.
We say that the coloring $(R,G)$ of the leaves is \emph{admissible} for $P$ if it is is a $(d,n)$-bicoloring (as defined in  \Cref{def:trees-d-n-coloring}), i.e. if for every split of $P$ both parts contain leaves of both colors.
Let now $\Grsub$ be the collection of cones in $\Gr$ corresponding to (uncolored) phylogenetic trees such that the coloring $(R,G)$ is admissible. Consider the coordinate projection
\begin{align}
	\begin{split}
		\pi^{(R,G)}:  \Grsub &\longrightarrow \Phspace \\ 
		(p_{ij})_{ij \in \binom{d+n}{2}} &\longmapsto (p_{ij})_{i \in R, j \in G}.
	\end{split}
\end{align}
We will show that the image of this map is indeed $\Phspace$. For a fixed uncolored tree $P$ and coloring $(R,G)$, let $\bicoltree$ denote the coloring of $P$ with respect to $(R,G)$.
Let $p \in \Grsub$, and denote by $P$ the uncolored, metric phylogenetic tree defined by $p$. Let $\pi^{(R,G)}(p) = A \in \Phspace$.  We say that $\pi^{(R,G)}$ \emph{preserves the combinatorial type of $P$} if the bicolored phylogenetic tree defined by $A$ has the same combinatorial type (sometimes also called \emph{tree topology}) as $P$. A tree is called a \emph{split tree} if it is has exactly one internal edge.

\begin{theorem}\label{th:pluecker-bijection}
	The map $\pi^{(R,S)}$ induces a bijection of fans $\Phspace$ and $\Grsub$, which preserves the combinatorial types of trees.
\end{theorem}

The proof of this theorem is deferred to \Cref{sec:appendix-proofs-bicoloring-trees}. The structure of the proof is as follows. We first describe a ``nice'' representation (modulo lineality spaces of $\Grsub$ and $\Phspace$) of $p$ and $A$ (\Cref{lem:representation-split-trees-bicoloing}). We then establish the result for the lineality spaces of the fans (\Cref{lem:bijection-lineality-space}) and rays (\Cref{th:bijection-rays}). Finally, we deduce the statement of \Cref{th:pluecker-bijection}.

By \Cref{prop:numer-of-bicolorings}, for each cone $ C\subseteq\Gr$ there exists at least one choice of $R\sqcup G =[d+n]$ such that the projection of $p\in \interior(C)$ gives the respective bicoloring of the tree corresponding to $C$.
However, this choice of $R\cup G$ cannot be made globally, as the following example shows.

\begin{example}[A global bicoloring is impossible]\label{example:bicoloring-not-global}
	\begin{figure}
		\centering
		\begin{tikzpicture}[scale = .75]
			\node[anchor=east] at (-1,1) {$2$};
			\node[anchor=east] at (-1,-1) {$1$};
			\node[anchor=north] at (1,-1) { $3$};
			\draw[thick] (0,0) -- (2,0);
			\draw[thick, dashed](2,0) -- (3,0);
			\draw[thick] (0,0) -- (-1,1);
			\draw[thick] (0,0) -- (-1,-1);
			\draw[thick] (1,0) -- (1,-1);
			\draw[dashed,thick] (3.8,0) circle (2em);
			\node at (2,-2) {$P$};
		\end{tikzpicture}  \hspace{5 em}
		\begin{tikzpicture}[scale = .75]
			\node[anchor=east] at (-1,1) {$2$};
			\node[anchor=east] at (-1,-1) {$3$};
			\node[anchor=north] at (1,-1) { $1$};
			\draw[thick] (0,0) -- (2,0);
			\draw[thick, dashed](2,0) -- (3,0);
			\draw[thick] (0,0) -- (-1,1);
			\draw[thick] (0,0) -- (-1,-1);
			\draw[thick] (1,0) -- (1,-1);
			\draw[dashed,thick] (3.8,0) circle (2em);
			\node at (2,-2) {$P'$};
		\end{tikzpicture} 
		\caption{The trees from \Cref{example:bicoloring-not-global}.}
		\label{fig: example coloring is not global}
	\end{figure}
	
	Consider any uncolored phylogenetic tree $\Ph$ with labeled leaves $1,\dots,d+n$ and splits $(\{1,2\},[d+n]\setminus\{1,2\})$ and $(\{1,2,3\},[d+n]\setminus\{1,2,3\})$, as shown in \Cref{fig: example coloring is not global}.
	We can choose a coloring such that $1 \in R$ and $2,3 \in G$.
	The cone in $\Phspace$ containing $\Ph$ is adjacent to the cone containing $\Ph'$, where $\Ph'$ is defined by the same set of splits as $\Ph$, except that the first split $(\{1,2\},[d+n]\setminus\{1,2\})$ is replaced by the split $(\{2,3\},[d+n]\setminus\{2,3\})$.
	However, the coloring with $1 \in R$ and $2,3 \in G$  is not an admissible bicoloring of $\Ph'$, since the split $(\{2,3\},[d+n]\setminus\{2,3\})$ has a part that does not contain leaves of both color classes.
\end{example}

\begin{remark}
The inverse map $(\pi^{(R,G)})^{-1}$ can be interpreted as ``forgetting the colors'' of a bicolored tree.
Given a bicolored phylogenetic tree $\Ph$, we forget the colors relabeling the leaves with $[d+n]$.
The relabeling is not canonical.
For example, we can assign to the red leaves the labels in $[d]$ and assign the labels $d+j, j \in [n]$ to the green leaves.
\end{remark}

We note that the map $\pi^{(R,G)}$ does not preserve positivity.
The cones in the totally positive Grassmannian correspond to trees with clockwise ordered labels \cite{speyer_tropicaltotallypositive}.
There are examples of caterpillar trees with a labeling of the leaves that is not in clockwise order.
On the other hand, there are trees with clockwise ordered labels that are not caterpillar trees. 
It was described in \cite[Example 3.10]{fink_stiefeltropicallinear} that a $1$-dimensional tropical linear space is a Stiefel tropical linear space if and only is it is a caterpillar tree.
This gives us a characterization of the preimage of $\tropdetplus[2]$.

\begin{prop}
	Fix $(R,G)$ such that $R \sqcup G = [d+n]$ and let $\Sigma^{(R,G)} \subseteq \Grsub \subseteq \Gr$ be the subfan consisting of all Stiefel tropical linear spaces for which $(R,G)$ is an admissible $(d,n)$-bicoloring. Then $\pi^{(R,G)}(\Sigma^{(R,G)}) = \tropdetplus[2]$ induces a bijection of fans.
\end{prop}

\begin{proof}
	A $1$-dimensional tropical linear space is a Stiefel tropical linear space if and only is it is a caterpillar tree \cite[Example 3.10]{fink_stiefeltropicallinear}. Thus, \Cref{cor:positive-trees-rk-2} implies that every admissible $(d,n)$-bicoloring of a tree associated to a Stiefel tropical linear space belongs to the tropicalization of a nonnegative matrix and vice versa. By \Cref{th:pluecker-bijection}, $\pi^{(R,G)}$ is a bijection of fans $\Grsub$ and $\tropdet[2]$ that preserves combinatorial types.  Restricting $\pi^{(R,G)}$ therefore induces a bijection of $\Sigma^{(R,G)}$ and $\tropdetplus[2]$.
\end{proof}

\section{Rank 3}\label{sec:rank-3}

In this section, we show the extensions and limitations of the techniques for certifying positivity for $\tropdet[3]$. The main idea is to identify a criterion for a matrix $A$ not to be contained in the positive determinantal prevariety $\prevarplus[3]$ by identifying a $(4\times 4)$-minor, such that $A$ is not contained in the respective positive tropical hypersurface. As $\tropdetplus[3]\subseteq \prevarplus[3]$, we thereby also obtain a condition for $A$ not to be contained in $\tropdetplus[3]$.

As before, we consider the columns of a matrix $A \in \tropdet[3]$ as a point configuration of $n$ points in $\TP^{d-1}$. Due to the rank condition (\Cref{prop:kapranov-rank-points-on-hyperplane}), these points lie on a common tropical plane. 
We view a tropical plane as an embedded pure $2$-dimensional polyhedral complex in $\TP^{d-1}$. It has $d$ unbounded rays $r_1,\dots r_d$, where the slope of $r_i$ is in standard unit direction $e_i$. Furthermore it has bounded edges with edge directions $\sum_{i \in I} e_{i}$ for $I \subseteq [d], |I| \geq 2$. More precisely, a tropical plane is a subcomplex of the polyhedral complex that is dual to a regular matroid subdivision of the hypersimplex $$\Delta(d,3) = [0,1]^d \cap \set{x\in\RR^d \mid \sum_{i=1}^d x_i = 3}.$$ 
where in this subdivision of $\Delta(d,3)$ each maximal cell corresponds to a matroid of rank $3$.
We describe this subcomplex in more detail in \Cref{sec:bicolored-tree-arrangements}.

Let $A \in \tropdet[3]$ and consider the induced point configuration.
The matrix $A\in\tropdet[3]$ (or equivalently the corresponding point configuration) is \emph{generic} with respect to a tropical plane $E$ if every point $A_j$ lies on the interior of a $2$-dimensional face of $E$.
We call a $2$-dimensional face of the plane $E$ 
a \emph{marked face} if it contains a point of the point configuration in its interior.
Recall from \Cref{sec:geom-triangle-crit} that we call the $2$-dimensional faces of a tropical plane in $\TP^3$ the wings of the plane.

\subsection{Starship criterion for positivity}

We establish a condition on the local properties of a tropical plane based on \Cref{th:geometric-triangle-criterion} (\hyperref[th:geometric-triangle-criterion]{Geometric triangle criterion}). The idea is as follows: Let $A \in \tropdet[3]$ and $E$ be a tropical plane containing the columns of $A$. The matrix $A$ is non-positive if there exists a non-positive $(4 \times 4)$-submatrix. The \hyperref[th:geometric-triangle-criterion]{geometric triangle criterion} describes the associated point configuration of $4$ points in $\TP^3$.  We identify a projection of $E$ which selects such a $(4 \times 4)$-submatrix to certify non-positivity. 
The condition to identify the correct submatrix solely depends on the collection of marked faces, i.e. a local structure of the underlying tropical plane. Since a tropical plane is dual to a matroid subdivision of $\Delta(d,3)$, we thus argue via normal cones of faces of matroid polytopes, reducing this problem to a question about flags of flats of the respective matroids.

\begin{lemma}\label{lem:matroid-flats}
	Let $M$ be a matroid of rank $3$ on $n$ elements, and $H_1, H_2, H_3$ be distinct flats of rank $2$. If
	$H_1 \cap H_2 \cap H_3 = F$ is a flat of rank $1$, then $H_3 \not \subseteq H_1 \cup H_2$. 
\end{lemma}

\begin{proof}
  Assume for contradiction that $H_3 \subseteq H_1 \cup H_2$. Let $h \in H_3 \setminus F$.
  Then $h \in H_1$ or $h \in H_2$, and without loss of generality we can assume $h \in H_1$.
  Since $F$ is a flat, and $F \cup \{h\} \subseteq H_1$ we get that $\rk(F \cup \{h\})=2=\rk(H_1)$.
  Hence, $\text{span}(F \cup \{h\}) = H_1$.
  However, at the same time, since $F \cup \{h\} \in H_3$ we have that  $\rk(F \cup \{h\})=2=\rk(H_3)$, and hence $\text{span}(F \cup \{h\}) = H_3$. Thus, $H_1 = \text{span}(F \cup \{h\}) = H_3$, contradicting the assumption that $H_1, H_2, H_3$ are distinct.
\end{proof}

\begin{lemma}\label{lem: rank 3 coordinate projection}
	Let $E \subseteq \TP^{d-1}$ be a realizable tropical plane. Let $F_1, F_2, F_3$ be distinct $2$-dimensional faces of $E$, intersecting in a common unbounded $1$-dimensional face $r$.
	Then there is an index set $I \subseteq [d]$ with $|I| = 4$ such that for the coordinate projection $\pi_I : \TP^{d-1} \to \TP^3$ onto these coordinates the following holds: $\pi_I(E) \subseteq \TP^3$ is a tropical plane with wings $\pi_I(F_1), \pi_I(F_2), \pi_I(F_3)$ intersecting in the common unbounded $1$-dimensional $\pi_I(r)$.
\end{lemma}

\begin{proof}
		Let $E \subseteq \TP^{d-1}$ be a realizable tropical plane. Then $E$ is a subcomplex of a polyhedral complex that is dual to a matroid subdivision of $\Delta(d,3)$, where each maximal matroid polytope corresponds to a matroid of rank $3$. 
		Let $v$ be the vertex of the ray $r = F_1 \cap F_3 \cap F_3$, and $P$ be the matroid polytope dual to $v$. Let $M$ be the corresponding matroid of rank $3$ on ground set $[n]$. 
		Each $2$-dimensional face $F_k$ spans the normal cone of a face of $P$. The $1$-dimensional faces of $F_k$ that are incident to $v$ have slopes $\sum_{i \in H^k_1} e_{i}$ and $\sum_{i \in H^k_2} e_{i}$ respectively. Here, for each $k = 1,2,3$, we have that $\emptyset \subset H^k_1 \subset H^k_2 \subset E$ is a chain of flats of $M$ \cite[Theorem 4.2.6]{maclagan15_introductiontropicalgeometry}. Thus, $H^k_1$ is a flat of rank $1$, and $H^k_2$ is a flat of rank $2$.
		By assumption, $F_1, F_2, F_3$ intersect in an unbounded $1$-dimensional face $r$.
		 Hence, there exists an element $f \in [n]$ such that $H^1_1 = H^2_1 = H^3_1 = \{f\}$. By assumption,  $H^1_2, H^2_2 , H^3_3$ are distinct, and so by \Cref{lem:matroid-flats} we can choose distinct $h_1 \in H^1_{2} \setminus (H^2_{2} \cup H^3_{2}),\ h_2 \in H^2_{2} \setminus (H^1_{2} \cup H^3_{2})$ and $h_3 \in H^3_{2} \setminus (H^1_{2} \cup H^2_{2})$. 
		 Let $I = \{h_1, h_2, h_3, f\}$. Then $\pi_I(r) \subseteq \TP^3$ is the ray spanned by $e_f$, and $\pi_I(F_k) = \cone(e_f, e_{h_k})$. 
                 
		Finally, we show that $\pi_I(E)$ is tropical plane.
		Since $E$ is realizable, $E$ is the tropicalization of a $2$-dimensional plane $\mathcal E$ in $\C\mathbb P^{d-1}$. We can first apply the coordinate projection to obtain a linear space $\pi_I(\mathcal E) \subseteq \C\mathbb P^{3}$ of dimension at most $2$. Note that $\pi_I(E)$ is the tropicalization of $\pi_I(\mathcal E)$, and is thus a tropical linear space of dimension at most $2$. But since $\pi_I(F_k) \subseteq \pi_I (E)$ and $\pi_I(F_k)$ is a $2$-dimensional cone, $ \pi_I (E)$ has dimension $2$ and is a tropical plane.
\end{proof}

If the assumptions of \Cref{lem: rank 3 coordinate projection} hold, i.e. if there is a tropical plane $E$ with $3$ marked $2$-faces $F_1, F_2, F_3$ intersecting in an unbounded ray, then we say that $E$ contains the \emph{starship}\footnote{more precisely, a starship of type ``Lambda-class T-4a shuttle''} formed by the marked faces $F_1, F_2, F_3$. Such a configuration can be seen in \Cref{fig:starship}.

\begin{figure}
	\centering
	\includegraphics[scale=0.9]{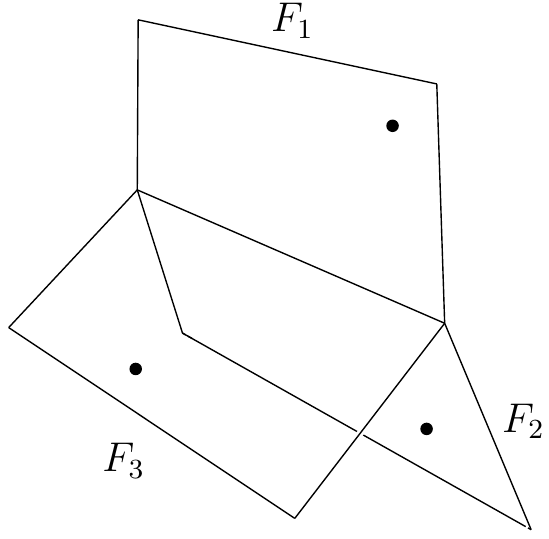}
	\caption{A configuration of three marked $2$-faces forming a starship, as in the assumptions of \Cref{lem: rank 3 coordinate projection}.}
	\label{fig:starship}
\end{figure}

\begin{theorem}[Starship criterion]\label{th:starship-critereon}
	Let $A\in \tropdet[3]$ be a matrix in the interior of a cone $C\subseteq \tropdet[3]$ and let $E$ be a tropical plane containing the points given by the columns of $A$.
	If $E$ has $3$ marked $2$-faces that intersect in an unbounded $1$-dimensional face, then $C$ is not positive.
\end{theorem}

\begin{proof}
  	Let $A_1, A_2, A_3$ be the points lying on $2$-faces $F_1, F_2, F_3$ respectively, and let $j \in [n]\setminus\{1, 2, 3\}$. 
  	By \Cref{lem: rank 3 coordinate projection} there exists a coordinate projection onto coordinates $I = \{i_1, \dots, i_4\}$ such that $\pi_I(F_1),\pi_I(F_2),\pi_I(F_3)$ are $2$-dimensional faces of 
  	the tropical plane $\pi_I(E) \subseteq \TP^3$, which intersect in a common unbounded ray in direction $e_{i_4}$. Note that the projection of the point $\pi_I(A_{k})$ marks the $2$-face $\pi_I(F_k)$ for $k= 1,2,3$, and $\pi_I(A_4) \in \pi_I(E)$.
  	This point configuration of $4$ points in $\TP^3$ is also represented by the $(4 \times 4)$-submatrix of $A$ with rows $I = \{i_1, i_2, i_3, i_4\}$ and columns $J = \{1, 2, 3, j\}$. 
  	Dually, this corresponds to $3$ marked edges of the simplex $\Delta_3$ incident to the triangle that is dual to the ray $e_{i_4}.$
  	By the \hyperref[th:geometric-triangle-criterion]{Geometric triangle criterion} (\Cref{th:geometric-triangle-criterion}) this $(4 \times 4)$-matrix is not positive. Thus, for the minor 
  	\[
  		f^{IJ}= \sum_{\sigma \in S_4} \sgn{\sigma} \prod_{k = 1}^4 x_{i_k \sigma(j_k)}
  	\]
  	holds that $A \not \in \tropcplus(V(f^{IJ}))$. It follows that $A$ is not contained in the positive tropical determinantal prevariety (as defined in \eqref{eq:positive-tropical-determinantal-prevariety} in \Cref{sec:determantal-varieties})
  	\[
  		A \not \in \bigcap_{\substack{f^{IJ} \text{ is a } \\ (4 \times 4)\text{-minor}}} \tropcplus(V(f^{IJ})) = \prevarplus[3]  
  	\]
  	and in particular 
  	\[
  		A \not \in  \bigcap_{f \in I_r} \tropcplus(V(f)) =  \tropdetplus[3] .
  	\]
\end{proof}

We give an example of a matrix $A \in \tropdet[3]$, in which the point configuration in $\TP^{d-1}$ does not contain a starship, but an appropriate coordinate projection does.

\begin{example}[{The converse of the \hyperref[lem: rank 3 coordinate projection]{Starship criterion} does not hold}]\label{ex:starship-criterion}
	Consider the matrix 
	\[
	A = 	\ma
	k & k & 0 & 0 & 0 \\
	0 & k & k & 0 & 1 \\
	0 & 0 & k & k & 0 \\
	0 & 0 & 0 & k & k \\
	k & 0 & 0 & 0 & k
	\trix
	\]
	for any $k  > 1$. This is a point configuration where 
	\[
	A_1 \in W_1 = \cone(e_1, e_5), \ A_2 \in W_2 = \cone(e_1, e_2), \ A_3 \in W_3 = \cone(e_2,e_3),
	\]
	\[
	 A_4 \in W_4 =  \cone(e_3, e_4), \ A_5 \in W_5 = e_2 + \cone(e_4, e_5),
	\]
	which are $2$-dimensional wings of a tropical plane $E \subseteq \TP^{4}$. Hence, this point configuration does not satisfy the assumptions of the \hyperref[th:starship-critereon]{Starship criterion} (\Cref{th:starship-critereon}) -- it does not contain a starship.
	We project the marked wings $W_1, W_2, W_3, W_5$ onto the first $4$ coordinates. Then 
	$
		 \pi(W_1) = \cone(e_1), \ \pi(W_2) = \cone(e_1, e_2), \ \pi(W_3) = \cone(e_2,e_3),  \pi(W_5) = e_2 + \cone(e_4).
	$
	The projections $\pi(W_1), \pi(W_2), \pi(W_3)$ are wings of the (unique) tropical plane $E'$ in $\TP^3$ with apex at the origin, and the projection $\pi(W_5)$ is a ray in the wing $\cone(e_2, e_3)$ of $E'$. Thus, the submatrix
	\[
	\begin{blockarray}{ccccc}
		& A_1 & A_2 & A_3 & A_5  \\
		\begin{block}{c(cccc)}
			1 & k & k & 0 & 0 \\ 
			2 & 0 & k & k & 1 \\ 
			3 & 0 & 0 & k & 0 \\
			4 & 0 & 0 & 0 & k \\
		\end{block}
	\end{blockarray}
	\]
	constitutes a starship (with unbounded ray in direction $e_2$) with respect to $E'$.
	If $i_1 = 1, i_2 =2, i_3 = 3, i_4 = 5$ and
	\[
		f^{IJ}= \sum_{\sigma \in S_4} \sgn{\sigma} \prod_{k = 1}^4 x_{i_k \sigma(k)}
	\]
	then the
	 \hyperref[th:geometric-triangle-criterion]{Geometric triangle criterion} (\Cref{th:geometric-triangle-criterion})
	implies that we have that $A \not \in \tropcplus(V(f^{IJ}))$ and hence $A \not \in \tropdetplus[3]$. If $C$ is a cone of $\tropdet[3]$ containing $A$ in its relative interior, then this implies that $C$ is not positive. Hence, the converse of \Cref{th:starship-critereon} does not hold.
	We continue with this in \Cref{ex:bicolored-tree-arrangements}.
	\end{example}

\subsection{Bicolored tree arrangements}\label{sec:bicolored-tree-arrangements}

Tree arrangements were introduced in \cite{herrmann_howdrawtropical} for studying the Dressian $\Dr{d,3}$. 
It was shown that tree arrangements encode matroid subdivisions of the hypersimplex $\Delta(d,3)$ by looking at the induced subdivision on the boundary. 
In particular, this implies that we can associate a tree arrangement to every tropical plane.
In this section, we extend this idea and introduce \emph{bicolored tree arrangements}, which correspond to a tropical plane with a configuration of points on it.
However, we will see in \Cref{ex:bicolored-tree-arrangements} that this is not a one-to-one correspondence.\\

We first describe the established bijection between tropical planes and (uncolored) tree arrangements, following \cite{herrmann_howdrawtropical}.
As introduced at the beginning of this section, a tropical plane is a subcomplex of the polyhedral complex that is dual to a regular matroid subdivision of the hypersimplex $\Delta(d,3)$.
Inside the affine space $\set{x\in\RR^d \mid \sum_{i=1}^d x_i = 3}$, the facets of $\Delta(d,3)$ are given by $x_i = 0$ and $x_i = 1$.
A tropical plane is the polyhedral complex dual to the subcomplex of a matroid subdivision of $\Delta(d,3)$ consisting of the faces which are not contained in $\set{x_i = 0}$.
Every matroid subdivision of $\Delta(d,3)$ is uniquely determined by the restriction of the subdivision to the $n$ facets of $\Delta(d,3)$ defined by $\set{x_i = 1}$ \cite[Section 4]{herrmann_howdrawtropical}.
We restrict this matroid subdivision to the remaining facets given by $\set{x_i = 1}$.
These facets are isomorphic to a hypersimplex $\Delta(d,2)$, so the restricted subdivisions are dual to a tropical line in $\TP^{d-2}$.
This tropical line has rays in directions $e_1, \dots, e_{i-1}, e_{i+1}, \dots, e_d$.
As these tropical lines are in bijection with (uncolored) phylogenetic trees, this yields a \emph{tree arrangement.}
We extend this idea as follows.

\begin{const}[Bicolored tree arrangements]
Let $A \in \tropdet[3]$ be the matrix giving $n$ points $A_1, \dots, A_n$ on a tropical plane $E$ in $\TP^{d-1}$.
Let $\Sigma$ be the corresponding subdivision of $\Delta(d,3)$.
If $A$ is generic with respect to $E$, then every point $ A_j, j\in [n]$ lies in the interior of a $2$-face $F$, where each $1$-dimensional face of $F$ has slope  $\sum_{i \in I} e_{i}$.
When restricting to a facet $\set{x_{i'} = 1}$ of $\Delta(d,3)$, then the subdivision of this facet is dual to the collection of faces of $E$ that contains the unbounded ray in direction $ e_{i'}$.
We denote the collection of these unbounded faces by $\mathcal F_{i'}$.

Let $J_{i'} \subseteq [n]$ be the set of points lying on a face in $\mathcal F_{i'}$.
To obtain a bicolored tree arrangement, 
project all of these $2$-dimensional faces of $E$ and the points in $J_{i'}$ onto the coordinates $1,\dots,i'-1,i'+1,\dots,d$.
The projection of the $2$-faces form tropical line $L_{i'}$, and the projection of the points are points on $L_{i'}$.
Hence, applying \Cref{construction:bicolored-phylogenetic-trees} induces a bicolored phylogenetic tree $P_{i'}$.
We call the collection of these bicolored trees $P_1, \dots, P_d$ a \emph{bicolored tree arrangement}.
\end{const}

	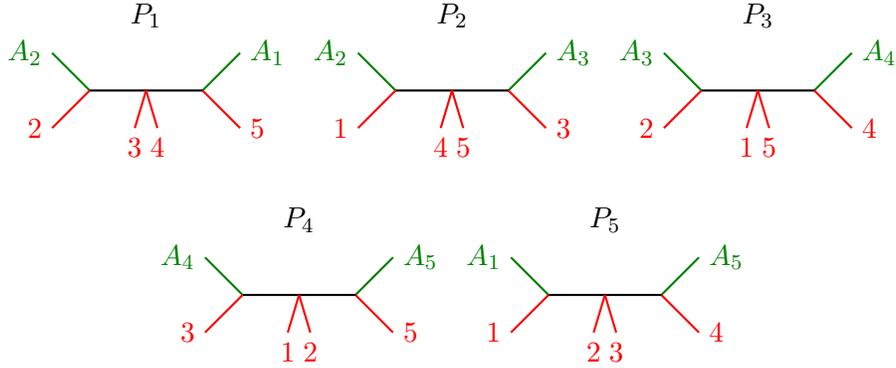
\begin{figure} [t]
	\centering
	\begin{tikzpicture}[scale=0.5]
		\node[anchor=east] at (-1,1) {\color{DarkGreen}$A_2$};
		\node[anchor=east] at (-1,-1) {\color{red} $2$};
		\node[anchor=north] at (1.2,-1) {\color{red}$3$};
		\node[anchor=north] at (1.8,-1) {\color{red}$4$};
		\draw[thick, color=DarkGreen] (3,0) -- (4,1);
		\draw[thick, color=red] (3,0) -- (4,-1);
		\node[anchor=west] at (4,-1) {\color{red}$5$};
		\node[anchor=west] at (4,1) {\color{DarkGreen}$A_1$};
		\draw[thick] (0,0) -- (3,0);
		\draw[thick, color=DarkGreen] (0,0) -- (-1,1);
		\draw[thick, color=red] (0,0) -- (-1,-1);
		\draw[thick, color=red] (1.5,0) -- (1.2,-1);
		\draw[thick, color=red] (1.5,0) -- (1.8,-1);
		\node at (1.5,2) {$P_1$};
	\end{tikzpicture} 
	\begin{tikzpicture}[scale=0.5]
		\node[anchor=east] at (-1,1) {\color{DarkGreen}$A_2$};
		\node[anchor=east] at (-1,-1) {\color{red} $1$};
		\node[anchor=north] at (1.2,-1) {\color{red}$4$};
		\node[anchor=north] at (1.8,-1) {\color{red}$5$};
		\draw[thick, color=DarkGreen] (3,0) -- (4,1);
		\draw[thick, color=red] (3,0) -- (4,-1);
		\node[anchor=west] at (4,-1) {\color{red}$3$};
		\node[anchor=west] at (4,1) {\color{DarkGreen}$A_3$};
		\draw[thick] (0,0) -- (3,0);
		\draw[thick, color=DarkGreen] (0,0) -- (-1,1);
		\draw[thick, color=red] (0,0) -- (-1,-1);
		\draw[thick, color=red] (1.5,0) -- (1.2,-1);
		\draw[thick, color=red] (1.5,0) -- (1.8,-1);
		\node at (1.5,2) {$P_2$};
	\end{tikzpicture} 
	\begin{tikzpicture}[scale=0.5]
		\node[anchor=east] at (-1,1) {\color{DarkGreen}$A_3$};
		\node[anchor=east] at (-1,-1) {\color{red} $2$};
		\node[anchor=north] at (1.2,-1) {\color{red}$1$};
		\node[anchor=north] at (1.8,-1) {\color{red}$5$};
		\draw[thick, color=DarkGreen] (3,0) -- (4,1);
		\draw[thick, color=red] (3,0) -- (4,-1);
		\node[anchor=west] at (4,-1) {\color{red}$4$};
		\node[anchor=west] at (4,1) {\color{DarkGreen}$A_4$};
		\draw[thick] (0,0) -- (3,0);
		\draw[thick, color=DarkGreen] (0,0) -- (-1,1);
		\draw[thick, color=red] (0,0) -- (-1,-1);
		\draw[thick, color=red] (1.5,0) -- (1.2,-1);
		\draw[thick, color=red] (1.5,0) -- (1.8,-1);
		\node at (1.5,2) {$P_3$};
	\end{tikzpicture} \\
\vspace*{1em}
	\begin{tikzpicture}[scale=0.5]
		\node[anchor=east] at (-1,1) {\color{DarkGreen}$A_4$};
		\node[anchor=east] at (-1,-1) {\color{red} $3$};
		\node[anchor=north] at (1.2,-1) {\color{red}$1$};
		\node[anchor=north] at (1.8,-1) {\color{red}$2$};
		\draw[thick, color=DarkGreen] (3,0) -- (4,1);
		\draw[thick, color=red] (3,0) -- (4,-1);
		\node[anchor=west] at (4,-1) {\color{red}$5$};
		\node[anchor=west] at (4,1) {\color{DarkGreen}$A_5$};
		\draw[thick] (0,0) -- (3,0);
		\draw[thick, color=DarkGreen] (0,0) -- (-1,1);
		\draw[thick, color=red] (0,0) -- (-1,-1);
		\draw[thick, color=red] (1.5,0) -- (1.2,-1);
		\draw[thick, color=red] (1.5,0) -- (1.8,-1);
		\node at (1.5,2) {$P_4$};
	\end{tikzpicture} 
	\begin{tikzpicture}[scale=0.5]
		\node[anchor=east] at (-1,1) {\color{DarkGreen}$A_1$};
		\node[anchor=east] at (-1,-1) {\color{red} $1$};
		\node[anchor=north] at (1.2,-1) {\color{red}$2$};
		\node[anchor=north] at (1.8,-1) {\color{red}$3$};
		\draw[thick, color=DarkGreen] (3,0) -- (4,1);
		\draw[thick, color=red] (3,0) -- (4,-1);
		\node[anchor=west] at (4,-1) {\color{red}$4$};
		\node[anchor=west] at (4,1) {\color{DarkGreen}$A_5$};
		\draw[thick] (0,0) -- (3,0);
		\draw[thick, color=DarkGreen] (0,0) -- (-1,1);
		\draw[thick, color=red] (0,0) -- (-1,-1);
		\draw[thick, color=red] (1.5,0) -- (1.2,-1);
		\draw[thick, color=red] (1.5,0) -- (1.8,-1);
		\node at (1.5,2) {$P_5$};
	\end{tikzpicture} 
	\caption{The bicolored tree arrangement of the non-positive matrix in \Cref{ex:starship-criterion}.}
	\label{fig:example-bicolored-tree-arrangement}
\end{figure}

\begin{theorem}\label{th:bicolored-tree-arrangements}
	Let $A \in \tropdet[3][d,n]$ be generic with respect to the tropical plane $E$.
	If $A$ is positive, then every tree in the induced bicolored tree arrangement is a caterpillar tree.
\end{theorem}

\begin{proof}
  Let $P$ be a tree in the bicolored tree arrangement that is not a caterpillar tree.
  We show that $A$ is not positive.
  After relabeling we can assume that $P= P_d$, i.e. $P$ is the tree on the $d$-th facet.
  Since $P$ is not a caterpillar tree, it has an internal vertex that is incident to at least $3$ internal edges.
  Thus,  $P$ corresponds to a tropically collinear point configuration, on which there are points with labels $1, 2, 3$ on a tropical line $L\in \TP^{d-2}$ whose tropical convex hull in $\TP^{d-2}$ contains the $3$ internal edges.
  Consider the $((d-1) \times 3)$-matrix $\overline{A}^{\{1, 2, 3\}}$ consisting of the respective columns $A_{1}, A_{2}, A_{3}$. This matrix $\overline{A}^{\{1, 2, 3\}}$ has Kapranov rank $2$.
  Then by \Cref{cor:positive-trees-rk-2}, it has no positive lift of rank $2$.
  Thus, $\overline{A}^{\{1, 2, 3\}}$ has a $(3 \times 3)$-submatrix $B$ with row indices $i_1, i_2, i_3$, such that (possibly after relabeling ) the column $B_{k}$ lies on the ray of a tropical line in $\TP^2$ with slope $e_k $ for $ k = 1,2,3$. 

  Pick any additional column $j$, and consider the $(4 \times 4)$-submatrix $D$ with row indices $i_1, i_2, i_3, d$ and column indices $1, 2, 3, j$.
  Then the points given by the columns $D_{k}$, $k\in [3]$, and $D_j$ lie on a common tropical plane $E\subseteq \TP^3$. By genericity of $A$ w.r.t. $E$, the points $D_{k}$, $k \in [3]$ lie in the interior of the faces of $E$ that are (up to translation) the cones spanned by the rays $e_{k}$ and $ e_d$, respectively. 
  Dually, this corresponds to $3$ marked edges of the simplex $\Delta_3$ incident to the triangle that is dual to the ray $e_{d}.$
  By the \hyperref[th:geometric-triangle-criterion]{Geometric triangle criterion} (\Cref{th:geometric-triangle-criterion}) this $(4 \times 4)$-matrix is not positive.
  As in the proof of \Cref{th:starship-critereon}, this implies that $A$ is not positive.
\end{proof}

\begin{example}\label{ex:bicolored-tree-arrangements}
	Consider the matrix from \Cref{ex:starship-criterion}.
	This matrix is not positive.
	However, the bicolored trees in this arrangement are all caterpillar trees, as depicted in \Cref{fig:example-bicolored-tree-arrangement}. Thus, the converse of \Cref{th:bicolored-tree-arrangements} does not hold.
\end{example}

\begin{remark}
	The Starship criterion can be obtained as a corollary of \Cref{th:bicolored-tree-arrangements} in the special case that $A$ is generic w.r.t to the tropical plane $E$. Indeed, if $A$ is positive, then every tree in the bicolored tree arrangement is a caterpillar tree. However, a starship with unbounded ray in direction $e_{i'}$ yields a tree $P_{i'}$ that is not a caterpillar tree. Thus, $A$ is not positive.
\end{remark}

In both statements of \Cref{th:starship-critereon} and \Cref{th:bicolored-tree-arrangements}, the converse fails to be true. A main problem lies in that both the tree arrangement and the Starship criterion are only able to capture the geometry of the unbounded faces of the tropical plane. While this information is enough to reconstruct the entire plane \cite{herrmann_howdrawtropical}, this does not suffice to capture information about the point configuration on bounded faces of the tropical plane.

\bibliographystyle{alpha}
\bibliography{references}

\vspace*{\fill}
\subsection*{Affiliations}
\vspace{0.2cm}

\noindent \textsc{Marie-Charlotte Brandenburg} \\
\textsc{ Max Planck Institute for Mathematics in the Sciences \\
	Inselstra{\ss}e 22, 04103 Leipzig, Germany} \\
\url{marie.brandenburg@mis.mpg.de} \\

\noindent \textsc{Georg Loho} \\
\textsc{  University of Twente \\
	Drienerlolaan 5,
	7522 NB Enschede,
	Netherlands} \\
\url{g.loho@utwente.nl} \\

\noindent \textsc{Rainer Sinn} \\
\textsc{ Universität Leipzig \\
	Augustusplatz 10,
	04109 Leipzig, Germany} \\
\url{rainer.sinn@uni-leipzig.de }

\appendix

\section{Appendix: Lineality spaces and fan structures}\label{excursion:lineality-spaces}

\subsection{Lineality spaces}

We devote this section to describe the lineality spaces of the fans $\tropdet, \Phspace$ and $\Gr$, as well as their geometric interpretations. Understanding the lineality spaces is a crucial part of the arguments made in the proofs of  \Cref{lem:representation-split-trees-bicoloing,lem:bijection-lineality-space}, which form the foundation of \Cref{th:pluecker-bijection}.
The lineality space of $\tropdet$ is spanned by the vectors
\begin{equation}\label{eq:lineality space-full}
	\sum_{i=1}^d E_{ij} \text{ for } j\in[n] \text{ and }	\sum_{j=1}^n E_{ij} \text{ for } i\in[d],
\end{equation}
where $E_{ij}$ denotes the standard basis matrix in $\RR^{d\times n}$. Let $A \in \tropdet$ and consider the columns as a point configuration in $\TP^{d-1}$. Let $H$ be an $(r-1)$-dimensional tropical linear space through the columns of $A$. We now describe in which sense the combinatorics of the point configuration stays invariant modulo lineality space.
We write $\mathbf 1$ for the vector $(1,1,\ldots,1)^t\in \T^d$.
First, fix $j \in [n]$. Then
\[
A' := A + \sum_{i=1}^d E_{ij} = (A_1, \dots, A_j + \mathbf 1, \dots, A_n)
\]
so $A'$ is a matrix where for all columns holds $A'_k = A_k$, if $k \neq j$. However, as a point in $\TP^{d-1}$ we have $A_j \cong A_j + \mathbf 1 = A'$. Therefore, as a point configuration inside $\TP^{d-1}$, we consider the  point configurations given by the matrices $A$ and $A'$ to be the same. 
Second, fix  $i \in [d]$. Then
\[
A'' := A + \sum_{j=1}^n E_{ij} =  (A_1+e_i,  \dots, A_n+e_i)
\]
so $A''$ is a matrix where for all columns holds $A''_k = A_k + e_i$. Thus, the point configuration given by $A''$ is a translation by $e_i$ of the point configuration defined by $A$. The points of $A''$ lie on the translated tropical linear space $H + e_i$. Hence, the points in $\tropdet$ modulo lineality space correspond to point configurations in $\TP^{d-1}$ modulo translation. 

The lineality space of the space $\Phspace$ of bicolored phylogenetic trees coincides with the lineality space of $\tropdet[2]$. By construction, two collinear point configurations that are equal up translation induce the same bicolored phylogenetic trees. 

The lineality space of the tropical Grassmannian $\Gr$ is spanned by the vectors
\[
\sum_{\substack{ i \in [d+n] \\ i \neq k }} \tilde{e}_{ik}, \ k \in [d+n]
\]
where the vectors $\tilde{e}_{ij} = \tilde{e}_{ji} \in \TP^{\binom{d+n}{2}-1}$ are the standard basis vectors.
Let $p \in \Gr$ and fix $k \in [d+n]$. Then
\begin{equation}\label{eq:lineality-space-grassmannian}
	p' := (p + 	\sum_{\substack{ i \in [d+n] \\ i \neq k }} \tilde{e}_{ik} )_{st} = \begin{cases}
		p_{st} & \text{if } s,t \neq k \\
		p_{st} + 1 & \text{if } s = k \text{ or } t = k.
	\end{cases}
\end{equation}
Recall that $- p_{ij}$ is the  length of the path between leaves $i$ and $j$ of an (uncolored) phylogenetic tree $P$. That is, the tree $P$ has internal edges of certain lengths, and each leaf $i$ has a length $l_i$. 
The vector $p'$ is the tree metric of the tree $P'$ of the same combinatorial type, where all lengths of internal edges coincide with the lengths of internal edges of $P$. Furthermore, the leaf $i$ of $P'$ has length $l_i'$, and \eqref{eq:lineality-space-grassmannian} implies that $l_i = l'_i$ for $i \neq k$ and $l_k = l'_k -1$.
Therefore, the points in $\Gr$ modulo lineality space correspond to metric phylogenetic trees modulo leaf lengths.

\subsection{Rays and cones}\label{sec:appendix-rays-and-cones}
Modulo the above described lineality space, the fan $\Phspace$ is a simplicial fan in which every cone is generated by $d+n-3$ rays.
Recall from \Cref{sec:bicolored-phylogenetic-trees} that the rays of $\Phspace$ correspond to bicolored splits. As partitions, these are partitions $(S,([d]\sqcup[n])\setminus S)$ of the leaves $[d]\sqcup[n]$ such that both $S$ and $([d]\sqcup[n])\setminus S$ contain leaves of both color classes. 
As trees, these are bicolored phylogenetic trees with a unique internal edge separating the leaves in $S$ and $([d]\sqcup[n])\setminus S$. More precisely, if $r = \{\lambda A \mid \lambda \geq 0 \}$ is a ray of the fan $\Phspace$, then for each $\lambda >0$ the \Cref{construction:bicolored-phylogenetic-trees} produces the same split tree.

Two (bicolored) splits $(S_1,([d]\sqcup[n])\setminus S_1)= (S_1, S_1^c), (S_2,([d]\sqcup[n])\setminus S_2)=(S_2, S_2^c)$ are \emph{compatible} if 
\[
S_1 \subseteq S_2 \ \text{ or } \ S_1 \subseteq S_2^c \ \text{ or } \ S_2 \subseteq S_1^c \ \text{ or } \
S_1^c  \subseteq  S_2^c  .
\]
A set of bicolored splits form a cone $C$ in $\Phspace$ if and only if they are pairwise compatible. If a bicolored tree $P$ is contained in the interior of $C$, then the rays of $C$ are in correspondence to the bounded edges of $P$.

\begin{example}
	Consider the matrices 
	\[
	A = \ma 
	0 & 0 & 0 \\
	0 & 0 & 0 \\
	0 & 1& 1 \trix, \ 
	B = \ma 
	1 & 0 & 0 \\
	0 & 0 & 0 \\
	0 & 0 & 0 \trix 
	\]
	Both $A$ and $B$ span rays in $\Phspace$, as the corresponding phylogenetic trees are the splits $(r_1r_2g_1, r_3g_2 g_3)$ and $(r_1g_1, r_2 r_3 g_2 g_3)$ respectively. Note that the splits are compatible ($r_1g_1 \subseteq r_1r_2g_1$), so they span a $2$-dimensional cone. The sum $A+B$ is a point in the interior of the cone, and corresponds to the tree with exactly these two splits. Hence, the rays $A, B$ correspond to the bounded edges of the tree $A+B$.
\end{example}

The fan $\Gr$ is, modulo its lineality space, a simplicial fan in which every cone is generated by $d+n-3$ rays.
Recall from \Cref{sec:tropical-grassmannian} that the rays of $\Gr$ correspond to (uncolored) splits. As partitions, these are partitions $(S,([d+n])\setminus S)$ of the leaves $[d+n]$. 
As trees, these are phylogenetic trees with a unique internal edge separating the leaves in $S$ and $([d+n])\setminus S$. More precisely, if $r = \{\lambda p \mid \lambda \geq 0 \}$ is a ray of the fan $\Gr$, then for each $\lambda >0$ this is the tree metric for a tree with this topology.

Similarly to their bicolored counterparts,
two (uncolored) splits $(S_1, S_1^c), (S_2, S_2^c)$ are \emph{compatible} if 
\[
S_1 \subseteq S_2 \ \text{ or } \ S_1 \subseteq S_2^c \ \text{ or } \ S_2 \subseteq S_1^c \ \text{ or } \
S_1^c  \subseteq  S_2^c  .
\]
A set of splits forms a cone $C$ in $\Gr$ if and only if they are pairwise compatible. If a tree $P$ is contained in the interior of $C$, then the rays of $C$ are in correspondence to the bounded edges of $P$.

\begin{example}
	Let $d+n = 5$ and consider the tropical Plücker vectors
	
	\[
	\begin{array}{rcccccccccc}
		&12 & 13 & 14 & 15 & 23 & 24 & 25 & 34 & 35 & 45 \\
		p = ( &  0 &  -1& -1 &  -1  & -1 &  -1 &  -1 &  0  &  0 &  0) \\
		q =	( &0   &  0  & -1 &  -1  &   0 &  -1 &  -1 & -1  & -1 &  0 )
	\end{array}
	\]
	Both $p$ and $q$ span rays in $\Gr$, as the corresponding phylogenetic trees are split trees with splits $(12, 345)$ and $(123, 45)$ respectively, and in both cases the unique internal edge has length $1$ and all leaves have length $0$.
	Note that the splits are compatible ($12 \subseteq 123$), so they span a $2$-dimensional cone. The sum $p+q$ is a point in the interior of the cone, and corresponds to the tree with exactly these two splits. Moreover, the length of each internal edge is $1$ and all leaf lengths are $0$.
\end{example}

\section{Additional proofs}
\label{sec:additional-proofs}

\subsection{Positivity in tropical geometry}\label{sec:appendix-proof-positivity}

\begin{proof}[Proof of \Cref{prop:mixed-to-real-lifts}]
	Without loss of generality, we assume $d \leq n$.
	We first show the claim for a matrix of full rank $d$. 
	First set $B_{ij} = \lt(A_{ij})$.
	If the rank of the resulting matrix $B$ is less than $d$, then we can add terms of higher degree with generic real coefficients to the entries of $B$ to obtain a matrix of full rank such that $\lt(B_{ij}) = \lt(A_{ij})$ for all $(i,j) \in [d]\times [n]$ as claimed.

	Let now $\rk A = r < d$.
	We can assume that the first $r$ rows of $A$ are linearly independent and write the bottom rows $A_i$ for $r+1 \leq i \leq d$ as linear combinations of the first $r$, say
	\[
		A_i = \sum_{k=1}^r c_k^i A_k
	\]
	with $c_k^i \in \C$.
	We write $a_{kj} t^{\alpha_{kj}}$ for the leading term of $A_{kj}$ (for $(k,j)\in [r]\times [n]$) and $b_k^i t^{\beta^i_k}$ for the leading term of $c_k^i$ so that $a_{kj}\in \RR$ and $b_k^i\in \CC$.
	If the entry $A_{ij}$ for $i\geq r+1$ is non-zero, then its leading coefficient is therefore of the form $\sum_{k \in S} a_{kj}b_k^i$ for some subset $S\subseteq [r]$.
	We thus know that $\sum_{k \in S} a_{kj}b_k^i \in \RR$. Note that since $a_{ik} \in \RR$, this implies $\sum_{k \in S} b_k^i \in \RR$ and hence 
	\[
		\frac{1}{2} \sum_{k \in S} (b_k^i + \overline{b_k^i} ) = \frac{1}{2} \sum_{k \in S} b_k^i + \overline{ \sum_{k \in S} b_k^i}  =  \sum_{k \in S} b_k^i .
	\]
	
	To get the matrix $B$ as desired, we apply the first part of the proof to the first $r$ rows of $A$ so that we get rows $B_1,\ldots,B_r$ where each entry is a real Puisseux series.
	To fill in the last rows, we replace $c_k^i$ by $(c_k^i + \overline{c_k^i})/2 \in \R$, where $\overline{c}$ for a Puisseux series $c\in \C$ is defined as the series whose coefficients are the complex conjugate of the coefficients of $c$.
	Setting
	\[
		B_i = \sum_{k=1}^r \frac{1}{2} (c_k^i + \overline{c_k^i}) A_k
	\]
	for $i\geq r+1$ gives the leading term of $B_{ij}$ as $ \sum_{k \in S} \frac{1}{2} (b_k^i + \overline{b_k^i})  a_{ik} = \sum_{k \in S} b_k^i a_{ik}$.
\end{proof}

\begin{proof}[Proof of \Cref{prop:kapranov-rank-points-on-hyperplane}]
	Let $A \in \tropdet$. Then $A$ has Kapranov rank $r' \leq r$, and there exists a matrix $\tilde{A} \in \C^{d \times n}$ of rank at most $r'$ such that $A = \val(\tilde{A})$. Hence, the columns of $\tilde{A}$ are $n$ points on a linear space $H \subseteq \C^d$ of dimension $r'$. Equivalently, we can view them as $n$ points on a linear space in $\C\mathbb P^{d-1}$ of dimension $r'-1$. The tropicalization of this linear subspace yields an linear space of the same dimension in $\TP^{d-1}$ containing the columns of $A$.
\end{proof}

\subsection{Bicolored phylogenetic trees}\label{sec:appendix-proofs-bicolored-trees}

	\begin{figure}
	\centering
	\begin{tikzpicture}[scale=0.8]
		\node[anchor=east,DarkGreen] at (-1,1) {$g_1$};
		\node[anchor=east,red] at (-1,-1) {$r_1$};
		\node[anchor=north,red] at (1,-1) { $r_2$};
		\draw[thick] (0,0) -- (2,0);
		\draw[thick, dashed](2,0) -- (3,0);
		\draw[thick,DarkGreen] (0,0) -- (-1,1);
		\draw[thick,red] (0,0) -- (-1,-1);
		\draw[thick,red] (1,0) -- (1,-1);
		\draw[dashed,thick] (3.8,0) circle (2em);
		\node at (2,-2) {$P$};
	\end{tikzpicture}  \hspace{2 em}
	\begin{tikzpicture}[scale=0.8]
		\node[anchor=east,DarkGreen] at (-1,1) {$g_1$};
		\node[anchor=east,red] at (-1,-1) {$r_2$};
		\node[anchor=north,red] at (1,-1) { $r_1$};
		\draw[thick] (0,0) -- (2,0);
		\draw[thick, dashed](2,0) -- (3,0);
		\draw[thick,DarkGreen] (0,0) -- (-1,1);
		\draw[thick,red] (0,0) -- (-1,-1);
		\draw[thick,red] (1,0) -- (1,-1);
		\draw[dashed,thick] (3.8,0) circle (2em);
		\node at (2,-2) {$P'$};
	\end{tikzpicture}\\ \vspace*{1em}
	\begin{tikzpicture}[scale=0.8]
		\node[anchor=east] at (-1,1) {$e_1$};
		\node[anchor=east] at (-1,-1) {$e_2$};
		\filldraw (-.5,0.5) circle (3 pt) {};
		\node[anchor=south west] at (-.5,.5) { $A_1$};
		\node[anchor = north] at (0,0) {$\mathbf 0$};
		\node[anchor = south west] at (0.5,0) {$E$};
		\draw[thick] (0,0) -- (2,0);
		\draw[thick, dashed](2,0) -- (3,0);
		\draw[thick] (0,0) -- (-1,1);
		\draw[thick] (0,0) -- (-1,-1);
		\draw[dashed,thick] (3.8,0) circle (2em);
		\node at (2,-1.5) {$A$};
	\end{tikzpicture}  \hspace{2 em}
	\begin{tikzpicture}[scale=0.8]
		\node[anchor=east] at (-1,1) {$e_1$};
		\node[anchor=east] at (-1,-1) {$e_2$};
		\filldraw (-.5,-0.5) circle (3 pt) {};
		\node[anchor=north west] at (-0.5,-0.5) { $B_1$};
		\node[anchor = south] at (0,0) {$\mathbf 0$};
		\node[anchor = south west] at (0.5,0) {$E'$};
		\draw[thick] (0,0) -- (2,0);
		\draw[thick, dashed](2,0) -- (3,0);
		\draw[thick] (0,0) -- (-1,1);
		\draw[thick] (0,0) -- (-1,-1);
		\draw[dashed,thick] (3.8,0) circle (2em);
		\node at (2,-1.5) {$B$};
	\end{tikzpicture} 
	\caption{The trees $P$ and $P'$ from \Cref{lem:triangulation-in-bicolored-trees} and the corresponding point configurations $A$ and $B$.}
	\label{fig:lemma-triangulation-in-bicolored-trees}
\end{figure}
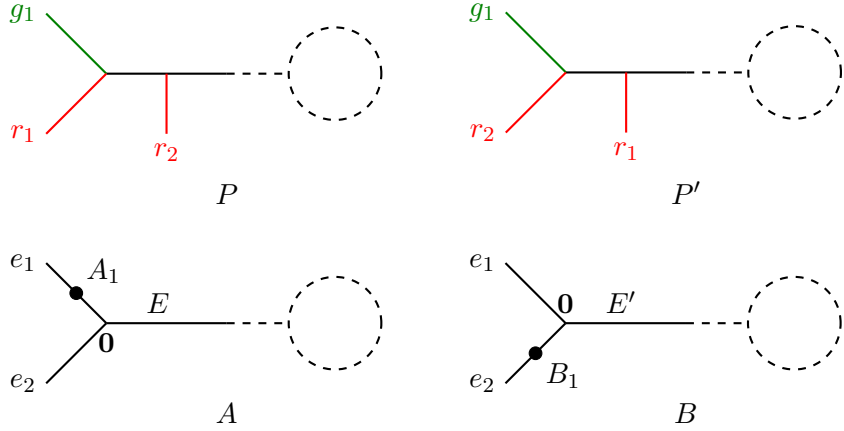

\begin{proof}[Proof of \Cref{lem:triangulation-in-bicolored-trees}]
	We reverse construction \Cref{construction:bicolored-phylogenetic-trees} to obtain a point configuration of points $A_1, \dots, A_n$ forming a matrix $A \in \mathcal C_P$.
	Let $A_1, \dots, A_n$  be such points and $L$ a tropical line containing them.
	The line $L$ is a $1$-dimensional polyhedral complex consisting of bounded edges and $d$ unbounded rays in directions $e_1, \dots, e_d \cong -(e_1 + \dots + e_{d-1})$.
	The splits of $P$ imply that the point $A_1$ lies on the ray of $L$ in direction $e_1$, the rays $e_1$ and $e_2$ meet in a common vertex $v$, and that there is a bounded edge $E$ of $L$ separating the rays in directions $e_1,e_2$ from all other rays (cf. \Cref{fig:lemma-triangulation-in-bicolored-trees}). Similarly, if $B$ is a matrix realizing $P'$, and $L'$ is the line through columns of $B$, then the splits of $P'$ imply that the point $B_1$ lies on the ray direction $e_2$, the rays $e_1$ and $e_2$ meet in a common vertex $v'$, and that there is an internal edge $E'$ of $L'$ separating the rays in directions $e_1,e_2$ from all other rays. 
	Since all remaining splits of $P$ and $P'$ are the same, we can assume that all columns $A_k = B_k$ coincide for $k \geq 2$. 
	Therefore, $L=L'$ and the point configurations only differ in the position of the points $A_1, B_1$ respectively, as depicted in \Cref{fig:lemma-triangulation-in-bicolored-trees} (bottom).
	After translation of the point configurations, we can assume that both for $L$ and $L'$ the vertices $v, v'$ are the origin $\mathbf 0$. The points $A_1$ and $B_1$ can be realized by the first two unit vectors respectively. Since $L$ and $L'$ satisfy the balancing condition, the internal edges $E, E'$ have slope $-e_1 - e_2 \cong e_3 + \dots e_d$, and so all remaining points in the two point configurations can be realized as points $M = (M_1, \dots, M_d)$, where $M_1 = M_2 = 0$ and $M_k > 0$ for all $3 \leq k \leq d$. 
	Therefore, the tree $P$ can be realized as the matrix $A$, and $P'$ is induced by the matrix $B$ as follows:
	\[
	A = \left( \begin{array}{c|c}
		1 & 
		\begin{array}{cccc}
			0 & 0 & \dots & 0
		\end{array} \\ 
		0 & 
		\begin{array}{cccc}
			0 & 0 & \dots & 0
		\end{array} \\ \hline
		\begin{array}{c}
			0 \\
			\vdots \\
			0
		\end{array} &
		M
	\end{array}\right) \ 
	B = \left( \begin{array}{c|c}
		0 & 
		\begin{array}{cccc}
			0 & 0 & \dots & 0
		\end{array} \\ 
		1 & 
		\begin{array}{cccc}
			0 & 0 & \dots & 0
		\end{array} \\ \hline
		\begin{array}{c}
			0 \\
			\vdots \\
			0
		\end{array} &
		M
	\end{array}\right),
	\]
	
	where $M$ is some $(d-2)\times (n-1)$-matrix such that $M_{ij} > 0\ \forall i \in [d-2], j \in [n-1]$.

	We need to show that, for a fixed $(3 \times 3)$-minor $f^{IJ}$, initial forms selected by the weights in $A$ and in $B$, respectively, are the same.
	So let $i_1 < i_2 <i_3, j_1 < j_2 < j_3$ be the indices of three rows and three columns, respectively.
	They define the polynomial $f$ corresponding to the $(3\times 3)$-minor
	\begin{equation*}
		\begin{split}
			f^{IJ} = x_{i_1 j_1} x_{i_2 j_2} x_{i_3 j_3} +  x_{i_1 j_2} x_{i_2 j_3} x_{i_3 j_1}  +  x_{i_1 j_3} x_{i_2 j_1} x_{i_3 j_2} \\
			-  x_{i_1 j_1} x_{i_2 j_3} x_{i_3 j_2}  -  x_{i_1 j_3} x_{i_2 j_2} x_{i_3 j_1}  -  x_{i_1 j_2} x_{i_2 j_1} x_{i_3 j_3}  
		\end{split} 
	\end{equation*}
	and $(3 \times 3)$-submatrices $A^{IJ}, B^{IJ}$ of $A, B$. 
	With this notation, we need to prove the equality $in_{A^{IJ}}(f^{IJ}) = in_{B^{IJ}}(f^{IJ})$,
	where $in_{A^{IJ}}(f^{IJ})$ is the polynomial consisting of those terms with minimal $A^{IJ}$-weight.
	\smallskip
	
	If $1 \not\in \{j_1, j_2, j_3\}$ or $1,2 \not \in \{i_1, i_2, i_3\}$ them $A^{IJ} = B^{IJ}$ and the claim holds.
	For $j_1 = 1$ and $i_1 = 1, i_2 = 2$ the submatrices are
	\[
	A^{IJ} = \ma
	1  & 0 & 0 \\
	0  & 0 & 0 \\
	0  & M_{i_3 j_2} & M_{i_3 j_3}
	\trix \
	\text{ and }
	B^{IJ} = \ma
	0  & 0 & 0 \\
	1  & 0 & 0 \\
	0  & M_{i_3 j_2} & M_{i_3 j_3}
	\trix
	\enspace .
	\]
	The choice $M_{ij} > 0 $ implies that 
	$in_{A^{IJ}}(f) = in_{B^{IJ}}(f) =   x_{i_1 j_2} x_{i_2 j_3} x_{i_3 j_1}   -  x_{i_1 j_3} x_{i_2 j_2} x_{i_3 j_1} $,
	where in each case both terms have weight $0$.
	If $j_1 = 1, i_1 = 1, i_2 \neq 2$ the submatrices are 
	\[
	A^{IJ} = \ma
	1  & 0 & 0 \\
	0  & M_{i_2 j_2} & M_{i_3 j_2} \\
	0  & M_{i_3 j_2} & M_{i_3 j_3}
	\trix \
	\text{ and }
	B^{IJ} = \ma
	0  & 0 & 0 \\
	0  & M_{i_3 j_2} & M_{i_3 j_2} \\
	0  & M_{i_3 j_2} & M_{i_3 j_3}
	\trix
	\enspace .
	\]
	The initial ideal with respect to these matrices can only differ if $M_{i_2 j_2} + M_{i_3 j_3}$ or $M_{i_3 j_2} + M_{i_2 j_3}$ are the terms with minimal $B^{IJ}$-weight.
	However, e.g. $M_{i_2 j_2} $ appears as weight of the term $x_{i_1 j_3} x_{i_2 j_2} x_{i_3 j_1} $, and since $M_{ij} > 0 $, the weight $M_{i_2 j_2} $ is strictly smaller.
	Finally, if $j_1 = 1, i_1 = 2$ then the argument is analogous to the above case $j_1 = 1, i_1 = 1, i_2 \neq 2$.
\end{proof}

\subsection{Bicolored Trees and the tropical Grassmannian $\Gr$}\label{sec:appendix-proofs-tropical-grassmannian}

\begin{proof}[Proof of \Cref{lem:number-of-elementary-splits}]
	By definition, an elementary split arises through the removal of an internal edge which separates $2$ leaves from all others. Within this proof, we call such an edge an ``outer edge''. Let $k$ be the number of outer edges. We double count the number of leaf-edge-incidences for outer edges (i.e. the pairs $(\text{leaf}, \text{outer edge})$ such that the leaf is adjacent to the outer edge). A single outer edge is adjacent to precisely $2$ leaves. Thus, the number of such pairs is $2k$. 
	On the other hand, a single leaf is adjacent to at most one outer edge. The tree has $d+n$ leaves, and hence the number of pairs is at most $d+n$. Combining both counts yields $2k \leq d +n$ or equivalently $k \leq \frac{d+n}{2}$.
\end{proof}

\begin{proof}[Proof of \Cref{prop:numer-of-bicolorings}]
	Let $P$ be a phylogenetic tree on $d+n$ leaves, and suppose $P$ has a $(d,n)$-bicoloring.
	Since every elementary split contains exactly one leaf of each color in the part with $2$ elements, it follows directly that there are at least $k$ leaves of each color, and thus $k \leq d$ and $k \leq n$.\\
	Conversely, let $k$ be the number of elementary splits of $P$ and $k\leq \min(d,n)$.
	For each of these $k$ elementary splits there are exactly $2$ possible colorings of the part with two elements.
	Thus, there are exactly $2^k$ possible colorings of the sets of size $2$ of the elementary splits.
	
	Since $P$ has $d+n$ leaves in total, there are $d+n-2k$ remaining leaves to color: $d-k$ in one color and $n-k$ in the other color.
	Note that for any such coloring, the removal of any internal edge will split the colored tree into $2$ parts, with leaves of both colors in both parts.
	Thus, any such coloring is a bicoloring.
	In total there are hence
	\[
	2^k {d+n - 2k \choose n-k} = 2^k {d+n-2k \choose d-k}
	\]
	$(d,n)$-bicolorings of a maximal tree with $k$ elementary splits. 
	Finally, by \Cref{lem:number-of-elementary-splits} a maximal phylogenetic tree on $m$ leaves has at most $k\leq {m \over 2}$ elementary splits.
	Thus, choosing values for $d,n\in \N$ such that $d+n = m$ and $k\leq \min(d,n)$ is always possible.
	For non-maximal leaves this is a lower bound: Let $(A,B$ be an inclusion-minimal split, i.e. a split such that $|A|\leq |B|$ and there exists no split $(C,D)$ such that $C \subseteq A$ or $D \subseteq A$. Then $A$ contains at least $2$ leaves and we can apply the same argument as above to inclusion-minimal splits instead of elementary splits.
\end{proof}

\subsection{Bicoloring trees and back}\label{sec:appendix-proofs-bicoloring-trees}

We devote this section to the proof of \Cref{th:pluecker-bijection}. To this end, fix a partition $(R,G)$ of $[d+n]$ such that $|R|=d$ and $|G| = n$.

\begin{lemma}\label{lem:representation-split-trees-bicoloing}
	Let $\Ph$ be the (uncolored) phylogenetic split tree with one internal edge of length $\lambda$, and $m=d+n$ leaves, where the removal of the internal edge splits the leaves into two parts $\mathcal S_1$ and $\mathcal S_2$. 
	Let $\bicoltree$ be the bicolored split tree
	$(R_1\sqcup G_1 , R_2 \sqcup G_2)$, i.e. a tree with one internal edge of positive length, red leaves $R=\set{r_1,\dots,r_n}$, and green leaves $G=\set{g_1,\dots,g_d}$, where the removal of the internal edge splits the leaves into two parts $R_1\sqcup G_1$ and $R_2 \sqcup G_2$, and additionally $R_1 \sqcup R_2 = R, G_1 \sqcup G_2 = G$.
	\begin{enumerate}[(i)]
		\item Let $p \in \Gr$ be the tropical Plücker vector corresponding to $P$. Then $p \sim p^\lambda$ modulo the lineality space of $\Gr$, where
		\[
		p^\lambda_{ij} = \begin{cases}
			-\lambda & \text{ if } i\in \mathcal S_1 \text{ and } j \in \mathcal S_2, \\
			0 & \text{ if }  i,j\in \mathcal S_1 \text{ or } i,j \in \mathcal S_2.
		\end{cases}
		\]

		\item Let $A \in \Phspace$ be a matrix corresponding to $\bicoltree$. Then there exists a unique $\lambda >0$ such that $A \sim  A^\lambda$ modulo the lineality space of $\Phspace$, where  $A^\lambda$ is the matrix with columns 
		$$
		A^\lambda_j = \begin{cases}
			- \lambda\sum_{i\in R_2} e_i & \text{ if } j\in G_1 \\
			- \lambda\sum_{i\in R_1} e_i & \text{ if } j\in G_2.
		\end{cases}.
		$$
	\end{enumerate}
\end{lemma}

\begin{proof}
	We begin by showing (i). 
	Let $\ell_1,\dots,\ell_m$ denote the lengths of the leaves $1,\dots,m$ of $P$.
	By \eqref{eq:Pluecker-vector-tree-metric} (in \Cref{sec:tropical-grassmannian}), the Plücker vector $p$ corresponding to $P$ is
	\begin{align*}
		p_{ij} &= \begin{cases}
			-\lambda - \ell_i - \ell_j & \text{ if } i\in \mathcal S_1 \text{ and } j \in \mathcal S_2\\
			-\ell_i - \ell_j & \text{ if }  i,j\in \mathcal S_1 \text{ or } i,j \in \mathcal S_2.
		\end{cases} \\
		&= p_{ij}^\lambda + p_{ij}^\ell,
	\end{align*}
	where we define
	$
	p^\ell_{ij} = - \ell_i - \ell_j \enspace .
	$
	Note that $p^\ell$ lies in the lineality space of $\Gr$, and so $p \sim p^\lambda$.
	
	For (ii), note that the bicolored split tree $\bicoltree$ is also generated by the matrix $A^\lambda $ 
	for any $\lambda > 0$.  Since $A$ and $A^\lambda$ generate the same tree, they are both contained in the interior of the same cone of $\Phspace$. Since $\bicoltree$ is a split tree, modulo lineality space of $\Phspace$, this cone is $1$-dimensional. We choose $\lambda = 1$ and consider the matrix $A^1$ as generator for this ray. Then for every point $A' \in \cone(A^1)$ there exists a unique $\lambda \geq 0$ such that $A' = \lambda A^1 = A^\lambda$. Hence, this also holds for the original matrix $A$, i.e. there exists a unique $\lambda > 0 $ such that $A \sim A^\lambda$ modulo lineality space of $\Phspace$.  
\end{proof}

\begin{lemma}\label{lem:bijection-lineality-space}
	If $p \sim p'$ then $\pi^{(R,G)}(p) \sim \pi^{(R,G)}(p')$.
	$\pi^{(R,G)}$ induces a bijection of lineality spaces of $\tropdet[2]$ and $\Gr$.
\end{lemma}
\begin{proof}
	Note, that since $\Grsub$ is a subfan of $\Gr$, for the respective lineality spaces holds the reverse inclusion.
	We first show that the images of the $d+n$ generators of the lineality space of $\Gr$ (as described in \Cref{excursion:lineality-spaces}) span the lineality space of $\Phspace$.
	 This implies that $\pi^{(R,G)}$ induces a bijection of lineality spaces. 
	Let $p = 
	\sum_{\substack{ i \in [d+n] \\ i \neq k }} \tilde{e}_{ik}, \ k \in [d+n]
	$. Since $R \sqcup G = [d+n]$, either $k \in R$ or $k \in G$. If $k \in R$, then for $i \in R, j \in G$ we have
	\[
	p_{ij} = 
	\begin{cases}
		0 & \text{ if } i=k \\
		1 & \text{ otherwise} \ .
	\end{cases}
	\]
	Thus, $\pi^{(R,G)}(p)$ is the matrix  $\pi^{(R,G)}(p)=\sum_{\substack{i =1 \\ i \neq k}^d}\sum_{j=1}^n E_{ij} $. On the other hand, if $k\in G$, then  $\pi^{(R,G)}(p)=\sum_{i=1}^d \sum_{\substack{j =1 \\ j \neq k}^d} E_{ij} $. 
	Indeed, these vectors span the same vector space as the vectors given in \eqref{eq:lineality space-full} in \Cref{excursion:lineality-spaces}.
	Hence, $\pi^{(R,G)}$ induces a bijection of lineality spaces of $\Gr$ and $\Phspace$. 
	
	Let $p, p' \in \Gr$ arbitrary, such that $p \sim p'$ modulo lineality space $\mathcal L$ of $\Gr$. Then $p \in p' + \mathcal L$, and since $\pi^{(R,G)}$ is linear, it follows that $\pi^{(R,G)}(p + \mathcal L) = \pi^{(R,G)}(p) + \pi^{(R,G)}(\mathcal L)$, where $\pi^{(R,G)}(\mathcal L)$ is the lineality space of $\Phspace$ by the above. Hence, $\pi^{(R,G)}(p) \sim \pi^{(R,G)}(p')$. 
\end{proof}

\begin{prop}\label{th:bijection-rays}
	Let $(R,S)$ be disjoint sets of size $n$ and $d$ respectively.
	The map $\pi^{(R,S)}$ induces a bijection of rays of the fans $\Grsub$  and $\Phspace$ and preserves the combinatorial types of split trees.
\end{prop}

\begin{proof}
	Let $p \in \Grsub \subseteq \Gr$ be a ray generator. Then $p$ corresponds to a split tree $\Ph$ with two parts $\mathcal S_1, \mathcal S_2$ of leaves.
	By \Cref{lem:representation-split-trees-bicoloing} we have that $p \sim p^\lambda$ modulo the lineality space of $\Gr$. By construction, $\pi^{(R,S)}(p^\lambda) = A^\lambda$, so by \Cref{lem:bijection-lineality-space} $p \sim p^\lambda$ implies $\pi^{(R,G)}(p) \sim A^\lambda$. 
	Note that $A^\lambda$ is the matrix representing the $(R,G)$-bicoloring $P^{(R,G)}$ of $P$. Hence, $\pi^{(R,G)}$ sends rays of $\Grsub$ onto rays of $\Phspace$, preserving the combinatorial type of split trees. Note that this implies injectivity: Let $r, r'$ be distinct rays of $\Grsub$ with respective ray generators $p, p'$. Then $p,p'$ correspond to phylogenetic split trees $P,P'$ of distinct combinatorial types. Since the $(R,G)$-bicoloring is admissible for both $P$ and $P'$, the matrices $\pi^{(R,G)}(p)$ and $ \pi^{(R,G)}(p')$ correspond to bicolored trees of distinct combinatorial types. Hence, $\pi^{(R,G)}(r)$ and $ \pi^{(R,G)}(r')$ are distinct rays of $\Phspace$.
	
	For surjectivity, let $r = \cone(A)$ be a ray of $\Phspace$. Then there is a unique $\lambda >0$ such that $A \sim A^\lambda$. Let $A^\ell$ be a matrix in the lineality space of $\Phspace$ such that $A = A^\lambda + A^\ell$. 
	By construction $\pi^{(R,G)}(p^\lambda) = A^\lambda$ and since $\pi^{(R,G)}$ is a bijection on the lineality spaces, there is a unique $p^\ell$ in the lineality space of $\Gr$ such that $A^\ell = \pi^{(R,G)}(p^\ell)$. Again, since $\pi^{(R,G)}$ is linear, it follows for $p := p^\lambda + p^\ell$ that
	\[
	\pi^{(R,G)}(p) = \pi^{(R,G)}(p^\lambda) + \pi^{(R,G)}(p^\ell) = A^\lambda + A^\ell = A. 
	\]
\end{proof}

\begin{proof}[Proof of \Cref{th:pluecker-bijection}]
	\Cref{th:bijection-rays} establishes the statement for rays. 
	Let $p \in \Grsub$ be an arbitrary Plücker vector and
	$P$ be the  corresponding uncolored phylogenetic tree with splits. Modulo lineality space of $\Gr$, we can assume that $P$ has all leaf lengths $0$.
	Let $S_1, \dots, S_k$ be the compatible splits of $P$ with internal edge lengths $\lambda_1, \dots, \lambda_k$,
	and $p^{S_1},\dots,p^{S_k}$ be the Plücker vectors of the corresponding split trees $P_1, \dots, P_k$. Since the $S_i$ are the splits of $P$, it follows that $p^{S_1}+\dots +p^{S_k} = p$ and $p^{S_1},\dots,p^{S_k}$ are ray generators for the rays of the cone $C$ such that $p\in \interior(C)$. Since $C$ is simplicial, this sum is unique. By above, $\pi^{(R,G)}$ is a bijection on the level of rays, so equivalently $\pi^{(R,G)}(p^{S_1}),\dots,\pi^{(R,G)}(p^{S_k})$ is a set of rays in $\Phspace$, corresponding to the bicolored split trees $P_1^{(R,S)}, \dots, P_k^{(R,S)}$. 
	Since the definition of compatibility is independent of the coloring, these bicolored splits are compatible and hence form a cone $C'$ in $\Phspace$. Again, $\Phspace$ is a simplicial fan, and so any point $C'$ has a unique representation as sum of ray generators of $C'$. 
	Consider $A = \pi^{(R,G)}(p^{S_1})+\dots+\pi^{(R,G)}(p^{S_k})$. Then by construction, this matrix corresponds to a tree with bicolored splits $P_1^{(R,S)}, \dots, P_k^{(R,S)}$, i.e. the tree $\bicoltree$. Finally, since $\pi^{(R,G)}$ is a coordinate projection, if $i\in R, j \in G$, then 
	\[
	p_{ij} = \pi^{(R,G)}(p^{S_1} +\dots+ p^{S_k})_{ij} =  \pi^{(R,G)}(p^{S_1})_{ij}+\dots+\pi^{(R,G)}(p^{S_k})_{ij} = A_{ij}
	\]
	and so $\pi^{(R,G)}(p) = A$. Since $\pi^{(R,G)}$ is a bijection of the rays, and all sums are unique, this extends to a bijection to the entire fan.
\end{proof}

\end{document}